\documentclass[11pt]{amsart}
\usepackage{amsmath,mathabx}
\usepackage{amssymb}
\usepackage{amsthm} 
\usepackage{bbm,url}
\usepackage{graphicx}
\usepackage{fullpage}
\usepackage{amssymb}

\usepackage[usenames,dvipsnames,svgnames,table]{xcolor}

\usepackage{float}
\usepackage{caption}
\usepackage{subcaption}
\newcommand{\eps}{\varepsilon}
\renewcommand{\epsilon}{\varepsilon}
\newcommand{\p}{\mathbb{P}}

\newtheorem{thm}{Theorem}
\newtheorem{defin}{Definition}

\newtheorem{lem}{Lemma}
\newtheorem{prop}{Proposition}
\newtheorem{remark}{Remark}
\newtheorem{cor}{Corollary}

\DeclareMathOperator*{\argmin}{arg\,min}

\begin{document}

\title{Lower Memory Oblivious (Tensor) Subspace Embeddings with Fewer Random Bits:  Modewise Methods for Least Squares} 

\author{M. A. Iwen, D. Needell, E. Rebrova, and A. Zare} 
\maketitle
\begin{abstract}
 In this paper new general modewise Johnson-Lindenstrauss (JL) subspace embeddings are proposed that are both considerably faster to generate and easier to store than traditional JL embeddings when working with extremely large vectors and/or tensors.  
 
 Corresponding embedding results are then proven for two different types of low-dimensional (tensor) subspaces.  The first of these new subspace embedding results produces improved space complexity bounds for embeddings of rank-$r$ tensors whose CP decompositions are contained in the span of a fixed (but unknown) set of $r$ rank-one basis tensors.  In the traditional vector setting this first result yields new and very general near-optimal oblivious subspace embedding constructions that require fewer random bits to generate than standard JL embeddings when embedding subspaces of $\mathbbm{C}^N$ spanned by basis vectors with special Kronecker structure.  The second result proven herein provides new fast JL embeddings of arbitrary $r$-dimensional subspaces $\mathcal{S} \subset \mathbbm{C}^N$ which also require fewer random bits (and so are easier to store -- i.e., require less space) than standard fast JL embedding methods in order to achieve small $\epsilon$-distortions.  These new oblivious subspace embedding results work by $(i)$ effectively folding any given vector in $\mathcal{S}$ into a (not necessarily low-rank) tensor, and then $(ii)$ embedding the resulting tensor into $\mathbbm{C}^m$ for $m \leq C r \log^c(N) / \epsilon^2$.  
 
 Applications related to compression and fast compressed least squares solution methods are also considered, including those used for fitting low-rank CP decompositions, and the proposed JL embedding results are shown to work well numerically in both settings.
\end{abstract}

\section{Motivation and Applications}

Due to the recent explosion of massively large-scale data, the need for geometry preserving dimension reduction has become important in a wide array of applications in signal processing (see e.g. \cite{RefWorks:45,RefWorks:373,ahmed2014compressive,zhang2013hyperspectral,gross2010quantum,candes2011robust}) and data science (see e.g. \cite{basri2003lambertian,candes2009exact}). This reduction is possible even on large dimensional objects when the class of such objects possesses some sort of lower dimensional intrinsic structure. For example, in classical compressed sensing \cite{RefWorks:45,RefWorks:373} and its related streaming applications \cite{charikar2002finding,cormode2005s,gilbert2008group,iwen2014compressed}, the signals of interest are \textit{sparse} vectors -- vectors whose entries are mostly zero. In matrix recovery \cite{candes2009exact,recht2010guaranteed}, one often analogously assumes that the underlying matrix is low-rank. Under such models, tools like the Johnson-Lindenstrauss lemma \cite{johnson1984extensions,achlioptas2003database,dasgupta2010sparse,krahmer2011new,larsen2017optimality}
 and the related restricited isometry property \cite{RefWorks:48,baraniuk2008simple} ask that the geometry of the signals be preserved after projection into a lower dimensional space. Typically, such projections are obtained via random linear maps that map into a dimension much smaller than the ambient dimension of the domain; $s$-sparse $n$-dimensional vectors can be projected into a dimension that scales like $s\log(n)$, and $n\times n$ rank-$r$ matrices can be recovered from $O(rn)$ linear measurements \cite{RefWorks:45,RefWorks:373,candes2009exact}. 
 Then, inference tasks or reconstruction can be performed from those lower dimensional representations. 
 
Here, our focus is on dimension reduction of \textit{tensors}, multi-way arrays that appear in an abundance of large-scale applications ranging from video and longitudinal imaging \cite{liu2012tensor,bengua2017efficient} to machine learning \cite{romera2013multilinear,vasilescu2005multilinear} and differential equations \cite{beck2000multiconfiguration,lubich2008quantum}. Although a natural extension beyond matrices, their complicated structure leads to challenges both in defining low dimensional structure as well as dimension reduction projections. In particular, there are many notions of tensor rank, and various techniques exist to compute the corresponding decompositions \cite{kolda2009tensor,zare2018extension}. In this paper, we focus on tensors with low CP-rank, tensors that can be written as a sum of a few rank-1 tensors written as outer products of basis vectors. The CP-rank and CP-decompositions are natural extensions of matrix rank and SVD, and are well motivated by applications such as topic modeling, psychometrics, signal processing, linguistics and many others  \cite{carroll1970analysis,harshman1970foundations,anandkumar2014tensor}.

\subsection{Tensor dimension reduction}
Although there are now some nice results for low-rank tensor dimension reduction, the majority of the work (see e.g. \cite{rauhut2017low, li2017near, tsitsikas2018core, wang2015fast}) gives theoretical guarantees for dimensional reducing projections that act on tensors via their matricizations or vectorizations. Two prominent examples are Tensor Random Projections TRP algorithm (\cite{sun2018tensor}), which is based on the Khatri-Rao product of many smaller random projection maps,
and TensorSketch (\cite{pagh2013compressed, pham2013fast}), which is based on the tensorisation of the CountSketch matrix approach (\cite{charikar2004finding}). However, as mentioned above, these methods do not respect the multi-modal structure of the tensor (one newer version of TensorSketch that actually does that is based on Tucker format \cite{shi2020higher}), and the theoretical guarantees are not as general as it would be desired: TRP was proved only for tensors of order $2$, and the TensorSketch method is mostly applicable to polynomial kernels, that is, a very special case of rank-one tensors when all the component vectors are copies of the same vector (e.g., \cite{pham2013fast, avron2014subspace, ahle2020oblivious}).

There are many motivating application areas that utilize efficient tensor dimension reduction, including the acceleration and improvement of machine learning algorithms (\cite{pham2013fast, li2017near, romera2013multilinear, vasilescu2005multilinear}) and finding tensor decompositions (an extensive review of the tensor dimension reduction techniques for low-rank tensor decompositions is given in \cite{malik2018low}).  Other practical applications range from video and longitudinal imaging \cite{liu2012tensor,bengua2017efficient} to differential equations \cite{beck2000multiconfiguration,lubich2008quantum}.

Here, our goal is to provide theoretical guarantees but for projections that act directly on the tensors themselves without the need for unfolding or vectorization. In particular, this means the projections can be defined \textit{modewise} using the CP-decomposition, and that the low dimensional representations are also tensors, not vectors. This extends the application for such embeddings to those that cannot afford to perform unfoldings or for which it is not natural to do so. In particular, for tensors in $\mathbb{C}^{n^d}$ for large $n$ and $d$, this avoids having to store an often impossibly large $m\times n^d$ linear map. In the next section, we elaborate on our main contributions. 

We also would like to acknowledge several papers that appeared during the latest stages of preparation of this work and its initial review process. These include the theoretical guarantees for the TRP method for low-rank CP and TT tensors (\cite{rakhshan2020tensorized}), new and considerably more efficient algorithm to
compute a linear sketch polynomial kernels (\cite{ahle2020oblivious}), and, finally, two works that are most related to our current paper, \cite{jin2019faster, malik2019guarantees}, showing that KFJLT (special modewise operator based on FFT matrices, see \eqref{equ:FastJLWard}) performs Johnson-Lindenstrauss type transform. The first result is more general than the second as it is applied to any tensors, and the latter one is applicable only to rank-one tensors, and the efficiency of the compression obtained in these two works is incompatible: the first one has better dependence on the dimensions of the original tensor, and the latter one has better dependence on the distortion allowed. The second part of our work uses the result of \cite{jin2019faster} to get an ultimate better result, so further discussion is continued in Section~\ref{sec:main_res2}. A very nice comparison between these recent results (including our work) is also presented in \cite{malik2019guarantees}.

\subsection{Our contributions}
In this paper we analyze modewise tensor embedding strategies for general $d$-mode tensors.
In particular, herein we focus on obliviously embedding an apriori unknown $r$-dimensional subspace of a given tensor product space $\mathbbm{C}^{n_1 \times \dots \times n_d}$ into a similarly low-dimensional vector space $\mathbbm{C}^{\tilde{\mathcal{O}}(r)}$ with high probability.  In contrast to the standard approach of effectively vectorizing the tensor product space and then embedding the resulting transformed subspace using standard JL methods involving a single massive $\tilde{\mathcal{O}}(r) \times \prod^d_{j=1} n_j$ matrix ${\bf M}$ (see, e.g., \cite{li2017near}), the approaches considered herein instead result in the need to generate and store $d+1$ significantly smaller matrices ${\bf A} \in \mathbbm{C}^{\tilde{\mathcal{O}}(r) \times \prod^d_{\ell=1} m_\ell}, {\bf A}_1 \in \mathbbm{C}^{m_1 \times n_1}, \dots, {\bf A}_d \in \mathbbm{C}^{m_d \times n_d}$ which are then combined to form a linear embedding operator $L:  \mathbbm{C}^{n_1 \times \dots \times n_d} \rightarrow \mathbbm{C}^{\tilde{\mathcal{O}}(r)}$ via
\begin{equation}
L(\mathcal{X}) := {\bf A} \left(\mathrm{vect}\left( \mathcal{X} \times_1 {\bf A}_1 \dots \times_d {\bf A}_d \right) \right),
\label{equ:GenNewJLOperator}
\end{equation}
where each $\times_j$ is a $j$-mode product (reviewed below in \S \ref{sec:TensorBasics}), and $\mathrm{vect}: \mathbbm{C}^{m_1 \times \dots \times m_d} \rightarrow \mathbbm{C}^{\prod^d_{\ell=1} m_\ell}$ is a trivial vectorization operator. See Figure~\ref{fig:pic_randproj} for an illustration of how the embedding operator $L$ in \eqref{equ:GenNewJLOperator} works in two stages to first map an example $3$-mode input tensor $\mathcal{X}$ to a smaller $3$-mode tensor $\mathcal{Y}$, and then to a compressed vector ${\bf z} = L(\mathcal{X})$.
\begin{figure}
	\centering
	\includegraphics[width=0.6\linewidth]{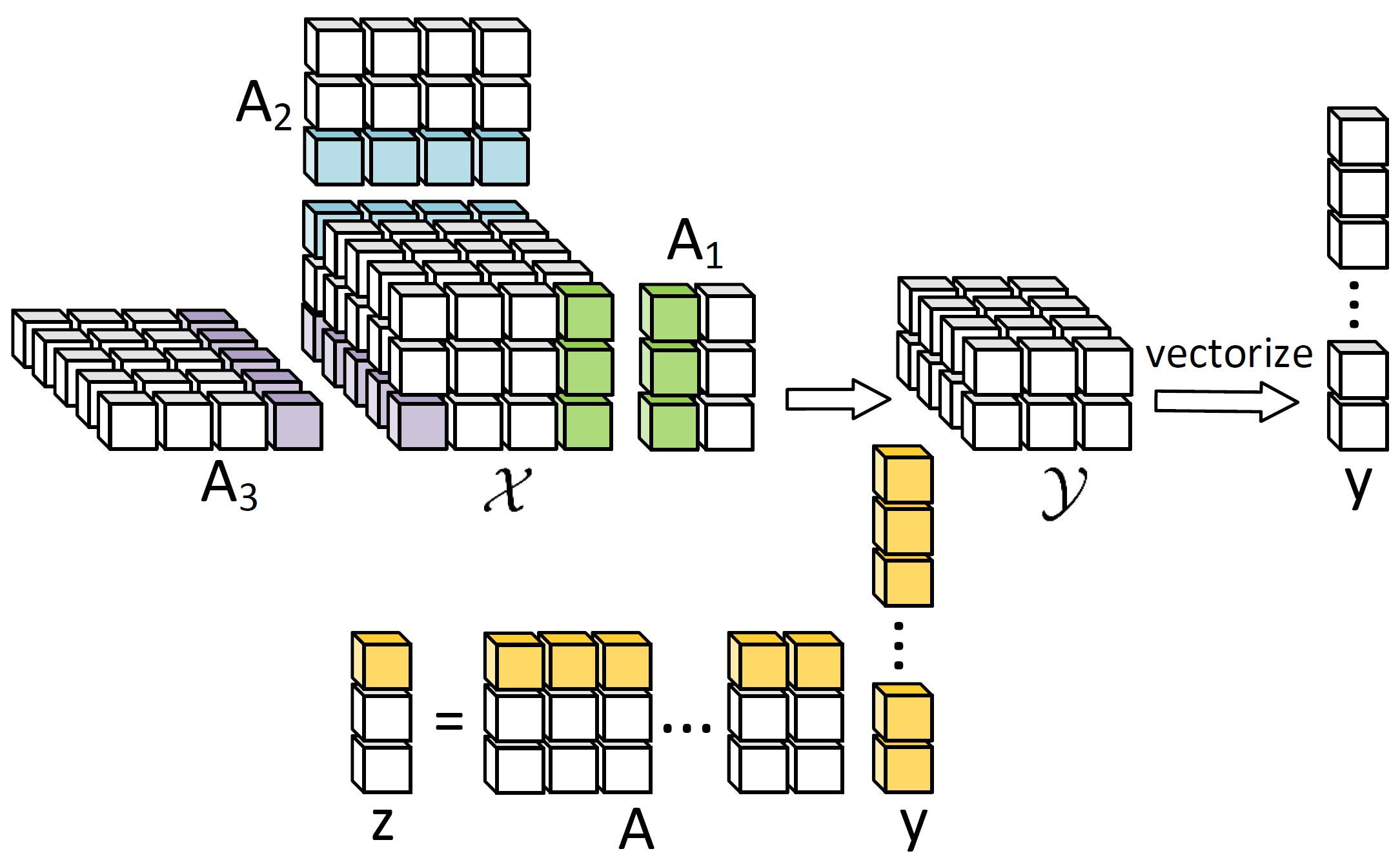}
	\caption{An example of $2$-stage JL embedding applied to a $3$-dimensional tensor $\mathcal{X} \in \mathbb{R}^{3 \times 4 \times 5}$. The output of the $1^{\rm st}$ stage is the projected tensor $\mathcal{Y}=\mathcal{X} \times_1 \mathbf{A}_1 \times_2 \mathbf{A}_2 \times_3 \mathbf{A}_3$, where $\mathbf{A}_j$ are JL matrices for $j \in \{1,2,3\}$, $\mathbf{A}_1 \in \mathbb{R}^{2 \times 3}$, $\mathbf{A}_2 \in \mathbb{R}^{3 \times 4}$, and $\mathbf{A}_3 \in \mathbb{R}^{4 \times 5}$, resulting in $\mathcal{Y} \in \mathbb{R}^{2 \times 3 \times 4}$. Matching colors have been used to show how the rows of $\mathbf{A}_j$ interact with the mode-$j$ fibers of $\mathcal{X}$ (and the intermediate partially compressed tensors) to generate the elements of the mode-$j$ unfolding of the result after each $j$-mode product. Next, the resulting tensor is vectorized (leading to $\mathbf{y} \in \mathbb{R}^{24}$), and a $2^{\rm nd}$-stage JL is then performed to obtain $\mathbf{z}=\mathbf{A}\mathbf{y}$ where $\mathbf{A} \in \mathbb{R}^{3 \times 24}$, and $\mathbf{z} \in \mathbb{R}^{3}$.}
	\label{fig:pic_randproj}
\end{figure}

Let $m' = \tilde{\mathcal{O}}(r)$ be the number of rows one must use for both ${\bf M}$ and ${\bf A}$ above (as we shall see, the number of rows required for both matrices will indeed be essentially equivalent).  The collective sizes of the matrices needed to define $L$ above will be much smaller (and therefore easier to store, transmit, and generate) than ${\bf M}$ whenever $\prod^d_{\ell=1} m_\ell + \sum^d_{\ell=1} n_\ell \left(\frac{m_\ell}{m'}  \right) \ll \prod^d_{j=1} n_j$ holds.  As a result, much of our discussion below will revolve around bounding the dominant $\prod^d_{\ell=1} m_\ell$ term on the left-hand side above, which will also occasionally be referred to as the {\it intermediate embedding dimension} below.  We are now prepared to discuss our two main results.

\subsubsection{General Oblivious Subspace Embedding Results for Low-Rank Tensor Subspaces Satisfying an Incoherence Condition}

The first of our results provides new oblivious subspace embeddings for tensor subspaces spanned by bases of rank-one tensors, as well as establishes related least squares embedding results of value in, e.g., the fitting of a general tensor with an accurate low-rank CPD approximation.  One of its main contributions is the generality with which it allows one to select the matrices ${\bf A}, {\bf A}_1, \dots, {\bf A}_d$ used to construct the JL embedding $L$ in \eqref{equ:GenNewJLOperator}.  In particular, it allows each of these matrices to be drawn independently from any desired nearly-optimal family of JL embeddings (as defined immediately below) that the user likes. 

\begin{defin}[$\epsilon$-JL embedding]
Let $\epsilon \in (0,1)$.  We will call a matrix $\mathbf{A} \in \mathbbm{C}^{m \times n}$ an $\epsilon$-JL embedding of a set $S \subset \mathbbm{C}^n$ into $\mathbbm{C}^m $ if
$$\| \mathbf{A} {\bf x} \|_2^2 = (1 + \epsilon_{\bf x}) \| {\bf x} \|_2^2$$
holds for some $\epsilon_{\bf x} \in (-\epsilon, \epsilon)$ for all ${\bf x} \in S$.  
\end{defin}
\begin{defin}\label{def:EtaOpt}
Fix $\eta \in (0,1/2)$ and let $\left\{ \mathcal{D}_{(m,n)} \right\}_{(m,n) \in \mathbbm{N} \times \mathbbm{N}}$ be a family of probability distributions where each $\mathcal{D}_{(m,n)}$ is a distribution over $m \times n$ matrices.  We will refer to any such family of distributions as being an {\it \bf $\boldsymbol \eta$-optimal family of JL embedding distributions} if there exists an absolute constant $C \in \mathbbm{R}^+$ such that, for any given $\epsilon \in (0,1)$, $m,n \in \mathbbm{N}$ with $m < n$, and nonempty set $\mathcal{S} \subset \mathbbm{C}^{n}$ of cardinality 
$$|S| \leq \eta \exp \left(\frac{\epsilon^2 m}{C} \right),$$ 
a matrix ${\bf A} \sim \mathcal{D}_{(m,n)}$ will be an $\epsilon$-JL embedding of $\mathcal{S}$ into $\mathbbm{C}^m$ with probability at least $1 - \eta$.
\end{defin}
In fact many $\eta$-optimal families of JL embedding distributions exist for any given $\eta \in (0,1/2)$ including, e.g., those associated with random matrices having i.i.d. subgaussian entries (see Lemma 9.35 in \cite{foucart2013book}) as well as those associated with sparse JLT constructions \cite{kane2014sparser}.  The next theorem proves that any desired combination of such matrices can be used to construct a JL embedding $L$ as per \eqref{equ:GenNewJLOperator} for any tensor subspace spanned by a basis of rank-one tensors satisfying an easily testable (and relatively mild\footnote{In fact the coherence condition required by Theorem~\ref{Thm:MAINRES1} will be satisfied by a generic basis of rank-one tensors with high probability (see \S\ref{sec:IncoherentBasesareCommon}).  Similar coherence results to those presented in \S\ref{sec:IncoherentBasesareCommon} have also recently been considered for random tensors in more general parameter regimes by Vershynin \cite{vershynin2019concentration}.}) coherence condition. We utilize the notations set forth below in Section \ref{sec:Notations}.

\begin{thm}
Fix $\epsilon, \eta \in (0,1/2)$ and $d \geq 3$.  Let $\mathcal{X} \in \mathbbm{C}^{n_1 \times \dots \times n_d}$, $n := \max_{j} n_j \geq 4r+1$, and $\mathcal{L}$ be an $r$-dimensional subspace of $\mathbbm{C}^{n_1 \times \dots \times n_d}$ spanned by a basis of rank-one tensors $\mathcal{B} := \left \{ \bigcirc^d_{\ell = 1} {\bf y}^{(\ell)}_k ~\big|~ k \in [r] \right\}$ (where $\bigcirc$ denotes a tensor outer product operator -- see \eqref{equ:tens_prod} below) with modewise coherence satisfying 
$$\mu_\mathcal{B}^{d-1} := \left( \max_{\ell \in [d]} \max_{k, h \in [r], k \neq h} \left| \left \langle {\bf y}^{(\ell)}_k , {\bf y}^{(\ell)}_h \right \rangle \right| \right)^{d-1} < 1/{2r}.
$$  
Then, one can construct a linear operator $L: \mathbbm{C}^{n_1 \times \dots \times n_d} \rightarrow \mathbbm{C}^{m'}$ as per \eqref{equ:GenNewJLOperator} with $m' \leq C' r \cdot \epsilon^{-2} \cdot \ln \left( \frac{47}{\epsilon \sqrt[r]{\eta}} \right)$ for an absolute constant $C' \in \mathbbm{R}^+$ so that with probability at least $1 - \eta$ 
\begin{equation}
\left| \left\| L\left( \mathcal{X} - \mathcal{Y} \right) \right\|^2_2 - \left\|  \mathcal{X} - \mathcal{Y} \right\|^2 \right| \leq \epsilon \left\|  \mathcal{X} - \mathcal{Y} \right\|^2
\label{equ:leastsquaresMainres1}
\end{equation}
will hold for all $\mathcal{Y} \in \mathcal{L}$.  

If $\mathcal{X} \notin \mathcal{L}$ the intermediate embedding dimension can be bounded above by 
\begin{equation}
\prod^d_{\ell=1} m_\ell ~\leq~ C^d \cdot r^d d^{3d}/\epsilon^{2d} \cdot \ln^d \left( n / \sqrt[d]{\eta}\right)
\label{equ:IntermedDimBoundXnotinL}
\end{equation}
for an absolute constant $C \in \mathbbm{R}^+$.  If, however, $\mathcal{X} \in \mathcal{L}$ then \eqref{equ:leastsquaresMainres1} holds for all $r < 1/2\mu_\mathcal{B}^{d-1}$ and \begin{equation}
\prod^d_{\ell=1} m_\ell ~\leq~ \tilde{C}^d \cdot r^2 \left(d / \eps\right)^{2d} \cdot \ln^d \left( 2r^2 d / {\eta}\right)
\label{equ:IntermedDimBoundXinL}
\end{equation}
can be achieved, where $\tilde C \in \mathbbm{R}^+$ is another absolute constant.
\label{Thm:MAINRES1}
\end{thm}

\textit{Proof Sketch for Theorem~\ref{Thm:MAINRES1}.}  
This is largely a restatement of Theorem~\ref{thm:GenmodewiseConstruction}.  When defining $L: \mathbbm{C}^{n_1 \times \dots \times n_d} \rightarrow \mathbbm{C}^{m'}$ as per \eqref{equ:GenNewJLOperator} following Theorem~\ref{thm:GenmodewiseConstruction} one should draw ${\bf A}_j \in \mathbbm{C}^{m_j \times n_j}$ with 
$m_j \geq C_j \cdot r d^3/\epsilon^2 \cdot \ln \left( n / \sqrt[d]{\eta}\right)$ {from} an $(\eta/4d)$-optimal family of JL embedding distributions for each $j \in [d]$, where each $C_j \in \mathbbm{R}^+$ is an absolute constant.  Furthermore, ${\bf A} \in \mathbbm{C}^{m' \times \prod^d_{\ell=1} m_\ell}$ should be drawn from an $(\eta/2)$-optimal family of JL embedding distributions with $m'$ as above.  The probability bound together with \eqref{equ:IntermedDimBoundXnotinL} both then follow.  The achievable intermediate embedding dimension when $\mathcal{X} \in \mathcal{L}$ in \eqref{equ:IntermedDimBoundXinL} can be obtained from Theorem~\ref{cor:MainOblEmb} since the bound $\prod^d_{\ell=1} m_\ell ~\leq~ \prod^d_{\ell=1} \tilde C_\ell \cdot r^{2/d} d^2/\eps^2 \cdot \ln \left( 2r^2 d / {\eta}\right)$ can then be utilized in that case.\hfill$\Box$\\

One can vectorize the tensors and tensor spaces considered in Theorem~\ref{Thm:MAINRES1} using variants of \eqref{equ:KronModenFlat} to achieve subspace embedding results for subspaces spanned by basis vectors with special Kronecker structure as considered in, e.g., two other recent papers that appeared during the preparation of this manuscript \cite{jin2019faster,malik2019guarantees}.  The most recent of these papers also produces bounds on what amounts to the intermediate embedding dimension of a JL subspace embedding along the lines of \eqref{equ:GenNewJLOperator} when $\mathcal{X} \in \mathcal{L}$ (see Theorem 4.1 in \cite{malik2019guarantees}).  Comparing \eqref{equ:IntermedDimBoundXinL} to that result we can see that Theorem~\ref{Thm:MAINRES1} has reduced the $r$ dependence of the effective intermediate embedding dimension achieved therein from $r^{d+1}$ to $r^2$ (now independent of $d$) for a much more general set of modewise embeddings.  However, Theorem~\ref{Thm:MAINRES1} incurs a worse dependence on epsilon and needs the stated coherence assumption concerning $\mu_\mathcal{B}$ to hold.  As a result, Theorem~\ref{Thm:MAINRES1} provides a large new class of modewise subspace embeddings that will also have fewer rows than those in \cite{malik2019guarantees} for a large range of ranks $r$ provided that $\mu_\mathcal{B}$ is sufficiently small and $\epsilon$ is sufficiently large.

Note further that the form of \eqref{equ:leastsquaresMainres1} also makes Theorem~\ref{Thm:MAINRES1} useful for solving least squares problems of the type encountered while computing approximate CP decompositions for an arbitrary tensor $\mathcal{X} \notin \mathcal{L}$ using alternating least squares methods (see, e.g., \S\ref{sec:CPleastSquaresstuff} for a related discussion as well as \cite{battaglino2018practical} where modewise strategies were shown to work well for solving such problems in practice).  Comparing Theorem~\ref{Thm:MAINRES1} to the recent least squares result of the same kind proven in \cite{jin2019faster} (see Corollary 2.4) we can see that Theorem~\ref{Thm:MAINRES1} has reduced the $r$ dependence of the effective intermediate embedding dimension achievable in \cite{jin2019faster} from $r^{2d}$ therein to $r^d$ in \eqref{equ:IntermedDimBoundXnotinL} for a much more general set of modewise embeddings.  In exchange, Theorem~\ref{Thm:MAINRES1} again incurs a worse dependence on epsilon and needs the stated coherence assumption concerning $\mu_\mathcal{B}$ to hold, however.  As a result, Theorem~\ref{Thm:MAINRES1} guarantees that a larger class of modewise JL embeddings can be used in least squares applications, and that they will also have smaller intermediate embedding dimensions as long as $\mu_\mathcal{B}$ is sufficiently small and $\epsilon$ sufficiently large.

\subsubsection{Fast Oblivious Subspace Embedding Results for Arbitrary Tensor Subspaces}\label{sec:main_res2}

Our second main result builds on Theorem 2.1 of Jin, Kolda, and Ward in \cite{jin2019faster} to provide improved fast subspace embedding results for arbitrary tensor subspaces (i.e., for low dimensional tensor subspaces whose basis tensors have arbitrary rank and coherence).  Let $N := \prod^d_{j=1} n_j$.  By combining elements of the proof of Theorem~\ref{Thm:MAINRES1} with the optimal $\epsilon$-dependence of Theorem 2.1 in \cite{jin2019faster} we are able to provide a fast modewise oblivious subspace embedding $L$ as per \eqref{equ:GenNewJLOperator} that will simultaneously satisfy \eqref{equ:leastsquaresMainres1} for all $\mathcal{Y}$ in an entirely arbitrary $r$-dimensional tensor subspace $\mathcal{L}$ with probability at least $1 - \eta$ while also achieving an intermediate embedding dimension bounded above by 
\begin{equation}
C^d \left(\frac{r}{\epsilon}\right)^2 \cdot \log^{2d-1} \left( \frac{N}{\eta} \right) \cdot \log^4 \left( \frac{\log \left(\frac{N}{\eta} \right)}{\epsilon} \right) \cdot \log N.
\label{equ:BestIntMedbeDim}
\end{equation}
Above $C > 0$ is an absolute constant.  Note that neither $r$ nor $\epsilon$ in \eqref{equ:BestIntMedbeDim} are raised to a power of $d$ which marks a tremendous improvement over all of the previously discussed results when $d$ is large.  See Theorem~\ref{thm:RachelNewSubspaceEmbed} for details.

As alluded to above, the results herein can also be used to create new JL subspace embeddings in the traditional vector space setting. Our next and final main result does this explicitly for arbitrary vector subspaces by restating a variant of Theorem~\ref{thm:RachelNewSubspaceEmbed} in that context.  We expect that this result may be of independent interest outside of the tensor setting.

\begin{thm}\label{thm2}
Fix $\epsilon, \eta \in (0,1/2)$ and $d \geq 2$.  Let ${\bf x} \in \mathbbm{C}^{N}$ such that $\sqrt[d]{N} \in \mathbbm{N}$ and $N \geq 4C' / \eta > 1$ for an absolute constant $C' > 0$, and let $\mathcal{L}$ be an $r$-dimensional subspace of $\mathbbm{C}^{N}$ for $\max \left(2r^2 - r,4r \right) \leq N$.  Then, one can construct a random matrix ${\bf A} \in \mathbbm{C}^{m \times N}$ with 
\begin{equation}
m ~\leq~ C \left[ r \cdot \epsilon^{-2} \cdot \log \left( \frac{47}{\epsilon \sqrt[r]{\eta}} \right) \cdot \log^4 \left( \frac{r \log \left(\frac{47}{\epsilon \sqrt[r]{\eta}} \right)}{\epsilon} \right) \cdot \log N \right],
\label{equ:FinalEmbeddingDimRachel} 
\end{equation}
for an absolute constant $C > 0$ such that with probability at least $1 - \eta$ it will be the case that
\begin{equation*}
\left| \left\| {\bf A}\left(  {\bf x} - {\bf y} \right) \right\|^2_2 - \left\|  {\bf x} - {\bf y} \right\|^2_2 \right| \leq \epsilon \left\| {\bf x} - {\bf y} \right\|^2_2
\end{equation*}
holds for all ${\bf y} \in \mathcal{L}$.  Furthermore, the matrix ${\bf A}$ requires only
\begin{equation}
\mathcal{O}\left( C_1^d \left(\frac{r}{\epsilon}\right)^2 \cdot \log^{2d-1} \left( \frac{N}{\eta} \right) \cdot \log^4 \left( \frac{\log \left(\frac{N}{\eta} \right)}{\epsilon} \right) \cdot \log^2 N + d \sqrt[d]{N} \right)
\label{equ:FinalNumRandomBitsRachel}
\end{equation}
random bits and memory for storage for an absolute constant $C_1 > 0$, and can be multiplied against any vector in just $\mathcal{O}\left(N \log N \right)$-time.  

Note that choosing ${\bf x} = {\bf 0}$ produces an oblivious subspace embedding result for $\mathcal{L}$, and that choosing $\mathcal{L}$ to be the column space of a rank-$r$ matrix produces a result useful for least squares sketching.
\end{thm}

\textit{Proof Sketch for Theorem~\ref{thm2}.}  
This follows from Theorem~\ref{thm:RachelNewSubspaceEmbed} after identifying $\mathbbm{C}^{N}$ with $\mathbbm{C}^{\sqrt[d]{N} \times \dots \times \sqrt[d]{N}}$, i.e., after effectively reshaping any given vectors ${\bf x}, {\bf y}$ under consideration into $d$-mode tensors $\mathcal{X, Y}$.  Note further that if $\sqrt[d]{N} \notin \mathbbm{N}$ then one can implicitly pad the vectors of interest with zeros until it is (i.e., effectively trivially embedding $\mathbbm{C}^{N}$ into $\displaystyle \mathbbm{C}^{ \left \lceil \sqrt[d]{N} ~\right \rceil^d}$) before preceding.\hfill$\Box$\\

\subsection{Organization} The remainder of the paper is organized as follows. Section \ref{sec:background} provides background and notation for tensors (Subsections~\ref{sec:Notations} and \ref{sec:TensorBasics}), as well as for Johnson-Lindenstrauss embeddings (Subsection~\ref{sec:JLbasics}). 

We start Section \ref{sec:main} with the definitions of the rank of the tensor (and low-rank tensor subspaces) and the maximal modewise coherence of tensor subspace bases. Then we work our way to Theorem~\ref{cor:MainOblEmb} which constructs
oblivious tensor subspace embeddings via modewise tensor products (for any fixed subspace having low enough modewise coherence). This result is very general in terms of JL-embedding maps one can use as building blocks in each mode. Finally, in Subsection~\ref{sec:IncoherentBasesareCommon} we discuss the assumption of modewise incoherence and provide several natural examples of incoherent tensor subspaces.

In Section~\ref{sec:CPleastSquaresstuff}, we describe the fitting problem for approximately low-rank tensors, and explain how modewise dimension reduction (as presented in Section \ref{sec:main}) reduces the complexity of the problem. Then we build the machinery to show that the solution of the reduced problem will be a good solution for the original problem (in Theorem~\ref{thm:GenmodewiseConstruction}). We conclude Section~\ref{sec:CPleastSquaresstuff} by introducing a two-step embedding procedure that allows one to further reduce the final embedding dimension (this is our second main embedding result, Theorem~\ref{thm:RachelNewSubspaceEmbed}). This improved procedure relies on a specific form of JL-embedding of each mode. Both embedding results can be applied to the fitting problem.

In Section \ref{sec:exp} we present some simple experiments confirming our theoretical guarantees, and then we conclude in Section \ref{sec:conc}.

\section{Notation, Tensor Basics, \& Linear Johnson-Lindenstrauss Embeddings}\label{sec:background} 

\label{sec:Notations} 
Tensors, matrices, vectors and scalars are denoted in different typeface for clarity below. Calligraphic boldface capital letters are always used for tensors, boldface capital letters for matrices, boldface lower-case letters for vectors, and regular (lower-case or capital) letters for scalars.  The matrix ${\bf I}$ will always represent the identity matrix.  The set of the the first $d$ natural numbers will be denoted by $[d] := \{1, \dots, d \}$ for all $d \in \mathbb{N}$. 

Throughout the paper, $\otimes$ denotes the Kronecker product of vectors or matrices, 
and $\bigcirc$ denotes the tensor outer product of vectors or tensors.\footnote{As \eqref{equ:tens_prod} suggests, it can be applied to tensors with arbitrary number of modes.} The symbol $\circ$ on the other hand represents the composition of functions (see e.g. Section \ref{sec:CPleastSquaresstuff}).
Numbers in parentheses used as a subscript or superscript on a tensor either denote {\it unfoldings} (introduced in Section \ref{sec:TensorBasics}) when appearing in a subscript, or else an element in a sequence when appearing in a superscript.
The notation $\otimes_{\ell \neq j} {\bf v}^{(\ell)}$ for a given set of vectors $\left \{ {\bf v}^{(\ell)} \right \}^d_{\ell = 1}$ will always denote the vector ${\bf v}^{(d)} \otimes \dots {\bf v}^{(j+1)}\otimes {\bf v}^{(j-1)} \dots \otimes {\bf v}^{(1)}$.  Additional tensor definitions and operations are reviewed below (see, e.g., \cite{kolda2009tensor,de2008tensor,vannieuwenhoven2012new,zare2018extension} for additional details and discussion).

\subsection{Tensor Basics}
\label{sec:TensorBasics}

The set of all $d$-mode tensors $\mathcal{X} \in \mathbbm{C}^{n_1 \times n_2 \times \dots \times n_d}$ forms a vector space over the complex numbers when equipped with component-wise addition and scalar multiplication.  The inner product of $\mathcal{X}, \mathcal{Y} \in \mathbbm{C}^{n_1\times n_2\times ... \times n_{d}}$ will be given by

\begin{equation}
\left\langle \mathcal{X},\mathcal{Y} \right\rangle := \sum_{i_1=1}^{n_1} \sum_{i_2=1}^{n_2} ...\sum_{i_{d}=1}^{n_{d}} \mathcal{X}_{i_1,i_2,...,i_{d}} ~ \overline{\mathcal{Y}_{i_1,i_2,...,i_{d}}}.
\end{equation}  
This inner product then gives rise to the standard Euclidean norm 

\begin{equation}
\| \mathcal{X} \| := \sqrt{\left\langle \mathcal{X},\mathcal{X} \right\rangle} = \sqrt{\sum_{i_1=1}^{n_1} \sum_{i_2=1}^{n_2} ...\sum_{i_{d}=1}^{n_{d}} \left| \mathcal{X}_{i_1,i_2,...,i_{d}} \right|^2}.
\end{equation}  
If $\left\langle \mathcal{X},\mathcal{Y} \right\rangle = 0$ we say that $\mathcal{X}$ and $\mathcal{Y}$ are {\it orthogonal}.  If $\mathcal{X}$ and $\mathcal{Y}$ are orthogonal and also have unit norm (i.e., have $\| \mathcal{X} \| = \| \mathcal{Y} \| = 1$) we say that they are {\it orthonormal}.

\bigskip

\noindent {\bf Tensor outer products:} The \textit {tensor outer product} of two tensors $\mathcal{X} \in \mathbbm{C}^{n_1 \times n_2 \times \dots \times n_d}$ and $\mathcal{Y} \in \mathbbm{C}^{n'_1 \times n'_2 \times \dots \times n'_{d'}}$, $\mathcal{X} \bigcirc \mathcal{Y} \in \mathbbm{C}^{n_1 \times n_2 \times \dots \times n_d \times n'_1 \times n'_2 \times \dots \times n'_{d'}}$, is a $(d + d')$-mode tensor whose entries are given by
\begin{equation}
\left( \mathcal{X} \bigcirc \mathcal{Y} \right)_{i_1, \dots, i_d, i'_1, \dots, i'_{d'}} = \mathcal{X}_{i_1,...,i_{d}} \mathcal{Y}_{i'_1, \dots, i'_{d'}}.
\label{equ:tens_prod}
\end{equation}
Note that when $\mathcal{X}$ and $\mathcal{Y}$ are both vectors, the tensor outer product will reduce to the standard outer product. Some additional standard properties are also listed below in Lemma~\ref{lem:ProdProps}.

\bigskip

\noindent {\bf Fibers:} Let tensor $\mathcal{X} \in \mathbbm{C}^{n_1 \times \dots \times n_{j-1} \times n_j \times n_{j+1} \times \dots \times n_d}$. The vectors in $\mathbbm{C}^{n_j}$ obtained by fixing all of the indices of $\mathcal{X}$ except for the one that corresponds to its $j^{\rm th}$ mode are called its {\it mode-$j$ fibers}.  Note that any such $\mathcal{X}$ will have $\prod_{\ell \neq j} n_\ell$ mode-$j$ fibers denoted by $\mathcal{X}_{i_1,\dots,i_{j-1},:,i_{j+1},\dots,i_{d}} \in \mathbbm{C}^{n_j}$.\\

\noindent {\bf Tensor matricization (unfolding):}  The process of reordering the elements of the tensor into a matrix is known as matricization or unfolding. The mode-$j$ matricization of a tensor $\mathcal{X} \in \mathbbm{C}^{n_1\times n_2\times ... \times n_d}$ is denoted as $\mathbf{X}_{(j)} \in \mathbbm{C}^{n_j \times \prod_{m \neq j} n_m}$ and is obtained by arranging $\mathcal{X}$'s mode-$j$ fibers to be the columns of the resulting matrix.\\

\noindent  {\bf $j$-mode products:}  The \textit{$j$-mode product} of a $d$-mode tensor $\mathcal{X} \in \mathbbm{C}^{n_1 \times \dots \times n_{j-1} \times n_j \times n_{j+1} \times \dots \times n_d}$ with a matrix $\mathbf{U} \in \mathbbm{C}^{m_j \times n_j}$ is another $d$-mode tensor $\mathcal{X} \times_j \mathbf{U} \in \mathbbm{C}^{n_1 \times \dots \times n_{j-1} \times m_j \times n_{j+1} \times \dots \times n_d}$.  Its entries are given by 
\begin{equation}
(\mathcal{X} \times_j \mathbf{U})_{i_1, \dots, i_{j-1}, \ell, i_{j+1}, \dots, i_d} = \sum_{i_{j}=1}^{n_j} \mathcal{X}_{i_{1},\dots,i_j,\dots,i_{d}} \mathbf{U}_{\ell,i_{j}}
\label{equ:DefModejProduct}
\end{equation}
for all $(i_1, \dots, i_{j-1}, \ell, i_{j+1}, \dots, i_d) \in [n_1] \times \dots \times [n_{j-1}] \times [m_j] \times [n_{j+1}] \times \dots \times [n_d]$. 
Looking at the mode-$j$ unfoldings of $\mathcal{X} \times_j \mathbf{U}$ and $\mathcal{X}$ one can easily see that their mode-$j$ matricization can be computed as a regular matrix product
\begin{equation}\label{unfolding_via_product}(\mathcal{X} \times_j \mathbf{U})_{(j)} = \mathbf{U} \mathbf{X}_{(j)}
\end{equation}
for all $j \in [d]$. The following simple lemma formally lists several important properties of tensor outer products and mode-wise products.  The proof of Lemma~\ref{lem:ProdProps} can be found in Appendix~\ref{AppPreliminaryProofs}.

\begin{lem}
Let $\mathcal{A}, \mathcal{B} \in \mathbbm{C}^{n_1 \times n_2 \times \dots \times n_d}$, $\mathcal{C}, \mathcal{D} \in \mathbb{C}^{n'_1 \times n'_2 \times \dots \times n'_{d'}}$, $\alpha, \beta \in \mathbbm{C}$, and $\mathbf{U}_\ell, \mathbf{V}_\ell \in \mathbbm{C}^{m_\ell \times n_\ell}$ for all $\ell \in [d]$.  The following four properties hold:
\begin{enumerate}
\item[(i)] $\left(\alpha \mathcal{A} + \beta \mathcal{B} \right) \bigcirc \mathcal{C} ~=~ \alpha \mathcal{A} \bigcirc \mathcal{C} +  \beta \mathcal{B} \bigcirc \mathcal{C} ~=~ \mathcal{A} \bigcirc \alpha \mathcal{C} + \mathcal{B} \bigcirc \beta \mathcal{C}$.\\

\item[(ii)] $\left\langle \mathcal{A} \bigcirc \mathcal{C}, \mathcal{B} \bigcirc \mathcal{D} \right\rangle ~=~ \left\langle \mathcal{A}, \mathcal{B}\right\rangle \left\langle \mathcal{C}, \mathcal{D}\right\rangle.$ \\

\item[(iii)] $\left( \alpha \mathcal{A} + \beta \mathcal{B} \right) \times_j \mathbf{U}_j = \alpha \left( \mathcal{A} \times_j \mathbf{U}_j \right) + \beta \left( \mathcal{B} \times_j \mathbf{U}_j \right)$.\\

\item[(iv)] $\mathcal{A} \times_j  \left( \alpha \mathbf{U}_j + \beta \mathbf{V}_j \right) = \alpha \left( \mathcal{A} \times_j \mathbf{U}_j \right) + \beta \left( \mathcal{A} \times_j \mathbf{V}_j \right)$. \\

\item[(v)] If $j \neq \ell$ then $\mathcal{A} \times_j \mathbf{U}_j \times_\ell \mathbf{V}_\ell = \left( \mathcal{A} \times_j \mathbf{U}_j \right) \times_\ell \mathbf{V}_\ell = \left( \mathcal{A} \times_\ell \mathbf{V}_\ell \right) \times_j \mathbf{U}_j = \mathcal{A} \times_\ell \mathbf{V}_\ell \times_j \mathbf{U}_j$ .\\

\item[(vi)] If $W \in \mathbb{C}^{p \times m_j}$ then $\mathcal{A} \times_j \mathbf{U}_j \times_j \mathbf{W} = \left( \mathcal{A} \times_j \mathbf{U}_j \right) \times_j \mathbf{W} = \mathcal{A} \times_j \left( \mathbf{W}\mathbf{U}_j \right) = \mathcal{A} \times_j \mathbf{W}\mathbf{U}_j $.
\end{enumerate}
\label{lem:ProdProps}
\end{lem}

\bigskip

A generalization of the observation~\eqref{unfolding_via_product} is available: unfolding the tensor
\begin{equation}
\mathcal{Y}= \mathcal{X}\times_1 {\bf U}^{(1)}\times_2 {\bf U}^{(2)}...\times_d {\bf U}^{(d)} =: {\displaystyle \mathcal{X} \bigtimes_{j=1}^d {\bf U}^{(j)}},
\label{equ:jmodeprod}
\end{equation} 
along the $j^{\rm th}$ mode is equivalent to 
\begin{equation}
{\bf Y}_{(j)} = {\bf U}^{(j)} {\bf X}_{(j)}\left({\bf U}^{(d)} \otimes \dots {\bf U}^{(j+1)}\otimes {\bf U}^{(j-1)} \dots \otimes {\bf U}^{(1)} \right)^\top,
\label{equ:KronModenFlat}
\end{equation}
where $\otimes$ is the matrix Kronecker product (see \cite{kolda2009tensor}).  In particular, \eqref{equ:KronModenFlat} implies that the matricization $\left( \mathcal{X} \times_j {\bf U}^{(j)} \right)_{(j)} = {\bf U}^{(j)} {\bf X}_{(j)}$.\footnote{Simply set ${\bf U}^{(m)} = {\bf I}$ (the identity) for all $m \neq n$ in \eqref{equ:KronModenFlat}.  This fact also easily follows directly from the definition of the $j$-mode product.}
On a related note, one can also express the relation between the vectorized forms of $\mathcal{X}$ and $\mathcal{Y}$ in \eqref{equ:jmodeprod} as
\begin{equation}
\text{vect}\left(\mathcal{Y}\right)=\left({\bf U}^{(d)} \otimes \dots \otimes {\bf U}^{(1)} \right) \text{vect}\left(\mathcal{X}\right),
\label{equ:jmodeprod_vect}
\end{equation}
\noindent where vect$\left( \cdot \right)$ is the vectorization operator.

Finally, it is also worth noting that trivial inner product preserving isomorphisms exist between a tensor space $\mathbbm{C}^{n_1\times n_2\times ... \times n_d}$ and any of its matricized versions (i.e., mode-$j$ matricization can be viewed as an isomorphism between the original tensor vector space $\mathbbm{C}^{n_1\times n_2\times ... \times n_d}$ and its mode-$j$ matricized target vector space $\displaystyle \mathbbm{C}^{n_j \times \prod_{m \neq j} n_m}$).  In particular, the process of matricizing tensors is linear.  If, for example, $\mathcal{X}, \mathcal{Y} \in \mathbbm{C}^{n_1\times n_2\times ... \times n_d}$ then one can see that the mode-$j$ matricization of $\mathcal{X} + \mathcal{Y} \in \mathbbm{C}^{n_1\times n_2\times ... \times n_d}$ is $\left( \mathcal{X} + \mathcal{Y} \right)_{(j)} = {\bf X}_{(j)} + {\bf Y}_{(j)}$ for all modes $j \in [d]$.  

\subsection{Linear Johnson-Lindenstrauss Embeddings}\label{sec:JLbasics}

Many linear $\epsilon$-JL embedding matrices exist \cite{johnson1984extensions,achlioptas2003database,dasgupta2010sparse,krahmer2011new,larsen2017optimality} with the best achievable $m = \mathcal{O}(\log\left(|S| \right) / \epsilon^2)$ for arbitrary $S$ (see \cite{larsen2017optimality} for results concerning the optimality of this embedding dimension).  Of course, one can define JL embedding on tensors in a similar way, namely, as linear maps approximately preserving tensor norm:

\begin{defin}[Tensor $\epsilon$-JL embedding]
A linear operator $L: \mathbbm{C}^{n_1\times n_2\times ... \times n_d} \rightarrow \mathbbm{C}^{m_1 \times \dots \times m_{d'}}$ is an $\epsilon$-JL embedding of a set $S \subset \mathbbm{C}^{n_1\times n_2\times ... \times n_d}$ into $\mathbbm{C}^{m_1 \times \dots \times m_{d'}}$ if
$$\left\| L \left({\mathcal X} \right) \right\|^2 = (1 + \epsilon_{\mathcal{X}}) \left\| {\mathcal X} \right\|^2$$
holds for some $\epsilon_{\mathcal{X}} \in (-\epsilon, \epsilon)$ for all $ {\mathcal X} \in S$.  
\end{defin}

It is easy to check that JL embeddings can preserve pairwise inner products.  The proof of the following Lemma~\ref{lem:InnProdJL} can be found in Appendix~\ref{AppPreliminaryProofs}.

\begin{lem}
Let ${\bf x},{\bf y} \in \mathbbm{C}^n$ and suppose that $\mathbf{A} \in \mathbbm{C}^{m \times n}$ is an $\epsilon$-JL embedding of the vectors
$$\left\{ {\bf x} - {\bf y}, {\bf x} + {\bf y}, {\bf x} - \mathbbm{i}{\bf y}, {\bf x} + \mathbbm{i}{\bf y}  \right\} \subset \mathbbm{C}^{n}$$
into $\mathbbm{C}^{m}$.  Then,
$$\left| \left\langle \mathbf{A}{\bf x},~ \mathbf{A}{\bf y} \right\rangle - \left\langle {\bf x},~ {\bf y}\right\rangle \right| ~\leq~ 2 \epsilon \left(  \| {\bf x} \|^2_2 + \| {\bf y} \|^2_2 \right) ~\leq~ 4 \epsilon \cdot \max \left\{ \| {\bf x} \|_2^2, \| {\bf y} \|_2^2 \right\}.$$

Moreover, if ${\mathcal X},{\mathcal Y} \in \mathbbm{C}^{n_1\times n_2\times ... \times n_d}$ and suppose that $L$ is an $\epsilon$-JL embedding of the tensors
$$\left\{ {\mathcal X} - {\mathcal Y}, {\mathcal X} + {\mathcal Y}, {\mathcal X} - \mathbbm{i}{\mathcal Y}, {\mathcal X}+ \mathbbm{i}{\mathcal Y} \right\} \subset \mathbbm{C}^{n_1\times n_2\times ... \times n_d}$$
into $\mathbbm{C}^{m_1 \times \dots \times m_{d'}}$.  Then,
$$\left| \left\langle L\left({\mathcal X}\right),~ L\left({\mathcal Y}\right) \right\rangle - \left\langle {\mathcal X},~ {\mathcal Y}\right\rangle \right| ~\leq~ 2 \epsilon \left(  \| {\mathcal X} \|^2 + \| {\mathcal Y} \|^2 \right) ~\leq~ 4 \epsilon \cdot \max \left\{ \| {\mathcal X} \|^2, \| {\mathcal Y} \|^2 \right\}.$$
\label{lem:InnProdJL}
\end{lem}


\bigskip

In the case where a more general set $S$ is embedded using JL embeddings, for example, a low-rank subspace of tensors, in order to pass to a smaller finite set, a discretization technique can be used. Due to linearity, it actually suffices to discretize the unit ball of the space in question.
In the next lemma we present a simple subspace embedding result based on a standard covering argument (see, e.g., \cite{baraniuk2008simple,foucart2013book}).  We include its proof in Appendix~\ref{AppPreliminaryProofs} for the sake of completeness. 

\begin{lem}
Fix $\epsilon \in (0,1)$.  Let $\mathcal{L}$ be an $r$-dimensional subspace of $\mathbbm{C}^{n}$, and let $\mathcal{C} \subset \mathcal{L}$ be an $(\epsilon/16)$-net of the $(r-1)$-dimensional Euclidean unit sphere $\mathcal{S}_{\ell^2} \subset \mathcal{L}$.  Then, if ${\bf A} \in \mathbbm{C}^{m \times n}$ is an $(\epsilon/2)$-JL embedding of $\mathcal{C}$ it will also satisfy
\begin{equation}
(1-\epsilon) \| {\bf x} \|^2_2 \leq \| {\bf A} {\bf x}\|^2_2 \leq (1 + \epsilon) \| {\bf x} \|_2^2
\label{equ:subspacepres}
\end{equation}
for all ${\bf x} \in \mathcal{L}$. Furthermore, we note that there exists an  $(\epsilon/16)$-net such that $\left| \mathcal{C} \right| \leq \left(\frac{47}{\epsilon} \right)^r$.\label{lem:simplenetsubspace}
\end{lem}

\begin{remark}
We will see later in the text that the cardinality $(47/\eps)^r$ (exponential in $r$) can be too big to produce tensor JL embeddings with optimal embedding dimensions.  In this case one can use a much coarser ``discretization"  to improve the dependence on $r$ based on, e.g., the next lemma. We point out that there is indeed a tradeoff when using an approach like Lemma \ref{lem:Subspaceembed2.0}; instead of controlling the norms of all vectors in a subspace by embedding a cover of the unit ball  as in Lemma \ref{lem:simplenetsubspace}, in Lemma \ref{lem:Subspaceembed2.0} we instead control the norms of all vectors in a subspace by embedding an orthonormal basis that approximately preserves their orthogonality.  The tradeoff is that one needs to preserve the angles between the orthonormal basis vectors quite accurately in order to ensure that all of the vectors in their span also have their norms preserved well as a result.
\end{remark}

\bigskip

With Lemma~\ref{lem:InnProdJL} in hand we are now able to prove a secondary subspace embedding result which, though it leads to suboptimal results in the vector setting, will be valuable for higher mode tensors.

\begin{lem}
Fix $\epsilon \in (0,1)$ and let $\mathcal{L}$ be an $r$-dimensional subspace of $\mathbbm{C}^{n_1 \times \dots \times n_d}$ spanned by a set of $r$ orthonormal basis tensors $\{ \mathcal{T}_k \}_{k \in [r]}$.  If $L$ is an $(\epsilon/4r)$-JL embedding of the $4 {r \choose 2} + r = 2r^2 - r$ tensors
$$\left( \bigcup_{1 \leq h < k \leq r} \left\{ {\mathcal T}_k - {\mathcal T}_h, {\mathcal T}_k + {\mathcal T}_h, {\mathcal T}_k - \mathbbm{i} {\mathcal T}_h,  {\mathcal T}_k + \mathbbm{i} {\mathcal T}_h \right\} \right) \bigcup \left\{ {\mathcal T}_k \right \}_{k \in [r]} \subset \mathcal{L}$$
into $\mathbbm{C}^{m_1 \times \dots \times m_{d'}}$, then
$$\left| \left\| L \left( {\mathcal X} \right) \right\|^2 -  \left\| {\mathcal X} \right\|^2\right|~\leq~ \epsilon\| {\mathcal X} \|^2$$
holds for all ${\mathcal X} \in \mathcal{L}$.
\label{lem:Subspaceembed2.0}
\end{lem}

\begin{proof}
Appealing to Lemma~\ref{lem:InnProdJL} we can see that $|\epsilon_{k,h}| := \left| \left \langle L \left( {\mathcal T}_k \right), L \left( {\mathcal T}_h \right) \right \rangle - \left \langle {\mathcal T}_k, {\mathcal T}_h\right \rangle \right| ~\leq~ \epsilon/r$ for all $h,k \in [r]$.  As a consequence, we have for any ${\mathcal X}~=~\sum^r_{k=1} \alpha_k {\mathcal T}_k \in \mathcal{L}$ that
\begin{align*}
\left| \left\| L \left( {\mathcal X} \right) \right\|^2 -  \left\| {\mathcal X} \right\|^2\right| ~&=~ \left| \sum^r_{k=1} \sum^r_{h=1} \alpha_k \overline{\alpha_h} \left(\left \langle L \left( {\mathcal T}_k \right),~L\left( {\mathcal T}_h \right) \right \rangle -  \left \langle  {\mathcal T}_k,~ {\mathcal T}_h \right \rangle \right)\right| ~=~ \left| \sum^r_{k=1} \sum^r_{h=1} \alpha_k \overline{\alpha_h} \epsilon_{k,h}\right|\\
&\leq~\sum^r_{k=1} \left| \alpha_k \right| \sum^r_{h=1} \left| \alpha_h \right| |\epsilon_{k,h}| ~\leq~ \sum^r_{k=1} \left| \alpha_k \right| \| {\boldsymbol \alpha} \|_2 \left(\frac{\epsilon}{\sqrt{r}} \right) ~\leq~ \epsilon \| {\boldsymbol \alpha} \|_2^2.
\end{align*}
To finish we now note that $\| {\mathcal X} \|^2 = \| {\boldsymbol \alpha} \|_2^2$ due to the orthonormality of the basis tensors $\{ \mathcal{T}_k \}_{k \in [r]}$.
\end{proof}

\section{Modewise Linear Johnson-Lindenstrauss Embeddings of Low-Rank Tensors}\label{sec:main}\

In this section, we consider low-rank tensor subspace embeddings for tensors with low-rank expansions in terms of rank-one tensors (i.e., for tensors with low-rank CP Decompositions).  Our general approach will be to utilize subspace embeddings along the lines of Lemmas~\ref{lem:simplenetsubspace} and~\ref{lem:Subspaceembed2.0} in this setting.  However, the fact that our basis tensors are rank-one will cause us some difficulties.  Principally, among those difficulties will be our inability to guarantee that we can find an orthonormal, or even fairly incoherent, basis of rank-one tensors that span any particular $r$-dimensional tensor subspace $\mathcal{L}$ we may be interested in below.

Going forward, we will consider the {\it standard form} of a given rank-$r$ $d$-mode tensor defined by 
\begin{equation}
\mathcal{Y} := \sum_{k=1}^r \alpha_k \bigcirc^d_{\ell = 1} {\bf y}^{(\ell)}_k \in \mathbbm{C}^{n_1 \times \dots \times n_d},
\label{equ:rankrTensor}
\end{equation}
where the vectors making up the rank-one basis tensors are normalized so that $\left\| {\bf y}^{(\ell)}_k \right\|_2 = 1$ for all $\ell \in [d]$ and $k \in [r]$. 

Given a set of rank-one tensors spanning a tensor subspace, one can define the coherence of the basis.
\begin{defin}[Modewise coherence of a basis of a rank-one tensors]\label{muBdef}
If a tensor subspace is spanned by a basis of rank-one tensors $\mathcal{B} :=\{\bigcirc^d_{\ell = 1} {\bf y}^{(\ell)}_k\,|\,k \in [r]\}$ with $\left\| {\bf y}^{(\ell)}_k \right\|_2 = 1$ for all $\ell \in [d]$ and $k \in [r]$, we denote the maximum modewise coherence of the basis and the basis coherence by
\begin{equation}\label{mubdef}
\mu_\mathcal{B} ~:=~ \max_{\ell \in [d]} \mu_{\mathcal{B},\ell} \quad \text{ and } \quad \mu'_{\mathcal{B}} := \max_{\substack{k, h \in [r]\\ k \neq h}}  \prod_{\ell = 1}^d \left| \left \langle {\bf y}^{(\ell)}_k,~{\bf y}^{(\ell)}_h \right \rangle \right|,
\end{equation}

\noindent respectively, where $\mu_{\mathcal{B},\ell}:=\max\limits_{\substack{k, h \in [r]\\ k \neq h}} \left| \left \langle {\bf y}^{(\ell)}_k , {\bf y}^{(\ell)}_h \right \rangle \right|$ is the modewise coherence of the basis for $\ell \in [d]$.
\end{defin}

Note that $\mu_\mathcal{B}, \mu'_\mathcal{B} \in [0,1]$ and that $\mu'_\mathcal{B} \le \prod_{\ell=1}^{d}\mu_{\mathcal{B},\ell} \le \mu_\mathcal{B}^d$ always hold.  
Given any tensor $\mathcal{Y}$ in the span of a basis $\mathcal{B}$ of rank-$1$ tensors we will also refer (with some abuse of notation) to its modewise coherence and maximum modewise coherence as being equal to the modewise coherence and maximum modewise coherence of the given basis $\mathcal{B}$ defined in Definition \ref{muBdef}.  That is, we will say that
\begin{equation}\label{mudef}\mu_{\mathcal{Y},\ell}= \mu_{\mathcal{B},\ell} \quad \text{for}~ \ell \in [d], \quad \text{and} \quad \mu_{\mathcal{Y}}= \mu_{\mathcal{B}}
\end{equation}
for all $\mathcal{Y} \in \mathcal{B}$.
Similarly, the basis coherence of any such $\mathcal{Y} \in \mathcal{B}$ will be said to equal the basis coherence also defined in Definition \ref{muBdef}, i.e.,
$\mu_{\mathcal{Y}}' = \mu_{\mathcal{B}}'$.  {\it It should be remembered below, however, that the quantities $\mu_{\mathcal{Y},\ell}$, $\mu_{\mathcal{Y}}$, $\mu_{\mathcal{Y}}'$ always depend on the particular basis $\mathcal{B}$ under consideration}. 

\bigskip
The main result of this section is the following oblivious subspace embedding theorem for low-rank tensors.

\begin{thm}\label{cor:MainOblEmb}
Fix $\delta, \eta \in \left(0,1/2\right)$ and $d \geq 2$. Let $\mathcal{L}$ be an $r$-dimensional subspace of $\mathbbm{C}^{n_1 \times \dots \times n_d}$ spanned by a basis of rank-$1$ tensors $\mathcal{B} := \left \{ \bigcirc^d_{\ell = 1} {\bf y}^{(\ell)}_k ~\big|~ k \in [r] \right\}$ with modewise coherence (as per~\eqref{mubdef}) satisfying $\mu_\mathcal{B}^{d-1} < 1/{2r}$. For each $j \in [d]$ draw ${\bf A}_j \in \mathbbm{C}^{m_j \times n_j}$ with 
\begin{equation}\label{eq:rDepOne}
m_j \geq \tilde C \cdot r^{2/d} d^2/\eps^2 \cdot \ln \left( 2r^2 d / {\eta}\right)
\end{equation}
from an $(\eta/d)$-optimal family of JL embedding distributions, where $\tilde C \in \mathbbm{R}^+$ is an absolute constant. Then with probability at least $1 - \eta$ we have
\begin{equation}\label{eq:main_dist_estimate}
\left| \left\| \mathcal{Y}\times_1 {\bf A}_1 \dots \times_d {\bf A}_d  \right\|^2 - \left\| \mathcal{Y} \right\|^2 \right| \leq \eps \left\|\mathcal{Y} \right\|^2,
\end{equation}
for all $\mathcal{Y} \in \mathcal{L}$.  
\end{thm}

\begin{remark}
Modewise incoherence assumption is necessary for our proof of Theorem~\ref{cor:MainOblEmb}. (Indeed, we initially obtain the upper estimate for the distortion in \eqref{eq:main_dist_estimate} in terms of $\|\alpha\|$ instead of $\|\mathcal{Y}\|$. As suggested by Lemma~\ref{lem:norm_comparison} below, in the case when $\mu_\mathcal{Y}$ is large these two norms can be incompatible.) However, numerical experiments with the coherent model show compatible results even for very coherent tensors. See, e.g., Figure~\ref{fig:JL_norm_synth} and additional relevant discussion in section~\ref{sec:exp}.
\end{remark}
The next subsection presents all the components of the proof of Theorem~\ref{cor:MainOblEmb}, whereas the details of the auxiliary lemmas and propositions are deferred to Appendix~\ref{AppMinSuppProofs}. 

\subsection{Proof of the Oblivious Tensor Subspace Embedding Theorem~\ref{cor:MainOblEmb}} \label{seq:proofof thm3}

The first auxiliary lemma deals with how $j$-mode products can change the standard form and modewise coherence of a given tensor that lies in a tensor subspace spanned by $r$ rank-$1$ tensors.

\begin{lem}
Let $j \in [d]$, ${\bf B} \in \mathbbm{C}^{m \times n_j}$, and $\mathcal{Y} \in \mathbbm{C}^{n_1 \times \dots \times n_d}$ be a rank-$r$ tensor as per \eqref{equ:rankrTensor} such that $\min_{k \in [r]} \left\| {\bf B} {\bf y}^{(j)}_k \right\|_2 > 0$.  Then
$\mathcal{Y}' := \mathcal{Y} \times_j {\bf B}$ can be written in standard form as
$$\mathcal{Y}' = \sum_{k=1}^r \alpha_k{\left\| {\bf B} {\bf y}^{(j)}_k \right\|_2} \left( \left( \bigcirc_{\ell < j} {\bf y}^{(\ell)}_k \right) \bigcirc \frac{{\bf B}{\bf y}^{(j)}_k}{\left\| {\bf B} {\bf y}^{(j)}_k \right\|_2} \bigcirc \left( \bigcirc_{\ell > j}^d {\bf y}^{(\ell)}_k \right) \right).$$
Furthermore, the $j$-mode coherence of $\mathcal{Y}'$ as above will satisfy 
$$\mu_{\mathcal{Y}',j} ~=~ \max_{\substack{k, h \in [r]\\ k \neq h}} \frac{\left| \left \langle {\bf B} {\bf y}^{(j)}_k , {\bf B} {\bf y}^{(j)}_h \right \rangle \right|}{\left\| {\bf B} {\bf y}^{(j)}_k \right\|_2 \left\| {\bf B} {\bf y}^{(j)}_h \right\|_2}$$
so that
\begin{align*}
\mu_{\mathcal{Y}'}
&=~\max\left( \mu_{\mathcal{Y}',j} ~, \max_{\ell \in [d]\setminus \{j \}} \max_{\substack{k, h \in [r]\\ k \neq h}} \left| \left \langle {\bf y}^{(\ell)}_k , {\bf y}^{(\ell)}_h \right \rangle \right| \right).
\end{align*}
\label{lem:NewStandForm}
\end{lem}

The proof of this and all subsequent intermediate results stated in this section can be found in Appendix~\ref{AppMinSuppProofs}.
The next lemma gives us a useful expression for the norm of a tensor after a $j$-mode product in terms of vector inner products.

\begin{lem}
Let $j \in [d]$, ${\bf B} \in \mathbbm{C}^{m \times n_j}$, and $\mathcal{Y} \in \mathbbm{C}^{n_1 \times \dots \times n_d}$ be a rank-$r$ tensor in standard form as per \eqref{equ:rankrTensor}.  Then,
\begin{equation*}
\| \mathcal{Y} \times_j {\bf B} \|^2 ~=~ \sum_{k,h=1}^r \sum^{\prod_{\ell \neq j} n_\ell}_{a = 1} \alpha_k \left( \otimes_{\ell \neq j} {\bf y}^{(\ell)}_k \right)_a \overline{\alpha_h \left( \otimes_{\ell \neq j} {\bf y}^{(\ell)}_h \right)_a} \left \langle  {\bf B}  {\bf y}^{(j)}_k,  {\bf B} {\bf y}^{(j)}_h \right \rangle,
\end{equation*}
where $(\mathbf{u})_a$ denotes the $a^{\rm th}$ coordinate of a vector $\mathbf{u}$.
\label{lem:ModejNormexp}
\end{lem}

The following proposition demonstrates that a single modewise Johnson-Lindenstrauss embedding of any low-rank tensor $\mathcal{Y}$ of the form \eqref{equ:rankrTensor} will preserve its norm up to an error depending on the overall $\ell^2$-norm of its coefficients ${\boldsymbol \alpha} \in \mathbbm{C}^r$.  In order to accomplish this, we will connect  Johnson-Lindenstrauss embeddings of combinations of the basis vectors ${\bf y}^{(\ell)}_k$ to the embedding properties of any low-rank tensor $\mathcal{Y}$ of the form \eqref{equ:rankrTensor}.  We will employ Lemma~\ref{lem:InnProdJL} for this purpose and consider the following sets $\mathcal{S}'_{j}$ defined using the basis vectors ${\bf y}^{(\ell)}_k$.  For each mode $j \in [d]$ of any rank-$r$ tensor as per \eqref{equ:rankrTensor}, we can associate the following set $\mathcal{S}'_{j}$ of $4 {r \choose 2} + r = 2r^2 - r$ vectors in $\mathbbm{C}^{n_j}$ that will be of use for us later together with Lemma~\ref{lem:InnProdJL} 
\begin{equation}\label{sets}
\mathcal{S}'_{j} := \left( \bigcup_{1 \leq h < k \leq r} \left\{ {\bf y}^{(j)}_k - {\bf y}^{(j)}_h, {\bf y}^{(j)}_k + {\bf y}^{(j)}_h, {\bf y}^{(j)}_k - \mathbbm{i} {\bf y}^{(j)}_h,  {\bf y}^{(j)}_k + \mathbbm{i} {\bf y}^{(j)}_h \right\} \right) \bigcup \left\{ {\bf y}^{(j)}_k \right\}.
\end{equation}

\begin{prop}
Let $j \in [d]$ and $\mathcal{Y} \in \mathbbm{C}^{n_1 \times \dots \times n_d}$ be a rank-$r$ tensor as per \eqref{equ:rankrTensor}.  Suppose that $\mathbf{A} \in \mathbbm{C}^{m \times n_j}$ is an $\left(\epsilon / 4 \right)$-JL embedding of all the vectors from the set $\mathcal{S}'_{j}$ defined as per~\eqref{sets} into $\mathbbm{C}^{m}$.  Let $\mathcal{Y}' := \mathcal{Y} \times_j {\bf A}$ and rewrite it in standard form so that
$$\mathcal{Y}' = \sum_{k=1}^r \alpha'_k \left( \left( \bigcirc_{\ell < j} {\bf y}^{(\ell)}_k \right) \bigcirc \frac{{\bf A}{\bf y}^{(j)}_k}{\left\| {\bf A} {\bf y}^{(j)}_k \right\|_2} \bigcirc \left( \bigcirc_{\ell > j}^d {\bf y}^{(\ell)}_k \right) \right).$$
Then all of the following hold:\\
\begin{enumerate}
\item[($\dagger$)]  $\left| \alpha'_k - \alpha_k \right| ~\leq~ \epsilon |\alpha_k| / 4$ for all $k \in [r]$ so that $\| \boldsymbol \alpha' \|_\infty \leq (1 + \epsilon/4) \| \boldsymbol \alpha \|_\infty$. \\

\item[($\dagger \dagger$)] $\mu_{\mathcal{Y}',j} ~\leq~ \frac{\mu_{\mathcal{Y},j} + \epsilon}{1 - \epsilon/4}$, and $\mu_{\mathcal{Y}',\ell} ~=~ \mu_{\mathcal{Y},\ell}$ for all $\ell \in [d] \setminus \{ j \}$.\\

\item[($\dagger \dagger \dagger$)] $\displaystyle \left| \| \mathcal{Y}' \|^2 -  \| \mathcal{Y} \|^2\right| ~\leq~ \epsilon (r+1) \| {\boldsymbol \alpha} \|^2_2$.

\end{enumerate}

\label{thm:ForInduction}
\end{prop}

Note that part ($\dagger \dagger \dagger$) of Proposition~\ref{thm:ForInduction} bounds $\displaystyle \left| \| \mathcal{Y}' \|^2 -  \| \mathcal{Y} \|^2\right|$ with respect to $\| {\boldsymbol \alpha} \|^2_2$. Traditional JL-type error guarantees typically want to prove error bounds of the form $\displaystyle \left| \| \mathcal{Y}' \|^2 -  \| \mathcal{Y} \|^2\right| \leq C_\epsilon \|  \mathcal{Y} \|^2$, however.  The next lemma bounds $\| {\boldsymbol \alpha} \|^2_2$ by $\| \mathcal{Y} \|^2$ so that the reader who desires such bounds can obtain them easily for any tensor with sufficiently small modewise coherence.

\begin{lem}\label{lem:norm_comparison}
Let $\mathcal{Y} \in \mathbbm{C}^{n_1 \times \dots \times n_d}$ be a rank-$r$ tensor as per \eqref{equ:rankrTensor} with the basis coherence $\mu_{\mathcal{Y}}' < (r-1)^{-1}$. Then,
$$\| {\boldsymbol \alpha} \|^2_2 ~\leq~ \left( \frac{1}{1 - (r-1) \mu_{\mathcal{Y}}'} \right) \| \mathcal{Y} \|^2 ~\leq~ \left( \frac{1}{1 - (r-1) \prod^d_{\ell =1} \mu_{\mathcal{Y},\ell}} \right) \| \mathcal{Y} \|^2 ~\leq~ \left( \frac{1}{1 - (r-1) \mu_\mathcal{Y}^d} \right) \| \mathcal{Y} \|^2.$$
\label{lem:coefintermsofYnorm}
\end{lem}

We are now prepared to establish Proposition~\ref{thm:ObliviousSubspaceEmbedded}, our main component of the proof of Theorem~\ref{cor:MainOblEmb} in this section.  Recall that combining it with Lemma~\ref{lem:coefintermsofYnorm} provides traditional JL-embedding error bounds.

\begin{prop}
Let $\epsilon \in (0, 3/4]$, $\mathcal{Y} \in \mathbbm{C}^{n_1 \times \dots \times n_d}$ be a rank-$r$ tensor expressed in standard form as per $\eqref{equ:rankrTensor}$, and ${\bf A}_j \in \mathbbm{C}^{m_j \times n_j}$ be an $\left(\epsilon / 4d \right)$-JL embedding of all the vectors from the set $\mathcal{S}'_{j}$ defined as per~\eqref{sets} into $\mathbbm{C}^{m_j}$ for each $j \in [d]$.  Then,
\begin{align}\label{theor2result}
\left| \left\| \mathcal{Y} \right\|^2 - \left\| \mathcal{Y} \times_1 {\bf A}_1 \dots \times_d {\bf A}_d \right\|^2 \right| ~&\leq~ \epsilon   \left( \mathbbm{e} +  \mathbbm{e}^{2} \sqrt{r(r-1)} \cdot \max\left( \epsilon^{d-1} , \mu_{\mathcal{Y}}^{d-1} \right) \right)  \| {\boldsymbol \alpha} \|^2_2 \\
&\leq~\epsilon \mathbbm{e}^{2}  \left( r + 1\right)  \| {\boldsymbol \alpha} \|^2_2 \nonumber
\end{align}
always holds. Here, $\mu_\mathcal{Y}$ is maximum modewise coherence of the tensor defined by \eqref{mudef}.
Furthermore, if $\mu_{\mathcal{Y}} = 0$ then
$$\left| \left\| \mathcal{Y} \right\|^2 - \left\| \mathcal{Y} \times_1 {\bf A}_1 \dots \times_d {\bf A}_d \right\|^2 \right| ~\leq~ \left( \epsilon +  \mathbbm{e}\sqrt{r(r-1)} \epsilon^d \right) \mathbbm{e} \| {\boldsymbol \alpha} \|^2_2.$$
\label{thm:ObliviousSubspaceEmbedded}
\end{prop}

In addition, Proposition~\ref{thm:ObliviousSubspaceEmbedded} can be extended to show that modewise compression preserves scalar products between two tensors $\mathcal{X}$ and $\mathcal{Y}$ spanned by the same rank-one tensors: 
\begin{prop}
Suppose that both $\mathcal{X}, \mathcal{Y} \in \mathbbm{C}^{n_1 \times \dots \times n_d}$ can be represented in terms of the same basis $\{ {\bf y}^{(\ell)}_k\}$ for $k = 1, \ldots, r$ and $l  = 1, \ldots, d$ in their standard form~\eqref{equ:rankrTensor}.
Let $\epsilon \in (0, 3/4]$, and ${\bf A}_j \in \mathbbm{C}^{m_j \times n_j}$ be a $\left(\epsilon / 4d \right)$-JL embedding of the set $S_j'$ defined as per~\eqref{sets} for each $j \in [d]$.  
Then,
\begin{align*}
\left| \left\langle \mathcal{X} \times_{j = 1}^d \mathbf{A}_j,~ \mathcal{Y} \times_{j = 1}^d \mathbf{A}_j \right\rangle - \left\langle \mathcal{X},~ \mathcal{Y}\right\rangle \right| 
&\leq 4 \epsilon' \cdot \frac{\max \left\{ \| \mathcal{X} \|^2, \| \mathcal{Y} \|^2 \right\}}{1 - \left(r-1\right) \mu_{\mathcal{Y}}'}
\end{align*}
where
\begin{equation}\label{epsprime}
\epsilon' :=  
\begin{cases}
\left( \epsilon + \mathbbm{e} \sqrt{r(r-1)} \epsilon^d \right) \mathbbm{e} & \textrm{if}~\mu_{\mathcal{Y}} = 0, \\
\epsilon  \left( \mathbbm{e} +  \mathbbm{e}^{2} \sqrt{r(r-1)} \cdot \max\left( \epsilon^{d-1} , \mu_{\mathcal{Y}}^{d-1} \right) \right) & \textrm{otherwise.} 
\end{cases}
\end{equation}
\label{cor:PreserveInnerProds}
\end{prop}

\begin{proof}
Using the polarization identity in combination with Lemma~\ref{lem:ProdProps} and Proposition~\ref{thm:ObliviousSubspaceEmbedded}, we can see that
\begin{align*}
\left| \left\langle \mathcal{X} \times_{j = 1}^d \mathbf{A}_j,~ \mathcal{Y} \times_{j = 1}^d \mathbf{A}_j \right\rangle - \left\langle \mathcal{X},~ \mathcal{Y}\right\rangle \right| &= \left| \frac{1}{4} \sum^3_{\ell = 0} \mathbbm{i}^\ell \left( \left\| \mathcal{X} \times_{j = 1}^d \mathbf{A}_j + \mathbbm{i}^\ell \mathcal{Y} \times_{j = 1}^d \mathbf{A}_j \right\|^2_2  -  \left\| \mathcal{X} + \mathbbm{i}^\ell \mathcal{Y} \right\|^2_2 \right)\right|\\ 
&\leq \frac{1}{4} \sum^3_{\ell = 0}  \epsilon' \left\| {\boldsymbol \beta } + \mathbbm{i}^\ell {\boldsymbol \alpha} \right\|^2_2 ~\leq~ \epsilon' \left( \left\| {\boldsymbol \beta } \right\|_2 + \left \| {\boldsymbol \alpha} \right\|_2 \right)^2\\
&\leq 2 \epsilon' \left( \| {\boldsymbol \beta } \|_2^2 + \| {\boldsymbol \alpha} \|_2^2 \right) ~\leq~ 4 \epsilon' \cdot \max \left\{ \| {\boldsymbol \beta } \|_2^2, \| {\boldsymbol \alpha} \|_2^2 \right\},
\end{align*}
where the second to last inequality follows from Young's inequality for products.  An application of Lemma~\ref{lem:coefintermsofYnorm} yields the final inequality.
\end{proof}

Propositions~\ref{thm:ObliviousSubspaceEmbedded} and~\ref{cor:PreserveInnerProds} guarantee that modewise JL-embeddings approximately preserve the norms and inner products between all tensors in the span of the set 
$$\mathcal{B} := \left \{ \bigcirc^d_{\ell = 1} {\bf y}^{(\ell)}_k ~\big|~ k \in [r] \right\} \subset \mathbbm{C}^{n_1 \times \dots \times n_d}.$$  
Let
$$\mathcal{L} := \textrm{span} \left( \left \{ \bigcirc^d_{\ell = 1} {\bf y}^{(\ell)}_k ~\big|~ k \in [r] \right\} \right).$$
Employing $\eta$-optimal JL embeddings (as per Definition~\ref{def:EtaOpt}), we can now prove the main result of this section, Theorem~\ref{cor:MainOblEmb}:

\begin{proof}[Proof of Theorem~\ref{cor:MainOblEmb}] Let $\mathcal{Y} \in \mathcal{L}$. By Proposition~\ref{cor:PreserveInnerProds}, the linear operator $L$ defined as $L(\mathcal{Z}) = \mathcal{Z} \times_1 {\bf A}_1 \dots \times_d {\bf A}_d $ is an $\eps$-JL embedding of $\mathcal{Y}$ if 
\begin{itemize}
\item $4/\left(1 - (r-1)\mu'_{\mathcal{B}}\right) \le 8$, and
\item each ${\bf A}_j$ is an $\left(\delta/4d\right)$-JL embedding of the set $S_j'$ of cardinality $|S_j'| \le 2r^2 - r$, where the dependence $\eps'(\delta)$ is defined by \eqref{epsprime}, and $\eps \ge 8 \eps'$.
\end{itemize}
The first condition is satisfied since basis incoherence condition implies 
$$\mu'_{\mathcal{B}} \le \mu^d_{\mathcal{B}} \le 1/2\left(r-1\right).$$ 
Hence, $8(1 - (r-1)\mu'_{\mathcal{B}}) \ge 4$. To check the second condition, note that due to \eqref{epsprime}, it is enough to use an $\eps$ that satisfies
$$
\eps \ge 8\delta e + 8\delta e^2 r \max\left(\delta^{d-1}, \mu_{\mathcal{B}}^{d-1}\right),
$$
and having $\delta := \eps/16e \cdot \left(1/r\right)^{1/d}$ ensures that it does. Finally, if each matrix ${\bf A}_j$ is taken from an $\left(\eta/d\right)$-optimal family of JL distributions it will be an $\left(\delta/4d\right)$-JL embedding of $S_j'$ into $\mathbbm{C}^{m_j}$ with probability $1 - \eta/d$ as long as
$$
|S_j'| = 2r^2 - r \le \frac{\eta}{d} \exp\left(\frac{\delta^2 m_j}{16d^2C}\right),
$$
which is satisfied for each $m_j$ defined in \eqref{eq:rDepOne}. Taking union bound over all $d$ modes then concludes the proof of Theorem~\ref{cor:MainOblEmb}. 
\end{proof}

\begin{remark}[JL-type embeddings for low-rank matrices]
Theorem~\ref{cor:MainOblEmb} (as well as the above results, including Proposition~\ref{thm:ObliviousSubspaceEmbedded}) can be applied in the special case where $\mathcal{X} = {\bf X}$ is a matrix in $\mathbbm{C}^{n_1 \times n_2}$. In this case, the CP-rank is the usual matrix rank, and the CP decomposition becomes the regular SVD decomposition of the matrix, which can be computed efficiently in parallel (see, e.g., \cite{iwen2016distributed}).  In particular, the basis vectors are orthogonal to each other in this case. The result of Theorem~\ref{cor:MainOblEmb} implies that taking $A$ and $B$ as matrices belonging to the $\left(\eta/2\right)$-JL embedding family and of sizes $n_1 \times m_1$ and $n_2 \times m_2$, respectively, such that $m_j \gtrsim r \ln(r/\sqrt{\eta})/\eps^2$ (for $j = 1, 2$), we get the following JL-type result for the Frobenius matrix norm: with probability $1 - \eta$,
$$
\|A^T {\bf X} B\|_F^2 = (1 + \tilde \eps)\|{\bf X}\|_F^2 \quad \text{ for some } |\tilde \eps| \le \eps.
$$
\end{remark}

\subsection{Naturally incoherent tensor bases}
\label{sec:IncoherentBasesareCommon}
Again, we remind the reader that Lemma~\ref{lem:coefintermsofYnorm} can be used in combination with the theorems and corollaries above/below in order to provide JL-embedding results of the usual type.  In order for Lemma~\ref{lem:coefintermsofYnorm} to apply, however, we need the coherence $\mu_{\mathcal{B}}'$ of the basis $\mathcal{B}$ to satisfy $\mu_{\mathcal{B}}' < (r-1)^{-1}$.  One popular set of bases with this property are those that result from considering tensors whose Tucker decompositions \cite{tucker1966some,kolda2001orthogonal,iwen2016distributed} have core tensors with a small number of nonzero entries.  More specifically, let $\mathcal{C} \in \mathbbm{C}^{n_1 \times \dots \times n_d}$, ${\bf U}^{(j)} \in \mathbbm{C}^{n_j \times n_j}$ be unitary for all $j \in [d]$, and $\mathcal{S} \subset [n_1] \times \dots \times [n_d]$ be a set of $r$ indices in $\mathcal{C}$.  Now consider the $r$-dimensional tensor subspace 
$$\mathcal{L}_{\rm Tucker} := \left \{ \mathcal{X} ~\Big|~ \mathcal{X} ~=~ \mathcal{C} \times_{j=1}^d {\bf U}^{(j)} ~\textrm{with}~ \mathcal{C}_{\bf i} = 0 ~\textrm{for~all}~ {\bf i} \notin \mathcal{S} \right\}.$$
One can see that any tensor $\mathcal{Y} \in \mathcal{L}_{\rm Tucker}$ can be written in standard form as per \eqref{equ:rankrTensor} with, for all $\ell \in [d]$, ${\bf y}^{(\ell)}_k = {\bf U}^{(\ell)}_{k'}$ for some column $k' \in [n_\ell]$.  As a result, $\mu_{\mathcal{Y}}' = \mu_{\mathcal{B}}' = 0$ will hold due to the orthogonality of the columns of each ${\bf U}^{(\ell)}$ matrix.  We therefore have the following special case of Proposition~\ref{thm:ObliviousSubspaceEmbedded} in this setting.

\begin{cor}
Suppose that $\mathcal{Y} \in \mathcal{L}_{\rm Tucker} \subset \mathbbm{C}^{n_1 \times \dots \times n_d}$.
Let $\epsilon \in (0, 3/4]$, and ${\bf A}_j \in \mathbbm{C}^{m_j \times n_j}$ be defined as per Proposition~\ref{thm:ObliviousSubspaceEmbedded} for each $j \in [d]$.  
Then,
$$\left| \left\| \mathcal{Y} \right\|^2 - \left\| \mathcal{Y} \times_1 {\bf A}_1 \dots \times_d {\bf A}_d \right\|^2 \right| ~\leq~ \epsilon'  \left\| \mathcal{Y} \right\|^2,$$
where
\begin{equation*}
\epsilon' :=  
\begin{cases}
\left( \epsilon + \mathbbm{e} \sqrt{r(r-1)} \epsilon^d \right) \mathbbm{e} & \textrm{if}~\mu_{\mathcal{B}} = 0, \\
\epsilon  \left( \mathbbm{e} +  \mathbbm{e}^{2} \sqrt{r(r-1)} \cdot \max\left( \epsilon^{d-1} , \mu_{\mathcal{B}}^{d-1} \right) \right) & \textrm{otherwise}. 
\end{cases}
\end{equation*}
\label{cor:TuckerLowRank}
\end{cor}

\begin{proof}
This follows from Proposition~\ref{thm:ObliviousSubspaceEmbedded} combined with Lemma~\ref{lem:coefintermsofYnorm} after noting that $\mu_{\mathcal{B}}' = 0$ holds.
\end{proof}

\bigskip

Another natural set of bases on which the property $\mu_{\mathcal{B}}' < (r-1)^{-1}$ is satisfied is the random family of sub-gaussian tensors.
The following Lemma~\ref{lem:gaussian_incoherence} shows that if all the components of all vectors ${\bf y}_{k}^{(j)}$ (for $j \in [d],k \in [r]$) are normalized independent $K$-subgaussian random variables (see Definition~\ref{def:subgaus} below), the coherence is actually low with high probability.
\begin{defin}\label{def:subgaus} A random variable $\xi$ is called K-subgaussian, if for all $t \ge 0$
$$
\p\left\{|\xi| > t\right\} \le 2 \exp\left(-t^2/K^2\right).
$$
\end{defin}
Informally, all normal random variables (with any mean and variance), and also those with lighter tails are K-subgaussian with some proper constant $K$. All bounded random variables are subgaussian.

\begin{lem}\label{lem:gaussian_incoherence} Let $\mu > 0$. Let $j \in [d]$ and $\mathcal{Y} \in \mathbbm{C}^{n_1 \times \dots \times n_d}$ be a rank-$r$ tensor as per \eqref{equ:rankrTensor}. Let $n = \min\limits_{i \in [d]} n_i$. If all components of all vectors ${\bf y}_k^{(j)}$ are normalized independent mean zero $K$-subgaussian random variables, with probability at least 
$1 - 2r^2d\exp\left(-c \mu^2 n\right)$ the maximum modewise coherence parameter of the tensor $\mathcal{Y}$ is at most $\mu$. Here, $c$ is a positive constant depending only on $K$.
\end{lem}
\begin{proof} For any $k \in [r]$ and $j\in [d]$ denote $\tilde {\bf y}_k^{(j)} := {\bf y}_k^{(j)}\cdot \|{\bf \tilde y}_k^{(j)}\|$. By definition, ${\bf \tilde y}_k^{(j)}$ are independent $K$-subgaussian random variables for all $k \in [r]$ and $j \in [d]$. Therefore, their norms are of order $\sqrt{n}$ with high probability: for any fixed $k, j$,
$$\p\left\{n/2 \le \|\tilde {\bf y}_k^{(j)}\|_2^2 \le 2n\right\} \ge 1 - 2 \exp\left(-c_1n/K^4\right)$$
(see, e.g. [\cite{vershynin2018high}, Section 3.1]).  Taking union bound, we can conclude that with probability at least $1 - 2rd\exp\left(-c_1n/K^4\right)$, all vectors ${\bf \tilde y}_k^{(j)}$ have their norms between $[\sqrt{n/2}, \sqrt{2n}]$.
   
For any mean zero independent $K$-subgaussian vectors ${\bf x}$ and ${\bf y}$,
\begin{align}
   &\p\left\{|\langle {\bf x}, {\bf y}\rangle| \ge \mu\|{\bf x}\| \|{\bf y}\|\right\} \nonumber\\
   &\le \p\left\{|\langle {\bf x}, {\bf y}\rangle| \ge \mu \|{\bf y}\| \sqrt{n/2}\right\} + \p\left\{\|{\bf x}\| < \sqrt{n/2}\right\}.
\end{align}
To bound the first term, let us use Hoeffding's inequality (see, e.g. [\cite{vershynin2018high}, Theorem 2.6.3]). Conditioning on ${\bf y}$, we have
$$\p_{\bf x}\left\{\Big| \sum_i x_i y_i \Big| \ge \mu \|{\bf y}\| \sqrt{n/2}\right\} \le 2 \exp\left(-\frac{c_2\mu^2 n}{2K^2}\right).
$$
Now, let ${\bf \tilde y}_k^{(j)} = {\bf x}$ and ${\bf \tilde y}_l^{(j)} = {\bf y}$. Integrating over ${\bf \tilde y}_l^{(j)}$ and then taking union bound over all choices of $k, l$ and $j$, we get $\left|\langle {\bf y}_k^{(j)}, {\bf y}_l^{(j)}\rangle \right| \le \mu$ for all component vectors in the tensor $\mathcal{Y}$ with probability at least 
$$1 - 2r^2d\exp\left( -\frac{c_2\mu^2 n}{2K^2}\right) -  2rd\exp\left(-\frac{c_1n}{K^4}\right) \ge 1 - 2r^2d\exp\left(-c \mu^2 n\right).$$
Lemma~\ref{lem:gaussian_incoherence} is proved.
\end{proof}

The following two elementary corollaries illustrate the applicability of our theory to independent subgaussian tensors.
In these corollaries, the term {\it subgaussian tensor} always refers to a tensor defined as per Lemma~\ref{lem:gaussian_incoherence}, and should not be confused with a tensor with subgaussian elements.
\begin{cor} 
Let $\eps \in (0, 3/4]$. Let $\mathcal{Y}$ be a subgaussian tensor defined as in Lemma~\ref{lem:gaussian_incoherence}. For low-rank tensors in high-dimensional spaces, such that 
$$n := \min\limits_{i \in [d]} n_i \ge \frac{\log(r^2d)}{\eps^2c}$$ 
(the small constant $c$ is the same as in Lemma~\ref{lem:gaussian_incoherence}), with probability at least $ 1 - \exp\left(c'\eps^2n\right)$, Proposition~\ref{thm:ObliviousSubspaceEmbedded} holds with better dependence on $\eps$, namely,
$$
\left| \left\| \mathcal{Y} \right\|^2 - \left\| \mathcal{Y} \times_1 {\bf A}_1 \dots \times_d {\bf A}_d \right\|^2 \right| \leq~\left(\eps^d r + \eps\right) e^{2}   \| {\boldsymbol \alpha} \|^2_2.
$$
Here, $c' >0$ is an absolute constant.
\end{cor}
\begin{proof} Apply Lemma~\ref{lem:gaussian_incoherence} with $\mu = \eps$. 
\end{proof}
\begin{cor} Let $\mathcal{Y}$ be a subgaussian tensor defined as in Lemma~\ref{lem:gaussian_incoherence}. If 
$$n := \min\limits_{i \in [d]} n_i \ge C r^{2/d} \log\left(\max\left(r,d\right)\right),$$ with probability at least $1 - \exp\left(-c'n/r^{2/d}\right)$, Lemma~\ref{lem:norm_comparison} gives a non-trivial lower bound $\|\mathcal{Y}\| \ge 0.99 \|{\boldsymbol \alpha}\|$. Here, $c' >0$ is an absolute constant.
 
In particular, the claim holds when $r\le C_1^d$ and $n \ge C_2 \max\{r,d\}$.
\end{cor}
\begin{proof} Apply Lemma~\ref{lem:gaussian_incoherence} with $\mu = \left(\frac{0.01}{r-1}\right)^{d^{-1}}$.
\end{proof}
\begin{remark}
Note that in the general case, when $r$ can be as large as $O(n^d)$, the $\mu_{\mathcal{Y}}$ estimate given in Lemma~\ref{lem:gaussian_incoherence} is not strong enough. Indeed, to have a non-trivial probability estimate, one must take $\mu > \sqrt{2d \log n/n})$. However, $\mu_{\mathcal{Y}} \sim \sqrt{d \log n/n} $ together with $r \sim n^d$ do not satisfy the condition of Lemma~\ref{lem:norm_comparison}, since $\left(r-1\right)\mu_{\mathcal{Y}}^d \sim (dn \log n)^{d/2} \gg 1$.

One could use alternative more sophisticated anti-concentration results instead of Lemma~\ref{lem:norm_comparison}. For example,  it was shown recently in \cite{vershynin2019concentration} that for any $r \le 0.99n^d$ and under some mild conditions,  $\|\mathcal{Y}\| \ge c n^{-d/2} \|{\boldsymbol \alpha}\|_2$ (in the independent subgaussian setting as discussed above). Note that this result contains additional non-favorable dependence on $n$. To the best of our knowledge, it is an open question whether general systems of independent (sub)gaussian vectors form tensors that satisfy norm anti-concentration like the one in Lemma~\ref{lem:norm_comparison}. See also the discussion in \cite{vershynin2019concentration}. 
\end{remark}

\section{Applications to Least Squares Problems and fitting CP models}
\label{sec:CPleastSquaresstuff}
Now, let us consider the following \emph{fitting} problem. Given tensor $\mathcal{X}$, which is suspected to have (approximately) low CP-rank $r$, we would like to find the rank-$r$ tensor $\mathcal{Y}$ in the standard form, as per~\eqref{equ:rankrTensor}, being closest to $\mathcal{X}$ in the tensor Euclidean norm. Although the $r$-dimensional basis (subspace) of $\mathcal{Y}$ is naturally unknown, a common way to tackle the fitting problem is to start with a randomly generated basis, and then update the basis tensors mode by mode improving the least square error. This brings us to a framework considered in the previous section: a tensor $\mathcal{Y}$ being in some fixed low-dimensional subspace at each step. Since this subspace is changing throughout the fitting process, the oblivious subspace dimension reduction technique is desirable. The fitting problem can be considered as a generalization of the embedding problem introduced in the previous section (with the addition of a potentially full-rank tensor $\mathcal{X}$ that is being approximated). 

In this section, we formalize the fitting problem and explain how we propose to use modewise dimension reduction for it. Then, we develop the machinery generalizing our methods from Section~\ref{cor:MainOblEmb} to incorporate an unknown tensor $\mathcal{X}$. Finally, we propose a more-sophisticated two-step dimension reduction process that further improves the resulting dimension for both embedding and fitting problems to almost log-optimal order $\mathcal{O}(r \eps^{-2})$.

As explained above, the common alternating least squares approach for fitting a low-rank CP decomposition along the lines of \eqref{equ:rankrTensor} to an arbitrary tensor $\mathcal{X} \in \mathbbm{C}^{n_1 \times \dots \times n_d}$ involves solving a sequence of least squares problems
\begin{equation}
\argmin_{\tilde{\bf y}^{(j)}_1, \dots, \tilde{\bf y}^{(j)}_r \in \mathbbm{C}^{n_j}} \left\| \mathcal{X} - \sum_{k=1}^r \alpha_k \bigcirc^d_{\ell = 1} {\bf y}^{(\ell)}_k \right \|
\label{equ:ALSbasic}
\end{equation}
for each $j \in [d]$ after fixing $\left \{ {\bf y}^{(\ell)}_k \right\}_{k\in [r], \ell \in [d] \setminus \{ j \} }$.  Here,  ${\bf y}^{(j)}_k = \tilde{\bf y}^{(j)}_k/\| \tilde{\bf y}^{(j)}_k\|_2$ $\forall j,k$ and $\alpha_k = \prod_{\ell = 1}^d \| \tilde{\bf y}^{(\ell)}_k\|_2$. One then varies $j$ through all values in $[d]$ computing \eqref{equ:ALSbasic} for each $j$ in order to update ${\bf y}^{(j)}_k$ $\forall j,k$ (potentially cycling through all $d$ modes many times).  This makes it particularly important to solve each least squares problem \eqref{equ:ALSbasic} efficiently.

Fix $j \in [d]$ and let ${\bf e}_h \in \mathbbm{C}^{n_j}$ be the $h^{th}$ column of the $n_j \times n_j$ identity matrix.  To see how our modewise tensor subspace embeddings can be of value for solving \eqref{equ:ALSbasic}, one can begin by noting that 
\begin{align*}
\left\| \mathcal{X} - \sum_{k=1}^r \alpha_k \bigcirc^d_{\ell = 1} {\bf y}^{(\ell)}_k \right \|^2 ~&=~ \left\| \mathbf{X}_{(j)} - \sum^r_{k = 1}  \alpha_k {\bf y}^{(j)}_k \left( \otimes_{\ell \neq j} {\bf y}^{(\ell)}_k \right)^\top \right\|_{\rm F}^2\\
&=~ \left\| \sum^{n_j}_{h=1} \left( \mathbf{X}^{(h)}_{(j)} - \sum^r_{k = 1}  \alpha_k {y}^{(j)}_{k,h} {\bf e}_h \left( \otimes_{\ell \neq j} {\bf y}^{(\ell)}_k \right)^\top \right) \right\|^2_{\rm F}
\end{align*}
where $\mathbf{X}_{(j)}$ denotes mode-$j$ matricization of $\mathcal{X}$, and all the rows of $\mathbf{X}^{(h)}_{(j)} \in \mathbbm{C}^{n_j \times \prod_{\ell \neq j} n_\ell}$ are zero except for its $h^{th}$-row which matches that of $\mathbf{X}_{(j)}$. We may now compute the squared Frobenius norm directly above row-wise and get that 
\begin{align*}
\left\| \mathcal{X} - \sum_{k=1}^r \alpha_k \bigcirc^d_{\ell = 1} {\bf y}^{(\ell)}_k \right \|^2 ~&=~ \sum^{n_j}_{h=1} \left\| \mathbf{x}_{j,h} - \sum^r_{k = 1}  \alpha_k {y}^{(j)}_{k,h} \left( \otimes_{\ell \neq j} {\bf y}^{(\ell)}_k \right) \right\|_{\rm F}^2\\
&=~ \sum^{n_j}_{h=1} \left\| \mathcal{X}^{(j,h)} - \sum^r_{k = 1}  \alpha_k {y}^{(j)}_{k,h} \bigcirc_{\ell \neq j} {\bf y}^{(\ell)}_k \right\|^2
\end{align*}
where $\mathbf{x}_{j,h} \in \mathbbm{C}^{\prod_{\ell \neq j} n_\ell}$ denotes the $h^{th}$-row of $\mathbf{X}_{(j)}$, and $\mathcal{X}^{(j,h)}$ its tensorized version.
As a consequence, \eqref{equ:ALSbasic} can be decoupled into $n_j$ separate least squares problems of the form
\begin{equation}
\argmin_{{\boldsymbol \alpha'_{j,h}} \in \mathbbm{C}^{r}} \left\| \mathcal{X}^{(j,h)} - \sum_{k=1}^r \alpha'_{j,h,k} \bigcirc^d_{\ell \neq j} {\bf y}^{(\ell)}_k \right \|
\label{equ:DecoupledALS}
\end{equation}
each involving one $(d-1)$-mode mode-$j$ slice, $\mathcal{X}^{(j,h)}$, of the original tensor $\mathcal{X}$.\footnote{$\mathcal{X}^{(j,h)}$ is in fact the $h^{\rm th}$ mode-$j$ slice of $\mathcal{X}$.} Here $\alpha'_{j,h,k} :=  \alpha_k {y}^{(j)}_{k,h}$ where $\alpha_k$ is known $\forall k \in [r]$ from \eqref{equ:ALSbasic}.  Note also that these $n_j$ separate least squares problems can,  if desired, be solved in parallel for each different $h \in [n_j]$.  

In order to solve each least squares problem \eqref{equ:DecoupledALS} we can now utilize modewise JL embeddings and instead solve the smaller least squares problem
\begin{equation}
\argmin_{{\boldsymbol \alpha'_{j,h}} \in \mathbbm{C}^{r}} \left\| \mathcal{X}^{(j,h)} \bigtimes_{\ell \neq j} {\bf A}_\ell  - \sum_{k=1}^r \alpha'_{j,h,k} \bigcirc^d_{\ell \neq j} {\bf y}^{(\ell)}_k \bigtimes_{\ell \neq j} {\bf A}_\ell \right \|
\label{equ:compDecoupledALS}
\end{equation}
provided that the $\left\{ {\bf y}^{(\ell)}_k \right\}_{k \in [r]}$ are sufficiently incoherent for all $\ell \in [d]\setminus \{j\}$ (an easy to check condition).  We can then update each entry of $\tilde {\bf y}^{(j)}_k$ by setting $\tilde {y}^{(j)}_{k,h} = \alpha'_{j,h,k} / \alpha_k$ for all $h \in [n_j]$ and $k \in [r]$.

\subsection{General Modewise JL embeddings for Tensors with Low Modewise Coherence}
\label{sec:GenModewiseLSsec}

We prove that the method described above works in the following Theorem~\ref{thm:GenmodewiseConstruction}, showing that the solution to \eqref{equ:compDecoupledALS} will be close to that of \eqref{equ:DecoupledALS} in terms of quality if the matrices ${\bf A}_j$ are chosen from appropriate $\eta$-optimal JL families of distributions:
\begin{thm}
Fix $\epsilon, \eta \in (0,1/2)$ and $d \geq 3$.  Let $\mathcal{X} \in \mathbbm{C}^{n_1 \times \dots \times n_d}$, $n := \max\limits_{j} n_j \geq 4r+1$, and $\mathcal{L}$ be an $r$-dimensional subspace of $\mathbbm{C}^{n_1 \times \dots \times n_d}$ spanned by a basis $\mathcal{B} := \left \{ \bigcirc^d_{\ell = 1} {\bf y}^{(\ell)}_k ~\big|~ k \in [r] \right\}$ of rank-$1$ tensors, with modewise coherence satisfying $\mu_\mathcal{B}^{d-1} < 1/{2r}$.  For each $j \in [d]$ draw ${\bf A}_j \in \mathbbm{C}^{m_j \times n_j}$ with 
\begin{equation}\label{eq:rDepTwo}
m_j \geq C_j \cdot r d^3/\epsilon^2 \cdot \ln \left( n / \sqrt[d]{\eta}\right)
\end{equation}
from an $(\eta/4d)$-optimal family of JL embedding distributions, where $C_j \in \mathbbm{R}^+$ is an absolute constant.  Furthermore, let ${\bf A} \in \mathbbm{C}^{m' \times \prod^d_{\ell=1} m_\ell}$ with 
$$m' \geq C' r \cdot \epsilon^{-2} \cdot \ln \left( \frac{47}{\epsilon \sqrt[r]{\eta}} \right)$$
be drawn from an $(\eta/2)$-optimal family of JL embedding distributions, where $C' \in \mathbbm{R}^+$ is an absolute constant.  Define $\tilde{L}: \mathbbm{C}^{n_1 \times \dots \times n_d} \rightarrow \mathbbm{C}^{m_1 \times \dots \times m_d}$ by $L(\mathcal{Z}) = \mathcal{Z} \times_1 {\bf A}_1 \dots \times_d {\bf A}_d$.  Then, with probability at least $1 - \eta$, the linear operator ${\bf A} \circ \mathrm{vect} \circ \tilde{L}:  \mathbbm{C}^{n_1 \times \dots \times n_d}  \rightarrow \mathbbm{C}^{m'}$ satisfies
\begin{equation*}
\left| \left\| {\bf A} \left(\mathrm{vect} \circ \tilde{L}\left( \mathcal{X} - \mathcal{Y} \right) \right) \right\|^2_2 - \left\|  \mathcal{X} - \mathcal{Y} \right\|^2 \right| \leq \epsilon \left\|  \mathcal{X} - \mathcal{Y} \right\|^2
\end{equation*}
for all $\mathcal{Y} \in \mathcal{L}$.  
\label{thm:GenmodewiseConstruction}
\end{thm}

\begin{remark}[About $r$ and $\eps$ Dependence]
\label{Rem:Repsdepthm4}
Fix $d, n,$ and $\eta$.  Looking at Theorem~\ref{thm:GenmodewiseConstruction} we can see that it's intermediate embedding dimension is 
$$\prod_{\ell = 1}^d m_\ell \leq  C_{d,\eta,n}^d r^d \eps^{-2d}$$ 
which effectively determines its overall storage complexity. 
Hence, Theorem~\ref{thm:GenmodewiseConstruction} will only result in an improved memory complexity over the straightforward single-stage vectorization approach if, e.g., the rank $r$ of $\mathcal{L}$ is relatively small.
The purpose of facultative vectorization and subsequent multiplication by an additional JL transform $\bf{ A}$ in Theorem~\ref{thm:GenmodewiseConstruction} is to reduce the resulting final embedding dimension to the near-optimal order $\mathcal{O}(r/\eps^2)$ from total dimension $\mathcal{O}_{\eta,n}(d^{3d} r^d \eps^{-2d})$ that we have after the modewise compression.
\end{remark}

In order to prove Theorem~\ref{thm:GenmodewiseConstruction}, we first establish that $\left\| \mathcal{X}^{(j,h)} \bigtimes_{\ell \neq j} {\bf A}_\ell \right\| \approx \left\| \mathcal{X}^{(j,h)}\right\|$ can also hold for all $j \in [d]$ and $h \in [n_j]$.  This is shown in the following lemma which is proven in Appendix~\ref{AppMinSuppProofs}.

\begin{lem}
Let $\epsilon \in (0, 1)$, $\mathcal{Z}^{(1)}, \dots, \mathcal{Z}^{(p)} \in \mathbbm{C}^{n_1 \times \dots \times n_d}$, and ${\bf A}_1 \in \mathbbm{C}^{m_1 \times n_1}$ be an $\left(\epsilon / \mathbbm{e}d \right)$-JL embedding of the all $p \left( \prod_{\ell = 2}^d n_\ell \right)$ mode-$1$ fibers of all $p$ of these tensors,
$$\mathcal{S}_1 := \bigcup_{t \in [p]}\left\{ \mathcal{Z}^{(t)}_{:,i_2,\dots,i_{d}}~|~ \forall i_{\ell} \in [n_\ell],~\ell \in [d] \setminus \{1\} \right\}\subset \mathbbm{C}^{n_1},$$
into $\mathbbm{C}^{m_1}$.  Next, set $\mathcal{Z}^{(1,t)} := \mathcal{Z}^{(t)} \times_1 {\bf A}_1 \in \mathbbm{C}^{m_1 \times n_2 \times \dots \times n_d}$~$\forall t \in [p]$, and then let ${\bf A}_2 \in \mathbbm{C}^{m_2 \times n_2}$ be an $\left(\epsilon / \mathbbm{e}d \right)$-JL embedding of all $p\left(m_1\prod_{\ell = 3}^d n_\ell \right)$ mode-$2$ fibers 
$$\mathcal{S}_2 := \bigcup_{t \in [p]} \left\{ \mathcal{Z}^{(1,t)}_{i_1,:,i_3,\dots,i_{d}}~|~ \forall {i_1} \in [m_1] ~\&~i_{\ell} \in [n_\ell],~\ell \in [d] \setminus [2] ~\right\}\subset \mathbbm{C}^{n_2}$$ 
into $\mathbbm{C}^{m_2}$. Continuing inductively, for each $j \in [d] \setminus [2]$ and $t \in [p]$ set $\mathcal{Z}^{(j-1,t)} := \mathcal{Z}^{(j-2,t)} \times_{j-1} {\bf A}_{j-1} \in \mathbbm{C}^{m_1 \times \dots \times m_{j-1} \times n_j \times \dots \times n_d}$, and then let ${\bf A}_j \in \mathbbm{C}^{m_j \times n_j}$ be an $\left(\epsilon / \mathbbm{e}d \right)$-JL embedding of all $p\left(\prod^{j-1}_{\ell = 1}m_\ell \right) \left( \prod_{\ell = j+1}^d n_\ell \right)$ mode-$j$ fibers 
$$\mathcal{S}_j := \bigcup_{t \in [p]} \left\{ \mathcal{Z}^{(j-1,t)}_{i_1,\dots,i_{j-1},:,i_{j+1},\dots,i_{d}}~|~ \forall {i_\ell} \in [m_\ell], \ell \in [j-1] ~\&~i_{\ell} \in [n_\ell],\ell \in [d] \setminus [j], ~\right\}\subset \mathbbm{C}^{n_j}$$ 
into $\mathbbm{C}^{m_j}$.
Then,
\begin{align*}
\left| \left\| \mathcal{Z}^{(t)} \right\|^2 - \left\| \mathcal{Z}^{(t)} \times_1 {\bf A}_1 \dots \times_d {\bf A}_d \right\|^2 \right| ~&\leq~ \epsilon \left\| \mathcal{Z}^{(t)} \right\|^2
\end{align*}
will hold for all $t \in [p]$.
\label{lem:IndivTensorJL}
\end{lem}

With Lemma~\ref{lem:IndivTensorJL} in hand we can now prove that the solution to \eqref{equ:compDecoupledALS} will be close to that of \eqref{equ:DecoupledALS} in terms of quality if the matrices ${\bf A}_j$ are chosen appropriately.  We have the following general result which directly applies to least squares problems as per \eqref{equ:compDecoupledALS} when $\tilde{L}(\mathcal{Z}) := \mathcal{Z} \bigtimes_{\ell \neq j} {\bf A}_\ell $ and ${\bf A} = {\bf I}$.

\begin{thm}[Embeddings for Compressed Least Squares]
Let $\mathcal{X} \in \mathbbm{C}^{n_1 \times \dots \times n_d}$, $\mathcal{L}$ be an $r$-dimensional subspace of $\mathbbm{C}^{n_1 \times \dots \times n_d}$ spanned by a set of orthonormal basis tensors $\{ \mathcal{T}_k \}_{k \in [r]}$, and $\mathbbm{P}_{\mathcal{L}^\perp}: \mathbbm{C}^{n_1 \times \dots \times n_d} \rightarrow \mathbbm{C}^{n_1 \times \dots \times n_d}$ be the orthogonal projection operator on the orthogonal complement of $\mathcal{L}$.  Fix $\epsilon \in (0,1)$ and suppose that the linear operator $\tilde{L}: \mathbbm{C}^{n_1\times n_2\times ... \times n_d} \rightarrow \mathbbm{C}^{m_1 \times \dots \times m_{d'}}$ has both of the following properties:
\begin{enumerate}
\item[(i)] $\tilde{L}$ is an $(\epsilon/6)$-JL embedding of all $\mathcal{Y} \in \mathcal{L} \cup \left\{ \mathbbm{P}_{\mathcal{L}^\perp}(\mathcal{X}) \right\}$ into $\mathbbm{C}^{m_1 \times \dots \times m_{d'}}$, and
\item[(ii)] $\tilde{L}$ is an $(\epsilon/24\sqrt{r})$-JL embedding of the $4r$ tensors
$$\mathcal{S}' := \bigcup_{k \in [r]} \left\{  \frac{\mathbbm{P}_{\mathcal{L}^\perp}(\mathcal{X})}{ \left \| \mathbbm{P}_{\mathcal{L}^\perp}(\mathcal{X}) \right \|} - \mathcal{T}_k,  \frac{\mathbbm{P}_{\mathcal{L}^\perp}(\mathcal{X})}{ \left \| \mathbbm{P}_{\mathcal{L}^\perp}(\mathcal{X}) \right \|} + \mathcal{T}_k, \frac{\mathbbm{P}_{\mathcal{L}^\perp}(\mathcal{X})}{ \left \| \mathbbm{P}_{\mathcal{L}^\perp}(\mathcal{X}) \right \|} - \mathbbm{i}\mathcal{T}_k,  \frac{\mathbbm{P}_{\mathcal{L}^\perp}(\mathcal{X})}{ \left \| \mathbbm{P}_{\mathcal{L}^\perp}(\mathcal{X}) \right \|} + \mathbbm{i}\mathcal{T}_k \right\} \subset \mathbbm{C}^{n_1\times n_2\times ... \times n_d}$$
into $\mathbbm{C}^{m_1 \times \dots \times m_{d'}}$.
\end{enumerate}
Furthermore, let $\mathrm{vect}: \mathbbm{C}^{m_1 \times \dots \times m_{d'}} \rightarrow \mathbbm{C}^{\prod^{d'}_{\ell = 1} m_\ell}$ be a reshaping vectorization operator, and ${\bf A} \in \mathbbm{C}^{m \times \prod^{d'}_{\ell = 1} m_\ell}$ be an $(\epsilon/3)$-JL embedding of the $(r+1)$-dimensional subspace 
$$\mathcal{L}' := \mathrm{span} \left\{ \mathrm{vect} \circ \tilde{L}\left(\mathbbm{P}_{\mathcal{L}^\perp}(\mathcal{X}) \right),~ \mathrm{vect} \circ \tilde{L}\left( \mathcal{T}_1 \right), ~\dots,~ \mathrm{vect} \circ \tilde{L}\left( \mathcal{T}_r \right) \right\} \subset \mathbbm{C}^{\prod^{d'}_{\ell = 1} m_\ell}$$
into $\mathbbm{C}^{m}$. Then,
\begin{equation*}
\left| \left\| {\bf A} \left(\mathrm{vect} \circ \tilde{L}\left(  \mathcal{X} - \mathcal{Y} \right) \right) \right\|^2_2 - \left\|  \mathcal{X} - \mathcal{Y} \right\|^2 \right| \leq \epsilon \left\|  \mathcal{X} - \mathcal{Y} \right\|^2
\end{equation*}
holds for all $\mathcal{Y} \in \mathcal{L}$.
\label{thm:SubspaceEmbedResult}
\end{thm}

\begin{proof}
Note that the theorem will be proven if $\tilde{L}$ is an $(\epsilon / 3)$--JL embedding of all tensors of the form $\left\{ \mathcal{X} - \mathcal{Y} ~\big|~ \mathcal{Y} \in \mathcal{L} \right\}$ into $\mathbbm{C}^{m_1 \times \dots \times m_{d'}}$ since any such tensor $\mathcal{X} - \mathcal{Y}$ will also have $\mathrm{vect} \circ \tilde{L} \left( \mathcal{X} - \mathcal{Y} \right) \in \mathcal{L}' $ so that
\begin{align*}
\Big| \left\| {\bf A} \left(\mathrm{vect} \circ \tilde{L}\left(  \mathcal{X} - \mathcal{Y} \right) \right) \right\|^2_2 &- \left\|  \mathcal{X} - \mathcal{Y} \right\|^2 \Big| \\
&\leq~ \left| \left\| {\bf A} \left(\mathrm{vect} \circ \tilde{L}\left(  \mathcal{X} - \mathcal{Y} \right) \right) \right\|^2_2 - \left\| \tilde{L} \left( \mathcal{X} - \mathcal{Y} \right) \right\|^2 \right| + \left| \left\| \tilde{L} \left( \mathcal{X} - \mathcal{Y} \right) \right\|^2 - \left\|  \mathcal{X} - \mathcal{Y} \right\|^2 \right|\\
&\leq~ \left| \left\| {\bf A} \left(\mathrm{vect} \circ \tilde{L}\left(  \mathcal{X} - \mathcal{Y} \right) \right) \right\|^2_2 - \left\| \mathrm{vect} \circ  \tilde{L} \left( \mathcal{X} - \mathcal{Y} \right) \right\|^2_2 \right| + \frac{\epsilon}{3}\left\|  \mathcal{X} - \mathcal{Y} \right\|^2 \\
&\leq~ \frac{\epsilon}{3} \left\| \mathrm{vect} \circ  \tilde{L} \left( \mathcal{X} - \mathcal{Y} \right) \right\|^2_2 + \frac{\epsilon}{3}\left\|  \mathcal{X} - \mathcal{Y} \right\|^2 \\
&=~ \frac{\epsilon}{3} \left\| \tilde{L} \left( \mathcal{X} - \mathcal{Y} \right) \right\|^2 + \frac{\epsilon}{3}\left\|  \mathcal{X} - \mathcal{Y} \right\|^2 \\
&\leq~\frac{\epsilon}{3} \left( 1 + \frac{\epsilon}{3} \right)  \left\|  \mathcal{X} - \mathcal{Y} \right\|^2 + \frac{\epsilon}{3} \left\|  \mathcal{X} - \mathcal{Y} \right\|^2 ~\leq~  \epsilon \left\|  \mathcal{X} - \mathcal{Y} \right\|^2.
\end{align*}

Let $\mathbbm{P}_{\mathcal{L}}$ be the orthogonal projection operator onto $\mathcal{L}$.  Our first step in establishing that $\tilde{L}$ is an $(\epsilon/3)$--JL embedding of all tensors of the form $\left\{ \mathcal{X} - \mathcal{Y} ~\big|~ \mathcal{Y} \in \mathcal{L} \right\}$ into $\mathbbm{C}^{m_1 \times \dots \times m_{d'}}$ will be to show that $\tilde{L}$ preserves all the angles between $\mathbbm{P}_{\mathcal{L}^\perp}(\mathcal{X})$ and $\mathcal{L}$ well enough that the Pythagorean theorem 
$$\| \mathcal{X} - \mathcal{Y} \|^2 ~=~ \| \mathbbm{P}_{\mathcal{L}^\perp}(\mathcal{X}) + \mathbbm{P}_{\mathcal{L}} \left(\mathcal{X} \right) - \mathcal{Y} \|^2 ~=~ \| \mathbbm{P}_{\mathcal{L}^\perp}(\mathcal{X}) \|^2 + \| \mathbbm{P}_{\mathcal{L}} \left(\mathcal{X} \right) - \mathcal{Y} \|^2$$
still approximately holds for all $\mathcal{Y} \in \mathcal{L}$ after $\tilde{L}$ is applied.  Toward that end, let $\boldsymbol{\gamma} \in \mathbbm{C}^r$ be such that $\mathbbm{P}_{\mathcal{L}} \left(\mathcal{X} \right) - \mathcal{Y} = \sum_{k \in [r]} \gamma_k \mathcal{T}_k$ and note that $\| \boldsymbol{\gamma} \|_2 = \left\| \mathbbm{P}_{\mathcal{L}} \left(\mathcal{X} \right) - \mathcal{Y} \right\|$ due to the orthonormality of $\{ \mathcal{T}_k \}_{k \in [r]}$.  Appealing to Lemma~\ref{lem:InnProdJL} we now have that
\begin{align}
\left| \left \langle \tilde{L} \left( \mathbbm{P}_{\mathcal{L}} \left(\mathcal{X} \right) - \mathcal{Y} \right),~ \tilde{L} \left( \mathbbm{P}_{\mathcal{L}^\perp}(\mathcal{X}) \right) \right \rangle \right| ~&=~ \left \| \mathbbm{P}_{\mathcal{L}^\perp}(\mathcal{X}) \right \| \left| \sum_{k \in [r]} \gamma_k \left \langle \tilde{L} \left( \mathcal{T}_k \right),~ \tilde{L} \left( \frac{\mathbbm{P}_{\mathcal{L}^\perp}(\mathcal{X})}{ \left \| \mathbbm{P}_{\mathcal{L}^\perp}(\mathcal{X}) \right \|} \right) \right \rangle \right| \nonumber \\ 
&\leq~ \left \| \mathbbm{P}_{\mathcal{L}^\perp}(\mathcal{X}) \right \| \left(\frac{\epsilon}{6\sqrt{r}} \right) \sum_{k \in [r]} \left|\gamma_k \right| ~\leq~  \frac{\epsilon}{6} \left \| \mathbbm{P}_{\mathcal{L}^\perp}(\mathcal{X}) \right \| \| \boldsymbol{\gamma} \|_2 \label{equ:LkillsCrossterm} \\
&\leq \frac{\epsilon}{12} \left( \left\| \mathbbm{P}_{\mathcal{L}^\perp}(\mathcal{X}) \right \|^2 + \left\| \mathbbm{P}_{\mathcal{L}} \left(\mathcal{X} \right) - \mathcal{Y} \right\|^2 \right)~=~ \frac{\epsilon}{12} \| \mathcal{X} - \mathcal{Y} \|^2. \nonumber
\end{align}

Using \eqref{equ:LkillsCrossterm} we can now see that
\begin{align*}
\Big| \left\| \tilde{L}\left(  \mathcal{X} - \mathcal{Y} \right) \right\|^2_2 &- \left\|  \mathcal{X} - \mathcal{Y} \right\|^2 \Big|\\ &=~ \left| \left\| \tilde{L}\left(  \mathcal{X} - \mathcal{Y} \right) \right\|^2_2 -  \| \mathbbm{P}_{\mathcal{L}^\perp}(\mathcal{X}) \|^2 - \| \mathbbm{P}_{\mathcal{L}} \left(\mathcal{X} \right) - \mathcal{Y} \|^2 \right|\\
&\leq \left| \left \| \tilde{L} \left(\mathbbm{P}_{\mathcal{L}^\perp}(\mathcal{X}) \right) \right \|^2 - \| \mathbbm{P}_{\mathcal{L}^\perp}(\mathcal{X}) \|^2 \right| + \left| \left \| \tilde{L} \left( \mathbbm{P}_{\mathcal{L}} \left(\mathcal{X} \right) - \mathcal{Y} \right) \right\|^2 - \| \mathbbm{P}_{\mathcal{L}} \left(\mathcal{X} \right) - \mathcal{Y} \|^2  \right| \\
 &\hspace{2.2in}+2 \left| \left \langle \tilde{L} \left( \mathbbm{P}_{\mathcal{L}} \left(\mathcal{X} \right) - \mathcal{Y} \right),~ \tilde{L} \left( \mathbbm{P}_{\mathcal{L}^\perp}(\mathcal{X}) \right) \right \rangle \right|\\
&\leq~\frac{\epsilon}{6} \left( \| \mathbbm{P}_{\mathcal{L}^\perp}(\mathcal{X}) \|^2 + \| \mathbbm{P}_{\mathcal{L}} \left(\mathcal{X} \right) - \mathcal{Y} \|^2 + \| \mathcal{X} - \mathcal{Y} \|^2 \right)~=~\frac{\epsilon}{3} \| \mathcal{X} - \mathcal{Y} \|^2.
\end{align*}
Thus, $\tilde{L}$ has the desired JL-embedding property required to conclude the proof.
\end{proof}

Theorems~\ref{thm:ObliviousSubspaceEmbedded} and~\ref{thm:SubspaceEmbedResult} together with Lemma~\ref{lem:IndivTensorJL} can now be used to demonstrate the existence of a large range of modewise Johnson-Lindenstrauss Transforms (JLTs) for oblivious tensor subspace embeddings. The following modewise JLT result for tensors describes the compression one can achieve from Theorem~\ref{thm:SubspaceEmbedResult} if the linear operator $L$ one employs is formed using $j$-mode products (as considered in Proposition~\ref{thm:ObliviousSubspaceEmbedded}) with ${\bf A}_j \in \mathbbm{C}^{m_j \times n_j}$ taken from $\eta$-optimal families of JL embedding distributions (in the sense of Definition~\ref{def:EtaOpt}). 

We are now ready to complete the proof of Theorem~\ref{thm:GenmodewiseConstruction}.

\begin{proof}[Proof of Theorem~\ref{thm:GenmodewiseConstruction}]
To begin, we note that ${\bf A}$ will satisfy the conditions required by Theorem~\ref{thm:SubspaceEmbedResult} with probability at least $1-\eta / 2$ as a consequence of Lemma~\ref{lem:simplenetsubspace}.  Thus, if we can also establish that $\tilde{L}$ will satisfy the conditions required by Theorem~\ref{thm:SubspaceEmbedResult} with probability at least $1-\eta / 2$, we will be finished with our proof by Theorem~\ref{thm:SubspaceEmbedResult} and the union bound.  

To establish that $\tilde{L}$ satisfies the conditions required by Theorem~\ref{thm:SubspaceEmbedResult} with probability at least $1-\eta / 2$, it suffices to prove that 
\begin{enumerate}
\item[(a)] $\tilde{L}$ will be an $(\epsilon/6)$-JL embedding of all $\mathcal{Y} \in \mathcal{L}$ into $\mathbbm{C}^{m_1 \times \dots \times m_{d}}$ with probability at least $1-\eta / 4$, and that
\item[(b)] $\tilde{L}$ will be an $(\epsilon/24\sqrt{r})$-JL embedding of the $4r+1$ tensors $\mathcal{S}' \cup \left\{ \mathbbm{P}_{\mathcal{L}^\perp}(\mathcal{X}) \right\} \subset \mathbbm{C}^{n_1\times n_2\times ... \times n_d}$ into $\mathbbm{C}^{m_1 \times \dots \times m_{d}}$ with probability at least $1-\eta / 4$, where the set $\mathcal{S}'$ is defined as in Theorem~\ref{thm:SubspaceEmbedResult},
\end{enumerate}   
and apply yet another union bound.

To show that (a) holds, we will utilize Proposition~\ref{thm:ObliviousSubspaceEmbedded} and Lemma~\ref{lem:coefintermsofYnorm}.  Since each ${\bf A}_j$ matrix is an $(\eta/4d)$-optimal JL embedding and the sets $\mathcal{S}'_{j}$ (defined as in Proposition~\ref{thm:ObliviousSubspaceEmbedded}) are such that $| \mathcal{S}'_{j}| < n^d$, we know that each ${\bf A}_j$ is an $\left( \epsilon/480 d\sqrt{r} \right)$-JL embedding of $\mathcal{S}'_{j}$ into $\mathbbm{C}^{m_j}$ with probability\footnote{Here we also implicitly use the fact that $\sqrt[d]{d} \leq \sqrt[\mathbbm{e}]{\mathbbm{e}}$ holds for all $d > 0$ in order to avoid a $\sqrt[d]{d}$ term appearing inside the logarithm in \eqref{eq:rDepTwo}.} at least $1-\eta/4d$.
Thus, Proposition~\ref{thm:ObliviousSubspaceEmbedded} holds with $\epsilon \rightarrow \epsilon/ 120 \sqrt{r}$ with probability at least $1-\eta/4$
. Note that the modewise coherence assumption that $\mu_\mathcal{B}^{d-1} < 1/{2r}$ both allows $\epsilon^{d-1}$ to reduce the $\sqrt{r(r-1)}$ factor in \eqref{theor2result} to a size less than one for any $\epsilon \leq 1/\sqrt{r} \leq (1/r)^{1/(d-1)}$, and also allows Lemma~\ref{lem:coefintermsofYnorm} to guarantee that $\left\| {\boldsymbol \alpha} \right\|_2^2 < 2 \left\| \mathcal{Y} \right\|^2$ holds for all $\mathcal{Y} \in \mathcal{L}$.  Hence, applying Proposition~\ref{thm:ObliviousSubspaceEmbedded} with $\epsilon \rightarrow \epsilon/120 \sqrt{r}$ will ensure that $\tilde{L}$ is an $(\epsilon/6)$-JL embedding of all $\mathcal{Y} \in \mathcal{L}$ into $\mathbbm{C}^{m_1 \times \dots \times m_{d}}$.  

To show that (b) holds we will utilize Lemma~\ref{lem:IndivTensorJL}.  Note that the $\mathcal{S}_j$ sets defined in Lemma~\ref{lem:IndivTensorJL} all have cardinalities $\left| \mathcal{S}_j \right| \leq p \cdot n^{d-1}$, where $p = 4r+1 \leq n$ in our current setting.  As a consequence we can see that the conditions of Lemma~\ref{lem:IndivTensorJL} will be satisfied with $\epsilon \rightarrow \epsilon/24\sqrt{r}$ for all $j \in [d]$ with probability at least $1-\eta/4$ by the union bound.  Hence, both (a) and (b) hold and our proof is concluded.
\end{proof}

We will now consider a final tensor subspace embedding result concerning a special case of modewise JL embeddings that is also made possible by our work above.  This result will exhibit better dependence with respect to both $\epsilon$ and $r$ than what is achieved by the more general modewise embedding constructions in Theorem~\ref{thm:GenmodewiseConstruction}.

\subsection{Fast and Memory Efficient Modewise JL Embeddings for Tensors}
\label{sec:fastFembed}

In this section we consider a fast Johnson-Lindenstrauss transform for tensors recently introduced in \cite{jin2019faster}, which is effectively based on applying fast JL transforms \cite{krahmer2011new} in a modewise fashion.\footnote{In fact, the fast transform described here differs cosmetically from the form in which it is presented in \cite{jin2019faster}.  However, one can easily see they are equivalent using \eqref{equ:jmodeprod_vect}.}  Given a tensor $\mathcal{Z} \in \mathbbm{C}^{n_1 \times \dots \times n_d}$ the transform takes the form
\begin{equation}
L_{\rm FJL}\left(\mathcal{Z}\right)~:=~\sqrt{\frac{N}{m}}~{\bf R} \left(\mathrm{vect} \left( \mathcal{Z} \times_1 {\bf F}_1{\bf D}_1 \dots \times_d {\bf F}_d{\bf D}_d \right) \right)
\label{equ:FastJLWard}
\end{equation}
where $\mathrm{vect}:  \mathbbm{C}^{n_1 \times \dots \times n_d} \rightarrow \mathbbm{C}^N$ for $N := \prod_{\ell = 1}^d n_\ell$ is the vectorization operator, ${\bf R} \in \{ 0,1\}^{m \times N}$ is a matrix containing $m$ rows selected randomly from the $N \times N$ identity matrix, ${\bf F}_\ell \in \mathbbm{C}^{n_\ell \times n_\ell}$ is a unitary discrete Fourier transform matrix for all $\ell \in [d]$, and ${\bf D}_\ell \in \mathbbm{C}^{n_\ell \times n_\ell}$ is a diagonal matrix with $n_\ell$ random $\pm 1$ entries for all $\ell \in [d]$.  The following theorem is proven about this transform in \cite{jin2019faster,krahmer2011new}.

\begin{thm}[See Theorem~2.1 and Remark~4 in \cite{jin2019faster}]
Fix $d \geq 1$, $\epsilon, \eta \in (0,1)$, and $N \geq C' / \eta$ for a sufficiently large absolute constant $C' \in \mathbbm{R}^+$.  Consider a finite set $\mathcal{S} \subset \mathbbm{C}^{n_1 \times \dots \times n_d}$ of cardinality $p = \left| \mathcal{S} \right|$, and let $L_{\rm FJL}: \mathbbm{C}^{n_1 \times \dots \times n_d} \rightarrow \mathbbm{C}^m$ be defined as above in \eqref{equ:FastJLWard} with 
$$m ~\geq~ C \left[ \epsilon^{-2} \cdot \log^{2d-1} \left( \frac{\max(p,N)}{\eta} \right) \cdot \log^4 \left( \frac{\log \left(\frac{\max(p,N)}{\eta} \right)}{\epsilon} \right) \cdot \log N\right],$$
where $C > 0$ is an absolute constant.  Then with probability at least $1 - \eta$ the linear operator $L_{\rm FJL}$ is an $\epsilon$-JL embedding of $\mathcal{S}$ into $\mathbbm{C}^m$.  If $d = 1$ then we may replace $\max(p,N)$ with $p$ inside all of the logarithmic factors above (see \cite{krahmer2011new}).
\label{thm:RachelsJL4tensors}
\end{thm}

Note that the fast transform $L_{\rm FJL}$ requires only $\mathcal{O}\left(m \log N + \sum_\ell n_\ell\right)$ i.i.d. random bits and memory for storage.  Thus, it can be used to produce fast and low memory complexity oblivious subspace embeddings.  The next Theorem does so.

\begin{thm}
Fix $\epsilon, \eta \in (0,1/2)$ and $d \geq 2$.  Let $\mathcal{X} \in \mathbbm{C}^{n_1 \times \dots \times n_d}$, $N = \prod_{\ell = 1}^d n_\ell \geq 4C' / \eta$ for an absolute constant $C' > 0$, $\mathcal{L}$ be an $r$-dimensional subspace of $\mathbbm{C}^{n_1 \times \dots \times n_d}$ for $\max \left(2r^2 - r,4r \right) \leq N$, and $L_{\rm FJL}: \mathbbm{C}^{n_1 \times \dots \times n_d} \rightarrow \mathbbm{C}^{m_1}$ be defined as above in \eqref{equ:FastJLWard} with 
$$m_1 ~\geq~ C_1 \left[ C_2^d \left(\frac{r}{\epsilon}\right)^2 \cdot \log^{2d-1} \left( \frac{N}{\eta} \right) \cdot \log^4 \left( \frac{\log \left(\frac{N}{\eta} \right)}{\epsilon} \right) \cdot \log N\right],$$
where $C_1, C_2 > 0$ are absolute constants.  Furthermore, let ${\bf L'}_{\rm FJL} \in \mathbbm{C}^{m_2 \times m_1}$ be defined as above in \eqref{equ:FastJLWard} for $d = 1$ with 
$$m_2 ~\geq~ C_3 \left[ r \cdot \epsilon^{-2} \cdot \log \left( \frac{47}{\epsilon \sqrt[r]{\eta}} \right) \cdot \log^4 \left( \frac{r \log \left(\frac{47}{\epsilon \sqrt[r]{\eta}} \right)}{\epsilon} \right) \cdot \log m_1\right],$$
where $C_3 > 0$ is an absolute constant.  Then, with probability at least $1 - \eta$ it will be the case that
\begin{equation*}
\left| \left\| {\bf L'}_{\rm FJL}  \left(L_{\rm FJL}\left(  \mathcal{X} - \mathcal{Y} \right) \right) \right\|^2_2 - \left\|  \mathcal{X} - \mathcal{Y} \right\|^2 \right| \leq \epsilon \left\|  \mathcal{X} - \mathcal{Y} \right\|^2
\end{equation*}
holds for all $\mathcal{Y} \in \mathcal{L}$.

In addition, the $\left( {\bf L'}_{\rm FJL} , L_{\rm FJL} \right)$ transform pair requires only $\mathcal{O}\left(m_1 \log N + \sum_\ell n_\ell\right)$
random bits and memory for storage (assuming w.l.o.g. that $m_2 \leq m_1$), and $ {\bf L'}_{\rm FJL} \circ L_{\rm FJL}:  \mathbbm{C}^{n_1 \times \dots \times n_d} \rightarrow \mathbbm{C}^{m_2}$ can be applied to any tensor in just $\mathcal{O}\left(N \log N \right)$-time.
\label{thm:RachelNewSubspaceEmbed}
\end{thm}

\begin{proof}
Let $\{ \mathcal{T}_k \}_{k \in [r]}$ be an orthonormal basis for $\mathcal{L}$ (note that these basis tensors need not be low-rank), and $\mathbbm{P}_{\mathcal{L}^\perp}: \mathbbm{C}^{n_1 \times \dots \times n_d} \rightarrow \mathbbm{C}^{n_1 \times \dots \times n_d}$ be the orthogonal projection operator onto the orthogonal complement of $\mathcal{L}$.   Theorem~\ref{thm:SubspaceEmbedResult} combined with Lemmas~\ref{lem:Subspaceembed2.0}~and~\ref{lem:simplenetsubspace} imply that the result will be proven if all of the following hold:
\begin{enumerate}
\item[(i)] $L_{\rm FJL}$ is an $(\epsilon/24r)$-JL embedding of the $2r^2 - r$ tensors
$$\left( \bigcup_{1 \leq h < k \leq r} \left\{ {\mathcal T}_k - {\mathcal T}_h, {\mathcal T}_k + {\mathcal T}_h, {\mathcal T}_k - \mathbbm{i} {\mathcal T}_h,  {\mathcal T}_k + \mathbbm{i} {\mathcal T}_h \right\} \right) \bigcup \left\{ {\mathcal T}_k \right \}_{k \in [r]} \subset \mathcal{L}$$
into $\mathbbm{C}^{m_1}$,
\item[(ii)] $L_{\rm FJL}$ is an $(\epsilon/6)$-JL embedding of $\left\{ \mathbbm{P}_{\mathcal{L}^\perp}(\mathcal{X}) \right\}$ into $\mathbbm{C}^{m_1}$, 
\item[(iii)] $L_{\rm FJL}$ is an $(\epsilon/24\sqrt{r})$-JL embedding of the $4r$ tensors
$$\bigcup_{k \in [r]} \left\{  \frac{\mathbbm{P}_{\mathcal{L}^\perp}(\mathcal{X})}{ \left \| \mathbbm{P}_{\mathcal{L}^\perp}(\mathcal{X}) \right \|} - \mathcal{T}_k,  \frac{\mathbbm{P}_{\mathcal{L}^\perp}(\mathcal{X})}{ \left \| \mathbbm{P}_{\mathcal{L}^\perp}(\mathcal{X}) \right \|} + \mathcal{T}_k, \frac{\mathbbm{P}_{\mathcal{L}^\perp}(\mathcal{X})}{ \left \| \mathbbm{P}_{\mathcal{L}^\perp}(\mathcal{X}) \right \|} - \mathbbm{i}\mathcal{T}_k,  \frac{\mathbbm{P}_{\mathcal{L}^\perp}(\mathcal{X})}{ \left \| \mathbbm{P}_{\mathcal{L}^\perp}(\mathcal{X}) \right \|} + \mathbbm{i}\mathcal{T}_k \right\} \subset \mathbbm{C}^{n_1\times ... \times n_d}$$
into $\mathbbm{C}^{m_1}$, and
\item[(iv)] $ {\bf L'}_{\rm FJL}$ is an $(\epsilon/6)$-JL embedding of a minimal $(\epsilon/16)$-cover, $\mathcal{C}$, of the $r$-dimensional Euclidean unit sphere in the subspace $\mathcal{L}' \subset \mathbbm{C}^{m_1}$ from Theorem~\ref{thm:SubspaceEmbedResult} with $L = L_{\rm FJL}$ into $\mathbbm{C}^{m_2}$.  Here we note that $\left| \mathcal{C} \right| \leq \left(\frac{47}{\epsilon} \right)^{r}$.
\end{enumerate}
Furthermore, if $m_1$ and $m_2$ are chosen as above for sufficiently large absolute constants $C_1, ~C_2,$ and $C_3$, then Theorem~\ref{thm:RachelsJL4tensors} implies that each of $(i) - (iv)$ above will fail to hold with probability at most $\eta / 4$.  The desired result now follows from the union bound.

The number of random bits and storage complexity follows directly form Theorem~\ref{thm:RachelsJL4tensors} after noting that each row of ${\bf R}$ in \eqref{equ:FastJLWard} is determined by $\mathcal{O}\left(\log N\right)$ bits.  The fact that ${\bf L'}_{\rm FJL} \circ L_{\rm FJL}$ can be applied to any tensor $\mathcal{Z}$ in $\mathcal{O}\left(N \log N \right)$-time again follows from the form of \eqref{equ:FastJLWard}.  Note that each $j$-mode product with ${\bf F}_j{\bf D}_j$ involves $\prod_{\ell \neq j} n_\ell$ multiplications of ${\bf F}_j{\bf D}_j$ against all the mode-$j$ fibers of the given tensor $\mathcal{Z}$, each of which can be performed in $\mathcal{O}(n_j \log(n_j))$-time using fast Fourier transform techniques (or approximated even more quickly using sparse Fourier transform techniques if $n_j$ is itself very large -- see e.g. \cite{gilbert2014recent,merhi2019new,bittens2019deterministic,iwen2010combinatorial,iwen2013improved,segal2013improved}).  The required vectorization and applications of ${\bf R}$ can then be performed in just $\mathcal{O}(N)$-time thereafter.  Finally, Fourier transform techniques can again be used to also apply ${\bf L'}_{\rm FJL}$ in $\mathcal{O}(m_1 \log m_1)$-time.
\end{proof}

\begin{remark}
To recap, in Sections~\ref{sec:GenModewiseLSsec}~and~\ref{sec:fastFembed} we presented two different results concerning modewise oblivious JL emdeddings for low-rank tensors subspaces, Theorem~\ref{thm:RachelNewSubspaceEmbed} and Theorem~\ref{thm:GenmodewiseConstruction}. Unlike Theorem~\ref{cor:MainOblEmb}, they are both suited for tensor low-rank fitting applications since they allow for an affine shift of an arbitrary low-rank tensor subspace $\mathcal{L}$ by an arbitrary (and not necessarily low-rank) fixed tensor $\mathcal{X}$. 

Fix $d, n, N$ and $\eta$.  Recalling Remark~\ref{Rem:Repsdepthm4} we can see that the intermediate embedding dimension provided by Theorem~\ref{thm:GenmodewiseConstruction} is $\prod_{\ell = 1}^d m_\ell \leq  C_{d,\eta,n}^d r^d \eps^{-2d}$.  In comparison we can see that Theorem~\ref{thm:RachelNewSubspaceEmbed} achieves an intermediate embedding dimension of size 
$$m_1 \leq C_{d,\eta,N}^d \left(\frac{r}{\epsilon}\right)^2 \cdot \log^4 \left( \frac{C_{d,\eta,N}}{\epsilon} \right).$$
Hence, Theorem~\ref{thm:RachelNewSubspaceEmbed} provides a significantly better intermediate embedding dimension for large $d$ (with respect to $r$ and $\epsilon$ dependence) than Theorem~\ref{thm:GenmodewiseConstruction} does despite the fact that both theorems ultimately achieve a near-optimal final embedding dimension.  Ultimately, this means that Theorem~\ref{thm:RachelNewSubspaceEmbed} provides more compactly storable multistage JL embeddings when $d$ is large than Theorem~\ref{thm:GenmodewiseConstruction} does.  Additionally, Theorem~\ref{thm:RachelNewSubspaceEmbed} does not require the basis tensors of any low-rank subspace to which it is applied to all be rank-one tensors, an advantage which is not employed in the framework of tensor low-rank fitting problems, but which might be useful in other settings.

On the other hand, Theorem~\ref{thm:GenmodewiseConstruction} is significantly more general for tensor subspaces with rank-one bases that have low modewise coherence:  it guarantees JL embedding properties for modewise products by any matrices from a large class of almost optimal JL embedding matrices including, e.g., sparse JL embedding matrices.  In contrast, Theorem~\ref{thm:RachelNewSubspaceEmbed} relies on a very particular modewise operation based on Discrete Fourier Transform (DFT) matrices.  \end{remark}

We are now prepared to consider the numerical performance of such modewise JL transforms.

\section{Experiments} \label{sec:exp}

In this section it is shown that the norms of several different types of (approximately) low-rank data can be preserved using JL embeddings, and trial least squares experiments with compressed tensor data are also performed to show the effect of these embeddings on solutions to least squares problems. All experiments were carried out in MATLAB.
The data sets used in the experiments consist of
\begin{enumerate}
\item {\it MRI data}: This data set contains three $3$-mode MRI images of size $240 \times 240 \times 155$ \cite{ADNI}.
\item {\it Randomly generated data}: This data set contains $10$ rank-$10$ $4$-mode tensors. Each test tensor is a $100\times 100 \times 100\times 100$ tensor that is created by adding $10$ randomly generated rank-$1$ tensors. More specifically, each rank-$10$ tensor is generated according to $$\mathcal{X}^{(m)}=\sum\limits_{k=1}^{r}\bigcirc_{j=1}^{d}\mathbf{y}_{k}^{(j)},$$
where $m \in [10]$, $r=10$, $d=4$ and $\mathbf{y}_{k}^{(j)} \in \mathbb{R}^{100}$. In the Gaussian case, each entry of $\mathbf{y}_{k}^{(j)}$ is drawn independently from the standard Gaussian distribution $\mathcal{N}\left( 0,1 \right)$. In the case of coherent data, low-variance Gaussian noise is added to a constant, i.e., each entry $\mathbf{y}_{k,\ell}^{(j)}$ of $\mathbf{y}_{k}^{(j)}$ is set as $1 + \sigma g_{k,\ell}^{(j)}$ with $g_{k,\ell}^{(j)}$ being an i.i.d. standard Gaussian random variable defined above, and $\sigma^2$ denoting the desired variance. In the experiments of this section, $\sigma=\sqrt{0.1}$ is used. In both cases, the $2$-norm of $\mathbf{y}_{k}^{(j)}$ is also normalized to $1$.

The reason for running experiments on both Gaussian and coherent data is to show that although coherence requirements presented in section \ref{sec:main} are used to help get general theoretical results for a large class of modewise JL embeddings, they do not seem to be necessary in practice.
\end{enumerate}

When JL embeddings are applied, experiments are performed using Gaussian JL matrices as well as Fast JL matrices. For Gaussian JL, $\mathbf{A}_j=\frac{1}{\sqrt{m}}\mathbf{G}$ is used for all $j \in [d]$, where $m$ is the target dimension and each entry in $\mathbf{G}$ is an i.i.d. standard Gaussian random variable $\mathbf{G}_{i,j}\sim \mathcal{N}\left( 0,1 \right)$. For Fast JL,  $\mathbf{A}_j=\frac{1}{\sqrt{m}}\mathbf{R}\mathbf{F}\mathbf{D}$ is used for all $j \in [d]$, where $\mathbf{R}$ denotes the random restriction matrix, $\mathbf{F}$ is the unitary DFT matrix scaled by $\sqrt{n_j}$,\footnote{Recall that $n_j$ is the size of the mode-$j$ fibers of the input tensor.} and $\mathbf{D}$ is a diagonal matrix with Rademacher random variables forming its diagonal \cite{krahmer2011new}. The embedded version of a test tensor $\mathcal{X}$ is always denoted by $L\left( \mathcal{X} \right)$, and is calculated by

\begin{equation}
L\left( \mathcal{X} \right) =
\left\{
\begin{array}{ll}
\mathcal{X}\times_1 \mathbf{A}_1\times \dots \times_d \mathbf{A}_d,  & \mbox{$1$-stage JL} \\\\
\mathbf{A}\left(\text{vect}\left(\mathcal{X}\times_1 \mathbf{A}_1\times \dots \times_d \mathbf{A}_d\right)\right), & \mbox{$2$-stage JL}
\end{array}
\right.
\label{equ:Exp_L(X)}
\end{equation}

\noindent where $\mathbf{A}$ is a JL matrix used in the $2^{\rm nd}$ stage. Obviously, $L\left( \mathcal{X} \right)$ is a vector in the $2$-stage case.

\subsection{Effect of JL Embeddings on Norm} \label{subsec:JL_exp}
In this section, numerical results have been presented, showing the effect of mode-wise JL embedding on the norm of $3$ MRI $3$-mode images treated as generic tensors, as well as randomly generated data.

The compression ratio for the $j^{\rm th}$ mode, denoted by $c_1^{(j)}$, is defined as the compression in the size of each of the mode-$j$ fibers, i.e.,
\begin{equation}\nonumber
c_1^{(j)}=\frac{m_j}{n_j}.
\label{equ:c_1j}
\end{equation}

The target dimension $m_j$ in JL matrices is chosen as $m_j=\left \lceil c_1 n_j\right \rceil$ for all $j \in [d]$, to ensure that {\it at least} a fraction $c_1$ of the ambient dimension in each mode is preserved. In the experiments, the compression ratio is set to be the same for all modes, i.e., $c_1^{(j)}=c_1$ for all $j \in [d]$. In the case of a $2$-stage JL embedding, the target dimension $m$ of the secondary JL embedding is chosen as
$$m=\left \lceil c_2 N\right \rceil,$$
where $c_2$ is the compression ratio in the $2^{\rm nd}$ stage, and $N$ is the length of the vectorized projected tensor after the modewise JL embedding. The total achieved compression is calculated by $c_{tot}=c_2\left(\prod_{j=1}^d c_1^{(j)}\right)$. When the $2^{\rm nd}$ stage embedding is skipped, $c_{tot}=\prod_{j=1}^d c_1^{(j)}$. In all experiments of \S \ref{sec:exp}, when a $2$-stage embedding is performed, $c_2=0.05$. Also, in figure legends, when two JL types are listed together, the first and second terms refer to the first and second stages, respectively. For example, in `Gaussian$+$RFD', Gaussian and RFD JL embeddings were used in the first and second stages, respectively. The term `vec' in the legends refers to vectorizing the data.

Assuming $\mathcal{X}$ denotes the original tensor and $L\left( \mathcal{X} \right)$ is the projected result, the relative norm of $\mathcal{X}$ is defined by
\begin{equation}\nonumber
c_{n,\mathcal{X}}=\frac{\| L\left( \mathcal{X} \right) \|}{\| \mathcal{X} \|}.
\label{equ:rel_norm}
\end{equation}
The results of this section depict the interplay between $c_{n,\mathcal{X}}$ and $c_1$ for randomly generated data, and $c_{n,\mathcal{X}}$ versus $c_{tot}$ for MRI data, where the numbers have been averaged over $1000$ trials, as well as over all samples for each value of $c_1$ or $c_{tot}$. In the case of Figure \ref{fig:JL_norm_synth}, $1000$ randomly generated JL matrices were applied to each mode of all $10$ randomly generated tensors. The results there indicate that the modewise embedding methods proposed herein still work on relatively coherent data despite the incoherence assumptions utilized in their theoretical analysis (recall Section~\ref{sec:main}).  In Figure \ref{fig:JL_norm_real}, $1000$ JL embedding choices have been averaged over each of the $3$ MRI images as well as the $3$ images themselves. As expected, it can be observed in both figures that increasing the compression ratio leads to better norm (and distance) preservation.

The MRI data experiments were done using various combinations of JL matrices in the first and second stages, and were compared with the $1$-stage (modewise) case and also JL applied to vectorized data. In Figure \ref{fig:JL_norm_real}\subref{fig:JL_norm_real_b}, the runtime plots show that vectorizing the data before applying JL embeddings is the most computationally intensive way of compressing the data, although it preserves norms the best, as Figure \ref{fig:JL_norm_real}\subref{fig:JL_norm_real_a} demonstrates. Due to the small mode sizes of the MRI data used in the experiments, modewise fast JL does not outperform modewise Gaussian JL in terms of computational efficiency in the modewise embeddings as one might initially expect (see the red and blue curves).  This is likely due to the fact that the individual mode sizes are too small to benefit from the FFT (recall all modes are $\leq 240$ in size), together with the need of Fourier methods to use less efficient complex number arithmetic. However, when the $2$-stage JL is employed for larger compression ratios, the vectorized data after the first stage compression is large enough to make the efficiency of fast JL over Gaussian JL embeddings clear (compare, e.g., the yellow and purple curves). 

\begin{figure}
	\centering
	\begin{subfigure}{0.45\textwidth}
		\centering
		\includegraphics[width=1\linewidth]{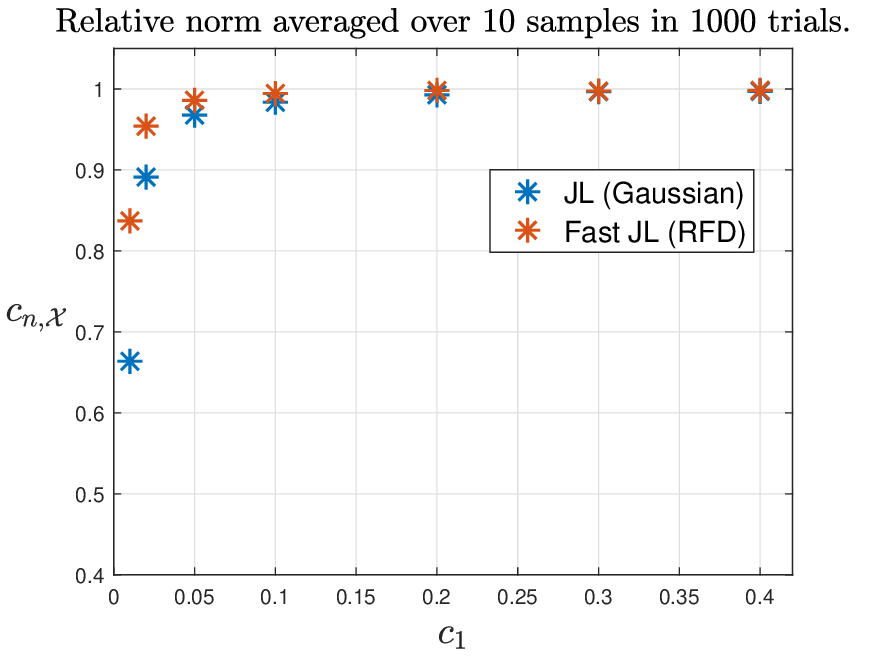}
		\caption{}
	\end{subfigure}
	\begin{subfigure}{0.45\textwidth}
		\centering
		\includegraphics[width=1\linewidth]{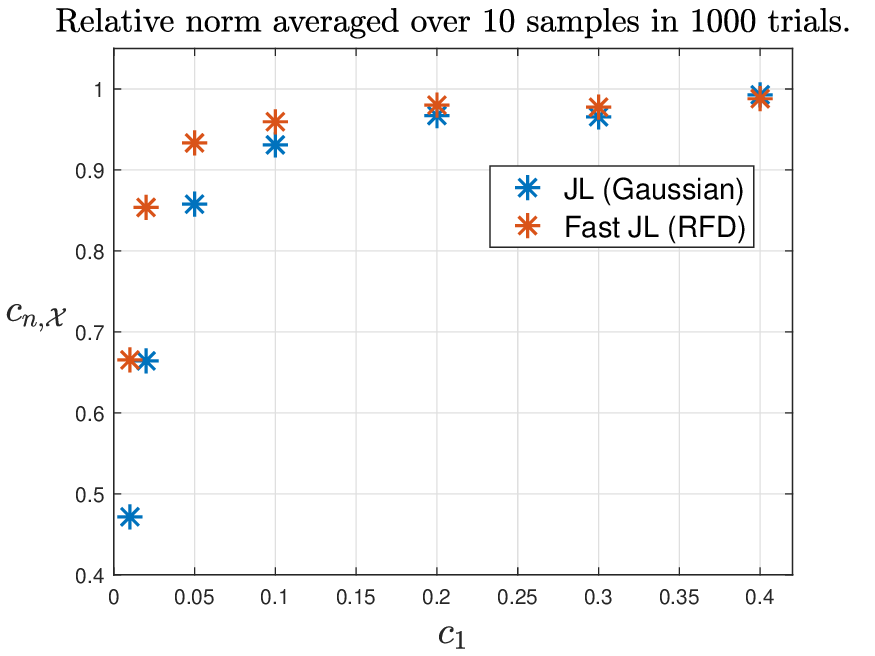}
		\caption{}
	\end{subfigure}
	
	\caption{Relative norm of randomly generated  $4$-dimensional data. Here, the total compression will be $c_{tot}=c_1^4$. (a) Gaussian data. (b) Coherent data.  Note that the modewise approach still preserves norms well for the coherent data indicating that the incoherence assumptions utilized in \S \ref{sec:main} can likely be relaxed.}
	\label{fig:JL_norm_synth}
\end{figure}

\begin{figure}
	\centering
	\begin{subfigure}{0.48\textwidth}
		\centering
		\includegraphics[width=1\linewidth]{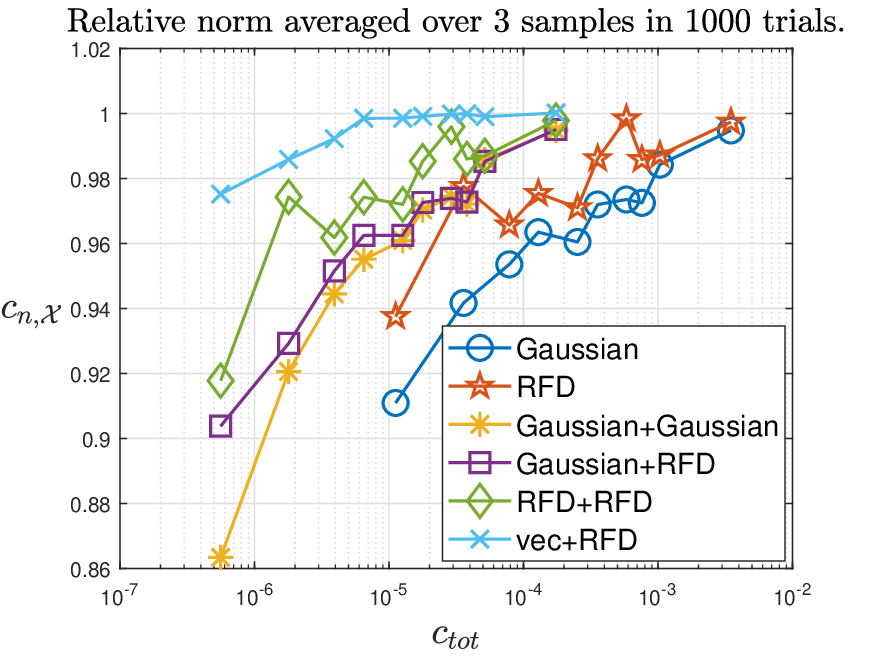}
		\caption{}
		\label{fig:JL_norm_real_a}
	\end{subfigure}
	\begin{subfigure}{0.48\textwidth}
		\centering
		\includegraphics[width=1\linewidth]{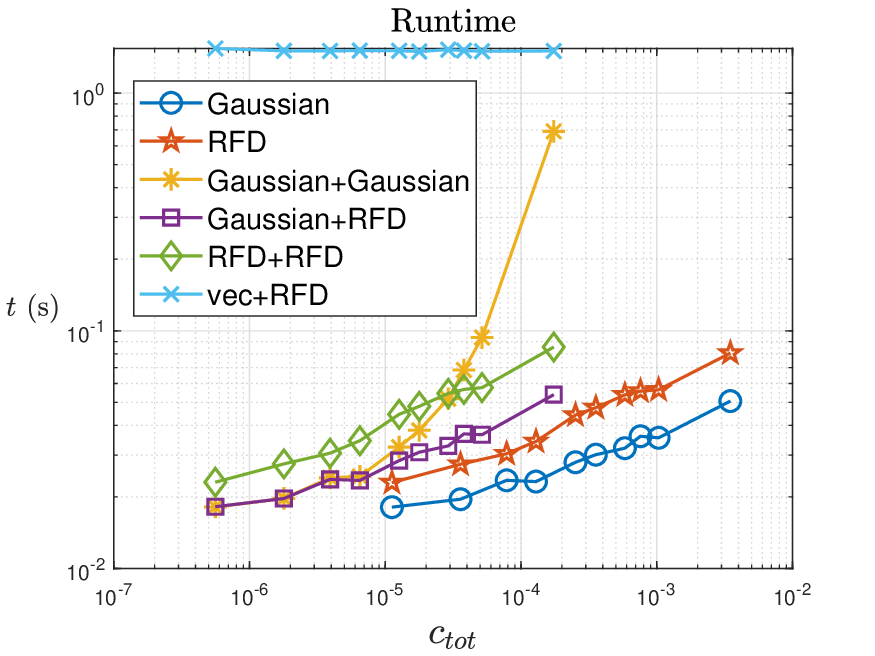}
		\caption{}
		\label{fig:JL_norm_real_b}
	\end{subfigure}
	
	\caption{Simulation results averaged over $1000$ trials for $3$ MRI data samples, where each sample is $3$-dimensional. In the $2$-stage cases, $c_2=0.05$ has been used. (a) Relative norm. (b) Runtime.}
	\label{fig:JL_norm_real}
\end{figure}

\subsection{Effect of JL Embeddings on Least Squares Solutions} \label{subsec:JL_LS}
In this section, the first sample of the three MRI data samples is used in the experiments. First, it is shown that this MRI sample has a relatively low-rank CP representations by plotting its CP reconstruction error for various values of rank. Next, the effect of modewise JL on least squares solutions is investigated by solving for the coefficients of the CP decomposition of the MRI sample in a least squares problem. This will be done by performing $1$-stage (modewise) and $2$-stage JL on the data, which we call compressed least squares, and will be compared with the case where a regular uncompressed least squares problem is solved instead.

\subsubsection{CPD Reconstruction}\label{CPD reconst} \label{subsec:cpd_reconst_exp}
Before the experimental results, a short description of the basic form of CPD calculation is presented as well as how the number of rank-$1$ tensors, $r$, is chosen. Given a tensor $\mathcal{X}$, assume $r$ is known beforehand. The problem is now the calculation of $\mathbf{y}_k^{(j)}$ for $j \in [d]$ and $k \in [r]$ and $\boldsymbol{\alpha}$ in \eqref{equ:rankrTensor}, i.e. the solution to
\begin{equation}
\min\limits_{\hat{\mathcal{X}}}\| \mathcal{X} - \hat{\mathcal{X}} \| \text{ with $\hat{\mathcal{X}} = \sum\limits_{k=1}^{r}\alpha_k~ \mathbf{y}^{(1)}_{k} \bigcirc \mathbf{y}^{(2)}_{k} \bigcirc \dots \bigcirc \mathbf{y}^{(d)}_{k}$}.
\label{equ:CP_optim}
\end{equation}

As the Euclidean norm a $d$-mode tensor is equal to the Frobenius norm of its mode-$j$ unfoldings for $j \in [d]$, by letting $\mathbf{y}^{(j)}_{k}$ be the $k^{\rm th}$ column of a matrix $\mathbf{Y}^{(j)} \in \mathbbm{C}^{n_j \times r}$, the above minimization problem can be written as
\begin{equation}\nonumber
\begin{split}
\min\limits_{\hat{\mathbf{Y}}^{(j)}}\left\| \mathbf{X}_{(j)}-\hat{\mathbf{Y}}^{(j)}\left( \mathbf{Y}^{(d)}\odot \dots \odot \mathbf{Y}^{(j+1)}\odot \mathbf{Y}^{(j-1)}\odot \dots \odot \mathbf{Y}^{(1)}  \right)^\top \right\|_{\rm F}
\end{split}
\label{}
\end{equation} 
\noindent where $\hat{\mathbf{Y}}^{(j)}=\mathbf{Y}^{(j)}\text{diag}\left( \boldsymbol{\alpha} \right)$, and $\odot$ denotes the Khatri-Rao product defined as the columnwise matching Kronecker product. The operator diag$\left( \cdot \right)$ creates a diagonal matrix with $\boldsymbol{\alpha}$ as its diagonal.
Once solved for, the columns of $\hat{\mathbf{Y}}^{(j)}$ can then be normalized and used to form the coefficients $\alpha_k=\prod_{j=1}^{d}\| \hat{{\bf y}}^{(j)}_k \|_2$ for $k \in [r]$, although this is optional, i.e., if the columns are not normalized, the coefficients $\alpha_k$ in the factorization will all be ones. This procedure is repeated iteratively until the fit ceases to improve (the objective function stops improving with respect to a tolerance) or the maximum number of iterations are exhausted. This procedure is known as CPD-ALS\footnote{Alternating Least Squares} \cite{kolda2009tensor}.
To choose the rank of the decomposition as well as obtaining the best estimates for $\mathbf{Y}^{(j)}$, a commonly used consistency diagnostic called CORCONDIA\footnote{CORe CONsistency DIAgnostic} can be employed \cite{bro2003new}.

In the remainder of this section, the relative reconstruction error of CPD is calculated and plotted for various values of rank $r$. Assuming $\mathcal{X}$ represents the data, this error is defined as
\begin{equation}\nonumber
e_{cpd}=\frac{\| \mathcal{X}-\hat{\mathcal{X}} \|}{\| \mathcal{X} \|},
\label{equ:rel_error_rec}
\end{equation}
\noindent where $\hat{\mathcal{X}}$ denotes the reconstruction of $\mathcal{X}$. Figure \ref{fig:cpd_rec} displays the results.

\begin{figure}
 \centering
\includegraphics[width=0.45\linewidth]{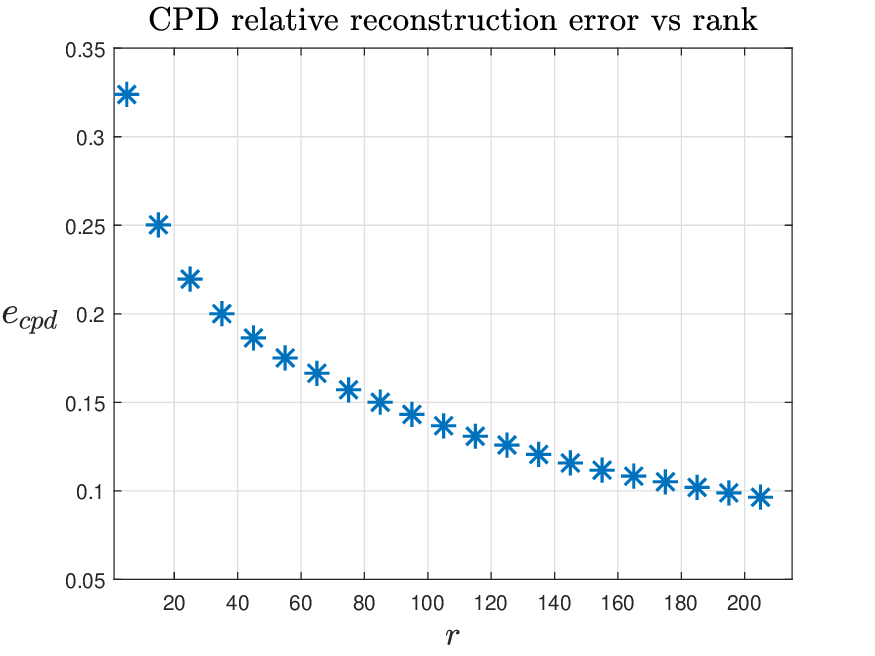}
\caption{Relative reconstruction error of CPD calculated for different values of rank $r$ for MRI data. As the rank increases, the error becomes smaller.}
 \label{fig:cpd_rec}
\end{figure}

\subsubsection{Compressed Least Squares Performance}Let ${\bf y}^{(j)}_k$ be known in $$\mathcal{X} \approx \sum_{k=1}^r \alpha_k \bigcirc^d_{j = 1} {\bf y}^{(j)}_k,$$ for $k \in [r]$ and $j \in [d]$. They can be obtained from a previous iteration in the CPD fitting procedure. Here, they come from the CPD of the data calculated in section \ref{CPD reconst}. Also, assume these vectors have unit norms. In general, as stated in section \ref{CPD reconst}, when ${\bf y}^{(j)}_k$ are obtained using a CPD algorithm, they do not necessarily have unit norms. Therefore, they are normalized and the norms are absorbed into the coefficients of CPD. In other words, $\alpha_k=\prod_{j=1}^{d}\| {\bf y}^{(j)}_k \|_2$ for $k \in [r]$. If the normalization of the vectors is not performed, $\alpha_k=1$ for $k \in [r]$. The coefficients of the CPD fit are the solutions to the following least squares problem,
\begin{equation}\nonumber
\boldsymbol{\alpha}=\argmin_{\boldsymbol{\beta}} \left\| \mathcal{X} - \sum_{k=1}^r \beta_k \bigcirc^d_{j = 1} {\bf y}^{(j)}_k \right \|.
\label{equ:LS_estim}
\end{equation}
As normalization of ${\bf y}^{(j)}_k$ was not performed when computing the CPD of the data in these experiments, the true solution will be $\boldsymbol{\alpha}={\bf 1}$. An approximate solution for the coefficients can be obtained by solving for
\begin{equation}\nonumber
\boldsymbol{\alpha}_P=\argmin_{\boldsymbol{\beta}} \left\| L\left(\mathcal{X}\right) - L\left(\sum_{k=1}^r \beta_k \bigcirc^d_{j = 1} {\bf y}^{(j)}_k \right) \right \|,
\label{equ:LS_estim_proj}
\end{equation}

\noindent where $\boldsymbol{\alpha}_P$ is the vector $\boldsymbol{\alpha}$ estimated for randomly projected data, and $L\left(\mathcal{X}\right)$ is defined as per \eqref{equ:Exp_L(X)}. This is in fact simply another way of demonstrating that solving \eqref{equ:compDecoupledALS} yields an approximate solution to \eqref{equ:DecoupledALS} for a ($d-1$)-mode tensor. Of course, both of these problems can be solved using the vectorized versions of the tensors instead. Indeed, for $\boldsymbol{\alpha}_P$, vectorization should be done after random projection of $\mathcal{X}$ and the rank-$1$ tensors, i.e.,
\begin{equation}\nonumber
\boldsymbol{\alpha}_P=\argmin_{\boldsymbol{\beta}} \left\| \mathbf{x}_P - \mathbf{B}\boldsymbol{\beta} \right \|_2=\left( \mathbf{B}^\ast \mathbf{B} \right)^{-1}\mathbf{B}^\ast \mathbf{x}_P,
\label{equ:LS_estim_proj_vect}
\end{equation}
\noindent where $\mathbf{x}_P=\text{vect}\left(L\left(\mathcal{X}\right)\right)$, and $\mathbf{B}$ is a matrix whose $k^{\rm th}$ column is $\text{vect} \left( L \left( \bigcirc^d_{j = 1}{\bf y}^{(j)}_k \right) \right)$\footnote{Again, it is clear that in the $2$-stage case, $L\left( \mathcal{X} \right)$ and $L \left( \bigcirc^d_{j = 1}{\bf y}^{(j)}_k \right)$ are vectors, and therefore, the operator $\text{vect}\left( \cdot \right)$ does not change the result.} for $k \in [r].$\footnote{The backslash operator was used to actually solve the resulting least squares problems in MATLAB.}
The error measure used to evaluate the approximate solution is defined as $$e_r=\left| \frac{e_P - e_T}{e_T} \right|,$$
where $e_T=\left\| \mathcal{X} - \sum_{k=1}^r \alpha_k \bigcirc^d_{j = 1} {\bf y}^{(j)}_k \right \|$ and $e_P=\left\| \mathcal{X} - \sum_{k=1}^r \alpha_{P,k} \bigcirc^d_{j = 1} {\bf y}^{(j)}_k \right \|$. This in fact compares the true CPD reconstruction error and the reconstruction error calculated using the approximate solution for the CPD coefficients $\boldsymbol{\alpha}_P$. The results are shown in Figure \ref{fig:LS_alpha}.

\begin{figure}
	\centering
	\begin{subfigure}{0.48\textwidth}
		\centering
		\includegraphics[width=1\linewidth]{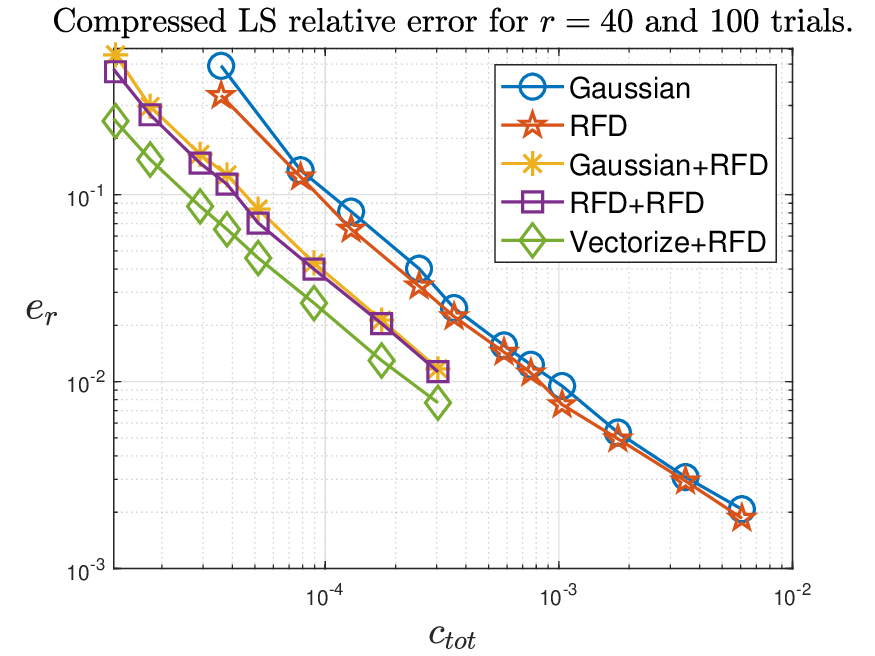}
		\caption{}
	\end{subfigure}
	\begin{subfigure}{0.48\textwidth}
		\centering
		\includegraphics[width=1\linewidth]{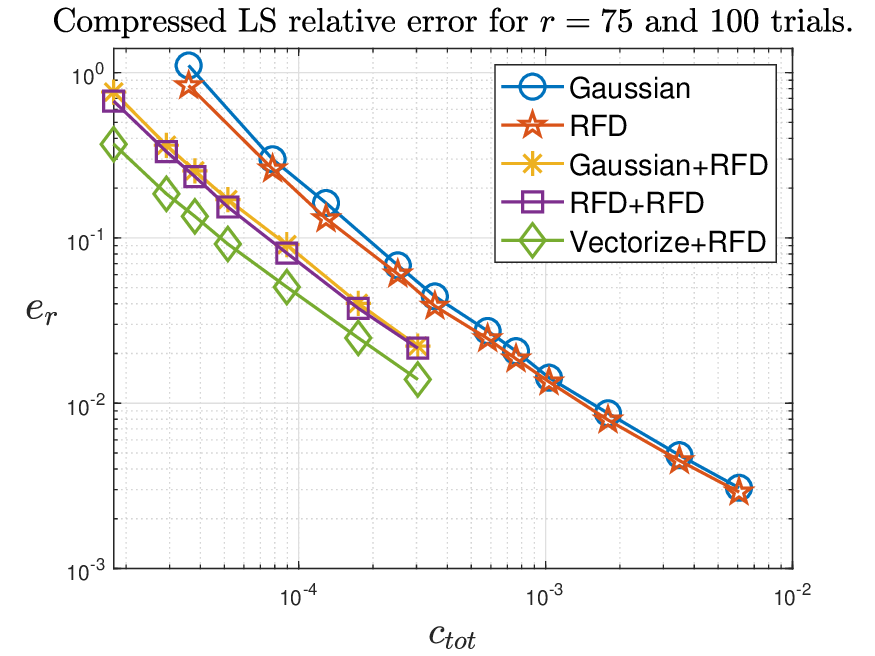}
		\caption{}
	\end{subfigure}
	\begin{subfigure}{0.48\textwidth}
		\centering
		\includegraphics[width=1\linewidth]{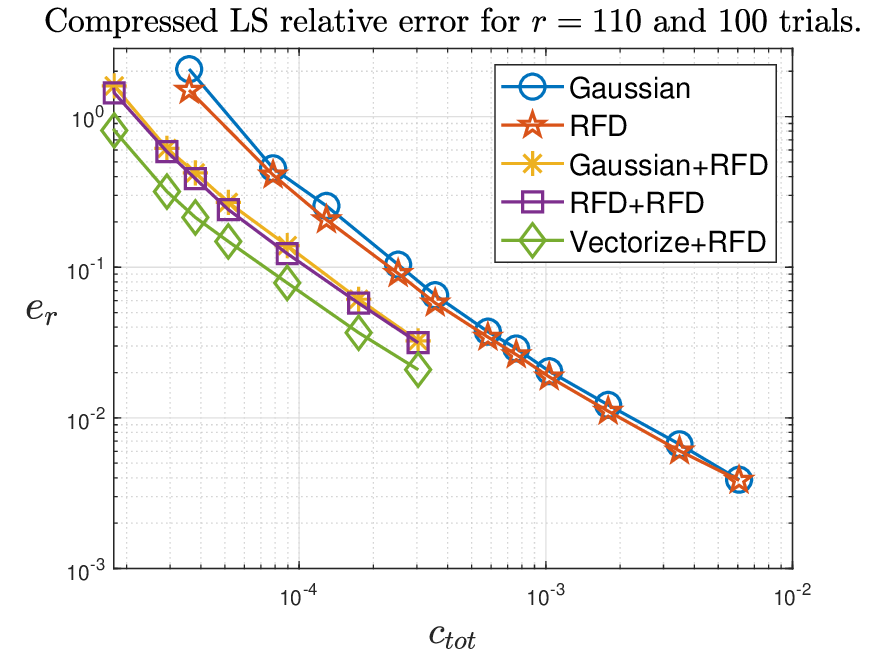}
		\caption{}
	\end{subfigure}
	\begin{subfigure}{0.48\textwidth}
		\centering
		\includegraphics[width=1\linewidth]{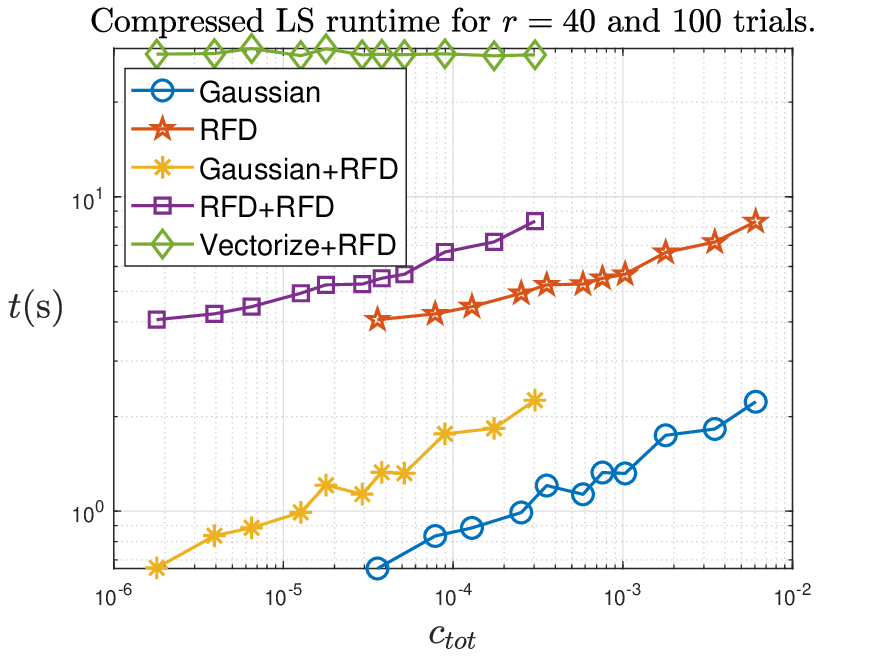}
		\caption{}
	\end{subfigure}
	
	\caption{Effect of JL embeddings on the relative reconstruction error of least squares estimation of CPD coefficients. In the $2$-stage cases, $c_2=0.05$ has been used. (a) $r=40$. (b) $r=75$. (c) $r=110$. (d) Average runtime for $r=40$.  The other runtime plots for $r = 75$ and $r = 110$ are qualitatively identical.}
	\label{fig:LS_alpha}
\end{figure}

\section{Conclusion}\label{sec:conc}
We have proposed general modewise Johnson-Lindenstrauss (JL) subspace embeddings that are faster to generate and significantly smaller to store than traditional JL embeddings especially for tensors in very large dimensions.  We provided a subspace embedding result with improved space complexity bounds for embeddings of rank-$r$ tensors in the setting of unknown basis tensors. This result also has applications in the vector setting, leading to general near-optimal oblivious subspace embedding constructions that require fewer random bits for subspaces spanned by basis vectors having special Kronecker structure.  
We also provided new fast JL embeddings for arbitrary $r$-dimensional subspaces using fewer random bits than standard  methods. We showcased these results for applications including compressed least squares and fitting low-rank CP decompositions, while also confirming our results experimentally.
There are several interesting future directions including the analysis of other randomly constructed embeddings, the construction of embeddings designed to maintain other types of structures (such as properties of the core tensor), and their effectiveness in reconstruction and inference tasks.

\section*{Acknowledgments} M. Iwen was supported in part\footnote{ 
Mark would also like to thank E.I. and D. M. for greatly accentuating his UCLA visit by squatting at his Airbnb Oct. 15 -- 19, 2019, as well as a to commit a written act of dogeza to his near-optimal wife for agreeing to his being over 2000 miles away during E's witching months. Mark also sends many thanks to E. S. for helping out with the baby in his place during his absence.} by NSF DMS 1912706 and NSF CCF 1615489, Deanna Needell and Elizaveta Rebrova by NSF CAREER DMS 1348721
and NSF BIGDATA 1740325, and Ali Zare by NSF CCF 1615489. Elizaveta Rebrova also acknowledges sponsorship by Capital Fund Management.

\bibliographystyle{abbrv}
\bibliography{references2}

\appendix
\section{Proofs of the Tensor properties and JL Results from Section~\ref{sec:background}} 
\label{AppPreliminaryProofs}
In this section, we give the proofs of the Lemmas \ref{lem:ProdProps}, \ref{lem:InnProdJL} and \ref{lem:simplenetsubspace}. The first result lists classical tensor properties we constantly rely on in this paper.

\begin{proof}[Proof of Lemma~\ref{lem:ProdProps}]
The first property follows from the fact that
\begin{align*}
\left(\left(\alpha \mathcal{A} + \beta \mathcal{B} \right) \bigcirc \mathcal{C} \right)_{i_1, \dots, i_d, i'_1, \dots, i'_{d'}} ~=~ \left(\alpha \mathcal{A} + \beta \mathcal{B} \right)_{i_1, \dots, i_d}\mathcal{C}_{i'_1, \dots, i'_{d'}}~=~\left(\alpha \mathcal{A}_{i_1, \dots, i_d} + \beta \mathcal{B}_{i_1, \dots, i_d}\right)\mathcal{C}_{i'_1, \dots, i'_{d'}}.
\end{align*}
To establish the property $(ii)$ we note that
\begin{align*}
\left\langle \mathcal{A} \bigcirc \mathcal{C}, \mathcal{B} \bigcirc \mathcal{D} \right\rangle &=  \sum_{i_1=1}^{n_1} \dots\sum_{i_{d}=1}^{n_{d}} {\sum_{i'_1=1}^{n'_1} \dots \sum_{i'_{d}=1}^{n'_{d'}} \mathcal{A}_{i_1,i_2,...,i_{d}} \mathcal{C}_{i'_1, \dots, i'_{d'}} \overline{\mathcal{B}_{i_1,i_2,...,i_d}}~ \overline{\mathcal{D}_{i'_1,...,i'_{d'}}}}\\
&= \left( \sum_{i_1=1}^{n_1} \dots\sum_{i_{d}=1}^{n_{d}} \mathcal{A}_{i_1,i_2,...,i_d} \overline{\mathcal{B}_{i_1,i_2,...,i_d}}  \right) \left( \sum_{i'_1=1}^{n'_1} \dots \sum_{i'_{d}=1}^{n'_{d'}}  \mathcal{C}_{i'_1, \dots, i'_{d'}}  \overline{\mathcal{D}_{i'_1, \dots, i'_{d'}}} \right)\\
&= \left\langle \mathcal{A}, \mathcal{B}\right\rangle \left\langle \mathcal{C}, \mathcal{D}\right\rangle.
\end{align*}
The facts $(iii), (iv)$ and $(vi)$ are easily established using mode-$j$ unfoldings formula~\eqref{unfolding_via_product}.  To establish $(iii)$, we note that 
\begin{align*}
\left(\left( \alpha \mathcal{A} + \beta \mathcal{B} \right) \times_j \mathbf{U}_j \right)_{(j)} &= \mathbf{U}_j \left( \alpha \mathcal{A} + \beta \mathcal{B} \right)_{(j)} = \mathbf{U}_j \left( \alpha \mathbf{A}_{(j)} + \beta \mathbf{B}_{(j)} \right)\\ &= \alpha \mathbf{U}_j \mathbf{A}_{(j)} + \beta \mathbf{U}_j \mathbf{B}_{(j)} = \alpha (\mathcal{A} \times_j \mathbf{U}_j)_{(j)} + \beta (\mathcal{B} \times_j \mathbf{U}_j)_{(j)}.
\end{align*}
Reshaping both sides of the derived equality back into their original tensor forms now completes the proof.\footnote{Here we are implicitly using that mode-$j$ unfolding provides a vector space isomorphism between $\mathbbm{C}^{n_1 \times n_2 \times \dots \times n_d}$ and $\mathbbm{C}^{n_j \times \prod_{\ell \in [d] \setminus \{ j \}} n_\ell}$ for all $j \in [d]$.}  The proof of $(iv)$ using unfoldings is nearly identical.  To prove $(vi)$ we may again use mode-$j$ unfoldings to see that
\begin{align*}
\left(\mathcal{A} \times_j \mathbf{U}_j \times_j \mathbf{W} \right)_{(j)} &= \mathbf{W} \left( \mathcal{A} \times_j \mathbf{U}_j \right)_{(j)} = \mathbf{W} \mathbf{U}_j \mathbf{A}_{(j)} = \left( \mathcal{A} \times_j \mathbf{W}\mathbf{U}_j \right)_{(j)}.
\end{align*}
Reshaping these expressions back into their original tensor forms again completes the proof.
To prove $(v)$, it is perhaps easiest to appeal directly to the component-wise definition of the mode-$j$ product given in equation \eqref{equ:DefModejProduct}.  Suppose that $\ell > j$ (the case $\ell < j$ is nearly identical).  Set $\mathbf{U} := \mathbf{U}_j$ and $\mathbf{V} := \mathbf{V}_\ell$ to simplify subscript notation.  We have for all $k \in [m_j]$, $l \in [m_\ell]$, and $i_q \in [n_q]$ with $q \notin \{ j,\ell \}$ that
\begin{align*}
\left( \left(\mathcal{A} \times_j \mathbf{U} \right) \times_\ell \mathbf{V} \right)_{i_1, \dots, i_{j-1}, k, i_{j+1}, \dots, i_{\ell-1}, l, i_{\ell+1}, \dots, i_d} &= \sum_{i_{\ell}=1}^{n_\ell} \left(\mathcal{A} \times_j \mathbf{U} \right)_{i_{1},\dots,i_{j-1}, k, i_{j+1},\dots,i_{\ell},\dots,i_{d}} \mathbf{V}_{l,i_{\ell}} \\
&= \sum_{i_{\ell}=1}^{n_\ell} \left(\sum_{i_{j}=1}^{n_j} \mathcal{A}_{i_{1},\dots,i_{j},\dots,i_{\ell},\dots,i_{d}} \mathbf{U}_{k,i_{j}} \right) \mathbf{V}_{l,i_{\ell}} \\
&= \sum_{i_{j}=1}^{n_j} \left( \sum_{i_{\ell}=1}^{n_\ell} \mathcal{A}_{i_{1},\dots,i_{j},\dots,i_{\ell},\dots,i_{d}} \mathbf{V}_{l,i_{\ell}} \right) \mathbf{U}_{k,i_{j}} \\
&= \sum_{i_{j}=1}^{n_j} \left(\mathcal{A} \times_\ell \mathbf{V} \right)_{i_{1},\dots, i_j, \dots, i_{\ell-1}, l, i_{\ell+1}, \dots,i_{d}} \mathbf{U}_{k,i_{j}} \\
&= \left( \left(\mathcal{A} \times_\ell \mathbf{U} \right) \times_j \mathbf{U}  \right)_{i_1, \dots, i_{j-1}, k, i_{j+1}, \dots, i_{\ell-1}, l, i_{\ell+1}, \dots, i_d}.
\end{align*}
\end{proof}  

\bigskip

Our second lemma of this appendix proves that JL embeddings can also preserve the inner products between all elements of a given finite set.

\begin{proof}[Proof of Lemma~\ref{lem:InnProdJL}]
The result for vectors is a well known consequence of the polarization identity for inner products.  We have that
\begin{align*}
\left| \left\langle \mathbf{A}{\bf x},~ \mathbf{A}{\bf y} \right\rangle - \left\langle {\bf x},~ {\bf y}\right\rangle \right| &= \left| \frac{1}{4} \sum^3_{\ell = 0} \mathbbm{i}^\ell \left( \left\| \mathbf{A} {\bf x} + \mathbbm{i}^\ell \mathbf{A} {\bf y} \right\|^2_2  -  \left\| {\bf x} + \mathbbm{i}^\ell {\bf y} \right\|^2_2 \right)\right| ~=~ \left| \frac{1}{4} \sum^3_{\ell = 0} \mathbbm{i}^\ell \epsilon_\ell \left\| {\bf x} + \mathbbm{i}^\ell {\bf y} \right\|^2_2 \right|\\
&\leq \frac{1}{4} \sum^3_{\ell = 0} \epsilon \left( \| {\bf x} \|_2 + \| {\bf y} \|_2 \right)^2 ~=~\epsilon \left( \| {\bf x} \|_2 + \| {\bf y} \|_2 \right)^2 ~=~ \epsilon \left( \| {\bf x} \|_2^2 + \| {\bf y} \|_2^2 + 2\| {\bf x} \|_2\| {\bf y} \|_2 \right)\\ 
&\leq 2 \epsilon \left( \| {\bf x} \|_2^2 + \| {\bf y} \|_2^2 \right) ~\leq~ 4 \epsilon \cdot \max \left\{ \| {\bf x} \|_2^2, \| {\bf y} \|_2^2 \right\},
\end{align*}
where the second to last inequality follows from Young's inequality for products.
The proof of the tensor counterpart is essentially identifical, with $L\left( \mathcal{X} \right)$ replacing $\mathbf{A}\mathbf{x}$, and making use of the linearity of $L$.
\end{proof}

\bigskip

The next lemma is a version of classical covering estimate in high dimensional spaces.

\begin{proof}[Proof of Lemma~\ref{lem:simplenetsubspace}]
The cardinality bound on $\mathcal{C}$ can be obtained from the covering results in Appendix C of \cite{foucart2013book}.  It is enough to establish $\eqref{equ:subspacepres}$ for an arbitrary ${\bf x} \in \mathcal{S}_{\ell^2}$ due to the linearity of ${\bf A}$ and $\mathcal{L}$.  Let $\Delta := \| {\bf A} \|_{2 \rightarrow 2} \geq 0$, and choose an element ${\bf y} \in \mathcal{C}$ with $\| {\bf x} - {\bf y} \| \leq \epsilon/16$.  We have that
\begin{align*}
\| {\bf A} {\bf x}\|_2 - \| {\bf x} \|_2 ~&\leq~  \| {\bf A} {\bf y}\|_2 + \| {\bf A} ( {\bf x} - {\bf y}) \|_2 -1~\leq~\sqrt{1+\epsilon/2} - 1 + \| {\bf A} ( {\bf x} - {\bf y}) \|_2 \\ 
&\leq~ (1+\epsilon/4) - 1 + \Delta \epsilon/16  ~=~ (\epsilon/4) (1 + \Delta / 4 )
\end{align*}
holds for all ${\bf x} \in \mathcal{S}_{\ell^2}$.  This, in turn, means that the upper bound above will hold for a vector ${\bf x}$ realizing $\| {\bf A} {\bf x} \| = \| {\bf A} \|_{2 \rightarrow 2}$ so that $\Delta - 1 ~\leq~(\epsilon/4) (1 + \Delta / 4 )$ must also hold.  As a consequence, $\Delta \leq 1 + \epsilon/4 + \Delta \epsilon/16 \implies \Delta \leq \frac{1 + \epsilon/4}{1-\epsilon/16} \leq 1+\epsilon/3$.  The upper bound now follows. \\
To establish the lower bound we define $\delta := \inf_{{\bf z} \in \mathcal{S}_{\ell^2}}\| {\bf A} {\bf z} \| \geq 0$ and note that this quantity will also be realized by some element of the compact set $\mathcal{S}_{\ell^2}$.  As above we consider this minimizing vector ${\bf x} \in \mathcal{S}_{\ell^2}$ and choose an element ${\bf y} \in \mathcal{C}$ with $\| {\bf x} - {\bf y} \| \leq \epsilon/16$ in order to see that
\begin{align*}
\delta - 1 = \| {\bf A} {\bf x}\|_2 - \| {\bf x} \|_2 ~&\geq~  \| {\bf A} {\bf y}\|_2 - \| {\bf A} ( {\bf x} - {\bf y}) \|_2 -1~\geq~\sqrt{1- \epsilon/2} - 1 - \| {\bf A} ( {\bf x} - {\bf y}) \|_2 \\ 
&\geq~ (1-\epsilon/3) - 1 - \Delta \epsilon/16  ~\geq~ -\left( \epsilon/3 + \epsilon/16 \left( 1 + \epsilon/3 \right) \right) \\
&\geq~ -\left( \epsilon/3 + \epsilon/16 + \epsilon/48 \right) = -5\epsilon/12.
\end{align*}
As a consequence, $\delta \geq 1-5\epsilon/12$.  The lower bound now follows.
\end{proof}

\section{Proofs of the Intermediate Results from Sections~\ref{seq:proofof thm3}~and~\ref{sec:GenModewiseLSsec}}
\label{AppMinSuppProofs}

In this section, we give the proofs of all auxiliary results for the proof of Theorem~\ref{cor:MainOblEmb}.  All the statements are listed in Section~\ref{seq:proofof thm3}.

\begin{proof}[Proof of Lemma~\ref{lem:NewStandForm}]
Using Lemma~\ref{lem:ProdProps}, the linearity of tensor matricization, and \eqref{equ:KronModenFlat} we can see that the mode-$j$ unfolding of $\mathcal{Y}'$ satisfies
\begin{align*}
{\bf Y'}_{(j)} ~&=~  {\bf B} {\bf Y}_{(j)} ~=~  {\bf B} \sum_{k=1}^r \alpha_k \left( \bigcirc^d_{\ell = 1} {\bf y}^{(\ell)}_k \right)_{(j)} ~=~ \sum_{k=1}^r \alpha_k  {\bf B} {\bf y}^{(j)}_k \left( \otimes_{\ell \neq j} {\bf y}^{(\ell)}_k \right)^\top\\
&=~ \sum_{k=1}^r \left( \alpha_k \left\| {\bf B} {\bf y}^{(j)}_k \right\|_2 \right)  \frac{{\bf B}{\bf y}^{(j)}_k}{\left\| {\bf B} {\bf y}^{(j)}_k \right\|_2} \left( \otimes_{\ell \neq j} {\bf y}^{(\ell)}_k \right)^\top.
\end{align*}
Refolding ${\bf Y'}_{(j)}$ back into a $d$-mode tensor then gives us our first equality.  The second two equalities now follow directly from the definitions of modewise coherence.
\end{proof}

\bigskip

\begin{proof}[Proof of Lemma~\ref{lem:ModejNormexp}]
Using Lemma~\ref{lem:ProdProps}, the linearity of tensor matricization, and \eqref{equ:KronModenFlat} once again we can see that
\begin{align*}
\| \mathcal{Y} \times_j {\bf B} \|^2 &= \left\| \sum_{k=1}^r \alpha_k \left( \bigcirc^d_{\ell = 1} {\bf y}^{(\ell)}_k \times_j {\bf B} \right) \right\|^2 ~=~ \left\| \sum_{k=1}^r \alpha_k {\bf B}  {\bf y}^{(j)}_k \left( \otimes_{\ell \neq j} {\bf y}^{(\ell)}_k \right)^\top \right\|^2_{\rm F} \\ &= \sum_{k,h=1}^r \left \langle \alpha_k {\bf B}  {\bf y}^{(j)}_k \left( \otimes_{\ell \neq j} {\bf y}^{(\ell)}_k \right)^\top, \alpha_h {\bf B}  {\bf y}^{(j)}_h \left( \otimes_{\ell \neq j} {\bf y}^{(\ell)}_h \right)^\top\right \rangle_{\rm F}
\end{align*}
where $\| \cdot \|_{\rm F}$ and $\left \langle \cdot, \cdot \right\rangle_{\rm F}$ denote the Frobenius matrix norm and inner product, respectively.  Computing the Frobenius inner products above columnwise by expressing each ${\bf B}  {\bf y}^{(j)}_k \left( \otimes_{\ell \neq j} {\bf y}^{(\ell)}_k \right)^\top$ as a sum of its individual columns (each represented as a matrix with only one nonzero column) we can further see that
\begin{equation*}
\| \mathcal{Y} \times_j {\bf B} \|^2 ~=~ \sum_{k,h=1}^r \sum^{\prod_{\ell \neq j} n_\ell}_{a = 1} \alpha_k \left( \otimes_{\ell \neq j} {\bf y}^{(\ell)}_k \right)_a \overline{\alpha_h \left( \otimes_{\ell \neq j} {\bf y}^{(\ell)}_h \right)_a} \left \langle  {\bf B}  {\bf y}^{(j)}_k,  {\bf B} {\bf y}^{(j)}_h \right \rangle.
\end{equation*}
as we wished to show.
\end{proof}

\bigskip

\begin{proof}[Proof of Proposition~\ref{thm:ForInduction}] We prove each property in order below.\\

\noindent \underline{\it Proof of $\mathbf (\boldsymbol{\dagger})$:} By Lemma~\ref{lem:NewStandForm} we have for all $k \in [r]$ that
$$\left| \alpha'_k - \alpha_k \right| ~=~ \left| \alpha_k{\left\| {\bf A} {\bf y}^{(j)}_k \right\|_2} - \alpha_k \right| ~=~  \left| {\left\| {\bf A} {\bf y}^{(j)}_k \right\|_2} - 1 \right| | \alpha_k | ~\leq~ \epsilon |\alpha_k| / 4$$
as we wished to prove.\\

\noindent \underline{\it Proof of $\mathbf (\boldsymbol{\dagger \dagger})$:} Appealing to Lemma~\ref{lem:NewStandForm} and the definition of $j$-mode coherence we have that 
$$\mu_{\mathcal{Y}',j} = \max_{\substack{k, h \in [r]\\ k \neq h}} \frac{\left| \left \langle {\bf A} {\bf y}^{(j)}_k , {\bf A} {\bf y}^{(j)}_h \right \rangle \right|}{\left\| {\bf A} {\bf y}^{(j)}_k \right\|_2 \left\| {\bf A} {\bf y}^{(j)}_h \right\|_2} ~\leq~ \max_{\substack{k, h \in [r]\\ k \neq h}} \frac{\left| \left \langle {\bf y}^{(j)}_k , {\bf y}^{(j)}_h \right \rangle \right| + \epsilon}{1 - \frac{\epsilon}{4}} = \frac{ \mu_{\mathcal{Y},j} + \epsilon}{1 - \frac{\epsilon}{4}},$$
where the inequality follows from Lemma~\ref{lem:InnProdJL} combined with ${\bf A}$ being an $\left(\epsilon / 4 \right)$-JL embedding.\\

\noindent \underline{\it Proof of $\mathbf (\boldsymbol{\dagger \dagger \dagger})$:}~~Applying Lemma~\ref{lem:ModejNormexp} with ${\bf B} = {\bf A}$ and ${\bf B} = {\bf I}$, respectively, we can see that 
\begin{equation}
\| \mathcal{Y}' \|^2 -  \| \mathcal{Y} \|^2 ~=~ \sum_{k,h=1}^r \sum^{\prod_{\ell \neq j} n_\ell}_{a = 1} \alpha_k \left( \otimes_{\ell \neq j} {\bf y}^{(\ell)}_k \right)_a \overline{\alpha_h \left( \otimes_{\ell \neq j} {\bf y}^{(\ell)}_h \right)_a} \left( \left \langle  {\bf A}  {\bf y}^{(j)}_k,  {\bf A} {\bf y}^{(j)}_h \right \rangle - \left \langle {\bf y}^{(j)}_k, {\bf y}^{(j)}_h \right \rangle \right). 
\label{equ:FiberinnerProd}
\end{equation}
Applying Lemma~\ref{lem:InnProdJL} to each inner product in \eqref{equ:FiberinnerProd} we can now see that 
$$\left \langle  {\bf A}  {\bf y}^{(j)}_k,  {\bf A} {\bf y}^{(j)}_h \right \rangle ~=~ \left \langle {\bf y}^{(j)}_k, {\bf y}^{(j)}_h \right \rangle + \epsilon_{k,h}$$ 
for some $\epsilon_{k,h} \in \mathbbm{C}$ with $\left| \epsilon_{k,h} \right| \leq \epsilon$.  As a result we have that
\begin{align*}
\left| \| \mathcal{Y} \times_j {\bf A} \|^2 -  \| \mathcal{Y} \|^2\right|~&=~\left| \sum_{k,h=1}^r \sum^{\prod_{\ell \neq j} n_\ell}_{a = 1} \alpha_k \left( \otimes_{\ell \neq j} {\bf y}^{(\ell)}_k \right)_a \overline{\alpha_h \left( \otimes_{\ell \neq j} {\bf y}^{(\ell)}_h \right)_a} \epsilon_{k,h}  \right|.\\
&=~\left| \sum_{k,h=1}^r \alpha_k \overline{\alpha_h} \epsilon_{k,h}  \sum^{\prod_{\ell \neq j} n_\ell}_{a = 1} \left( \otimes_{\ell \neq j} {\bf y}^{(\ell)}_k \right)_a \overline{ \left( \otimes_{\ell \neq j} {\bf y}^{(\ell)}_h \right)_a} \right|\\
&=~\left| \sum_{k,h=1}^r \alpha_k \overline{\alpha_h} \epsilon_{k,h}  \left \langle \bigcirc_{\ell \neq j} {\bf y}^{(\ell)}_k,\bigcirc_{\ell \neq j} {\bf y}^{(\ell)}_h \right  \rangle \right|\\ 
&\leq \left| \sum_{k=1}^r |\alpha_k|^2 \epsilon_{k,k}  \left \| \bigcirc_{\ell \neq j} {\bf y}^{(\ell)}_k \right \|^2 \right| + \left| \sum_{k \neq h} \alpha_k \overline{\alpha_h} \epsilon_{k,h}  \left \langle \bigcirc_{\ell \neq j} {\bf y}^{(\ell)}_k,\bigcirc_{\ell \neq j} {\bf y}^{(\ell)}_h \right  \rangle \right|.
\end{align*}

Noting that $\left \| \bigcirc_{\ell \neq j} {\bf y}^{(\ell)}_k \right \|^2 = 1$ by Lemma~\ref{lem:ProdProps} since $\left\| {\bf y}^{(\ell)}_k \right\|_2 = 1$ for all $\ell \in [d]$ and $k \in [r]$, we now have that
\begin{align}
\left| \| \mathcal{Y} \times_j {\bf A} \|^2 -  \| \mathcal{Y} \|^2\right|~&\leq~\epsilon \left| \sum_{k=1}^r |\alpha_k|^2 \right| + \left| \sum_{k \neq h} \alpha_k \overline{\alpha_h} \epsilon_{k,h}  \left \langle \bigcirc_{\ell \neq j} {\bf y}^{(\ell)}_k,\bigcirc_{\ell \neq j} {\bf y}^{(\ell)}_h \right  \rangle \right|\notag\\
&=~\epsilon \| {\boldsymbol \alpha} \|_2^2+ \left| \left \langle {\bf E^\top \boldsymbol \alpha}, {\boldsymbol \alpha}\right \rangle \right|,\notag
\end{align}
where ${\bf E} \in \mathbbm{C}^{r \times r}$ is zero on its diagonal, and $E_{k,h} = \epsilon_{k,h}  \left \langle \bigcirc_{\ell \neq j} {\bf y}^{(\ell)}_k,\bigcirc_{\ell \neq j} {\bf y}^{(\ell)}_h \right  \rangle$ for $k \neq h$.
As a result, $\left| \| \mathcal{Y} \times_j {\bf A} \|^2 -  \| \mathcal{Y} \|^2\right| \leq \left( \epsilon + \left\| {\bf E}^\top \right\|_{2 \rightarrow 2} \right) \| {\boldsymbol \alpha} \|^2_2$, where the operator norm $\left\| {\bf E}^\top \right\|_{2 \rightarrow 2}$ satisfies
$$\left\| {\bf E}^\top \right\|_{2 \rightarrow 2} ~\leq~ \| {\bf E} \|_{F} \leq \sqrt{\sum_{k \neq h} \left| \left \langle \bigcirc_{\ell \neq j} {\bf y}^{(\ell)}_k,\bigcirc_{\ell \neq j} {\bf y}^{(\ell)}_h \right  \rangle \right|^2 \epsilon^2 } ~=~ \epsilon ~ \sqrt{\sum_{k \neq h} \left| \left \langle \bigcirc_{\ell \neq j} {\bf y}^{(\ell)}_k,\bigcirc_{\ell \neq j} {\bf y}^{(\ell)}_h \right  \rangle \right|^2 }.$$  
Finally, Lemma~\ref{lem:ProdProps} and the definition of $\mu_\mathcal{Y}$ implies that
$$\| {\bf E} \|_{2 \rightarrow 2} ~\leq~ \epsilon \sqrt{r(r-1)} \prod_{\ell \neq j} \mu_{\mathcal{Y},\ell} ~\leq~ \epsilon r \mu_\mathcal{Y}^{d-1}.$$
Thus, we obtain the desired bound 
$$\left| \| \mathcal{Y} \times_j {\bf A} \|^2 -  \| \mathcal{Y} \|^2\right| ~\leq~ \epsilon \left( 1 + \sqrt{r(r-1)} \prod_{\ell \neq j} \mu_{\mathcal{Y},\ell} \right) \| {\boldsymbol \alpha} \|^2_2 \leq \epsilon \left( 1 + r \mu_\mathcal{Y}^{d-1} \right) \| {\boldsymbol \alpha} \|^2_2.$$
$\;$
\end{proof}

\bigskip

\begin{proof}[Proof of Lemma~\ref{lem:coefintermsofYnorm}]
Utilizing Lemma~\ref{lem:ProdProps} and the standard form of $\mathcal{Y}$ we can see that
\begin{align*}
\left| \| \mathcal{Y} \|^2 - \| {\boldsymbol \alpha} \|^2_2  \right| ~&=~ \left| \sum_{k, h=1}^r \alpha_k \overline{\alpha_h} \left \langle \bigcirc^d_{\ell = 1} {\bf y}^{(\ell)}_k,~\bigcirc^d_{\ell = 1} {\bf y}^{(\ell)}_h \right \rangle - \sum_{k = 1}^r |\alpha_k|^2 \right|\\ 
&=~ \left| \sum_{k \neq h}^r \alpha_k \overline{\alpha_h} \prod_{\ell = 1}^d \left \langle {\bf y}^{(\ell)}_k,~{\bf y}^{(\ell)}_h \right \rangle \right| ~\leq~ \mu_{\mathcal{Y}}'  \sum_{k \neq h}^r  \left| \alpha_k \overline{\alpha_h} \right|\\ 
&=~ \mu_{\mathcal{Y}}'  \left( \left( \sum_{k=1}^r |\alpha_k| \right)^2 - \sum_{k=1}^r |\alpha_k|^2 \right) ~\leq~ \mu_{\mathcal{Y}}' \left( \left( \sqrt{r} \| {\boldsymbol \alpha} \|_2 \right)^2 - \| {\boldsymbol \alpha} \|_2^2 \right)
\end{align*}

where the last inequality follows from Cauchy-Schwarz.  As a result we have that
$$\left| \| \mathcal{Y} \|^2 - \| {\boldsymbol \alpha} \|^2_2 \right| ~\leq~ \mu_{\mathcal{Y}}' (r - 1) \| {\boldsymbol \alpha} \|_2^2 $$
which in turn implies that 
$$ \| \mathcal{Y} \|^2 \geq \left(1 - (r - 1) \mu_{\mathcal{Y}}' \right)  \| {\boldsymbol \alpha} \|^2_2.$$
$\;$
\end{proof}

\bigskip

The following simple fact will be used repeatedly in the proof of Proposition~\ref{thm:ObliviousSubspaceEmbedded}.

\begin{remark}
Let $c, d \in \mathbbm{R}^+$.  Then, $\displaystyle \mathbbm{e}^c ~\geq~\left(1+\frac{c}{d} \right)^d$.
\label{lem:IknowCalculus}
\end{remark}

\bigskip

\begin{proof}[Proof of Proposition~\ref{thm:ObliviousSubspaceEmbedded}]
Let $\mathcal{Y}^{(0)} := \mathcal{Y}$, and for each $j \in [d]$ define the tensor 
$$\mathcal{Y}^{(j)}:= \mathcal{Y} \times_1 {\bf A}_1 \dots \times_j {\bf A}_j ~=~ \sum_{k=1}^r \alpha_{j,k} \bigcirc^d_{\ell = 1} {\bf y}^{(\ell)}_{j,k}$$
expressed in standard form via $j$ applications of Lemma~\ref{lem:NewStandForm}.  Note that parts ($\dagger$) and ($\dagger \dagger$) of Proposition~\ref{thm:ForInduction} imply that
\begin{enumerate}

\item[$(i)$] $\left| \alpha_{j,k} - \alpha_{j-1,k} \right| ~\leq~ \epsilon |\alpha_{j-1,k}| / 4d$ so that $|  \alpha_{j,k} | \leq (1 + \epsilon / 4d) |  \alpha_{j-1,k}|$ holds for all $k \in [r]$, and 

\item[$(ii)$] $\mu_{\mathcal{Y}^{(j)},j} ~\leq~ (\mu_{\mathcal{Y}^{(j-1)},j} + \epsilon/d)/(1 - \epsilon/4d)$, and $\mu_{\mathcal{Y}^{(j)},\ell} ~=~ \mu_{\mathcal{Y}^{(j-1)},\ell}$ for all $\ell \in [d] \setminus \{ j \}$,
\end{enumerate}
both hold for all and $j \in [d]$.
Using these facts it is not too difficult to inductively establish that both 
\begin{equation}
|  \alpha_{j,k} | \leq (1 + \epsilon / 4d)^j |  \alpha_{k}|,
\label{equ:InductCoefBound}
\end{equation}
and 
\begin{equation}
\displaystyle \prod_{\ell \neq j} \mu_{\mathcal{Y}^{(j-1)},\ell} ~\leq~ \left( \prod_{\ell < j}  \frac{\mu_{\mathcal{Y},\ell} + \epsilon/d}{1 - \epsilon/4d} \right) \prod_{\ell > j} \mu_{\mathcal{Y},\ell} ~\leq~ \left( \frac{\mu_{\mathcal{Y}} + \epsilon/d}{1 - \epsilon/4d} \right)^{j-1} \mu_{\mathcal{Y}}^{d-j},
\label{equ:InductCoherenceBound}
\end{equation}
 also hold for all $k \in [r]$ and $j \in [d]$.  Note that in \eqref{equ:InductCoherenceBound} we will let $\mu_{\mathcal{Y}}^{0} = 1$ even if $\mu_{\mathcal{Y}} = 0$ since this still yields the correct bound in the $j = d$ and $\mu_{\mathcal{Y}} = 0$ case.

Preceding with the desired error bound we can now see that
\begin{align*}
\left| \left\| \mathcal{Y} \right\|^2 - \left\| \mathcal{Y} \times_1 {\bf A}_1 \dots \times_d {\bf A}_d \right\|^2 \right| ~&=~ \left| \sum^{d-1}_{j = 0} \left\| \mathcal{Y}^{(j)} \right\|^2 - \left\| \mathcal{Y}^{(j+1)} \right\|^2 \right|\\
&\leq~\frac{\epsilon}{d} \sum^{d-1}_{j = 0} \left( 1 + \sqrt{r(r-1)} \prod_{\ell \neq j+1} \mu_{\mathcal{Y}^{(j)},\ell} \right) \| {\boldsymbol \alpha_j} \|^2_2\\
&\leq~\frac{\epsilon}{d} \sum^{d-1}_{j = 0} \left( 1 + \sqrt{r(r-1)} \left( \frac{\mu_{\mathcal{Y}} + \epsilon/d}{1 - \epsilon/4d} \right)^{j} \mu_{\mathcal{Y}}^{d-1-j} \right) (1 + \epsilon / 4d)^{2j} \| {\boldsymbol \alpha} \|^2_2\\
&\leq~\frac{\epsilon}{d} \sum^{d-1}_{j = 0} \left( 1 + \sqrt{r(r-1)} \left( \frac{\mu_{\mathcal{Y}} + \epsilon/d}{1 - \epsilon/4d} \right)^{j} \mu_{\mathcal{Y}}^{d-1-j} \right) (1 + 9\epsilon / 16d)^{j} \| {\boldsymbol \alpha} \|^2_2
\end{align*}
where we have used part ($\dagger \dagger \dagger$) of Proposition~\ref{thm:ForInduction}, \eqref{equ:InductCoefBound}, and \eqref{equ:InductCoherenceBound}.  Considering each term in the upper bound above separately, we have that
$$\left| \left\| \mathcal{Y} \right\|^2 - \left\| \mathcal{Y} \times_1 {\bf A}_1 \dots \times_d {\bf A}_d \right\|^2 \right| \leq \frac{\epsilon}{d}  \| {\boldsymbol \alpha} \|^2_2 \left( T_1 + \sqrt{r(r-1)}T_2 \right)$$
where
\begin{align*}
T_1 := \sum^{d-1}_{j = 0} (1 + 9\epsilon / 16d)^{j} ~=~ \frac{(1 + 9\epsilon / 16d)^d - 1}{9\epsilon / 16d} ~\leq~ \mathbbm{e} d
\end{align*}
using Remark~\ref{lem:IknowCalculus} and that $9\epsilon / 16 < 1$, and where
\begin{align*}
T_2 := \sum^{d-1}_{j = 0} \left( \frac{\mu_{\mathcal{Y}} + \epsilon/d}{1 - \epsilon/4d} \right)^{j} \mu_{\mathcal{Y}}^{d-1-j} (1 + 9\epsilon / 16d)^{j} ~\leq~ \sum^{d-1}_{j = 0} \left( \mu_{\mathcal{Y}} + \epsilon/d \right)^{j} \mu_{\mathcal{Y}}^{d-1-j} (1 + \epsilon / d)^j
\end{align*}
for $\epsilon \leq 3/4$.

Continuing to bound the second term we will consider three cases.  First, if $\mu_{\mathcal{Y}} = 0$ then 
$$T_2 ~\leq~ \left( \epsilon/d \right)^{d-1} (1 + \epsilon / d)^{d-1} ~\leq~ \mathbbm{e} \left( \epsilon/d \right)^{d-1},$$
using Remark~\ref{lem:IknowCalculus} and that $\epsilon< 1$.
Second, if $0 < \mu_{\mathcal{Y}} \leq \epsilon$ then
\begin{align*}
T_2 ~&\leq~ \sum^{d-1}_{j = 0} \left( \epsilon + \epsilon/d \right)^{j} \epsilon^{d-1-j} (1 + \epsilon / d)^j ~=~ \epsilon^{d-1} \sum^{d-1}_{j = 0} \left( 1 + 1/d \right)^{j} (1 + \epsilon / d)^j\\
~&\leq~ \epsilon^{d-1} d \left( 1 + 1/d \right)^{d} (1 + \epsilon / d)^d  ~\leq~ d \mathbbm{e}^{2} \epsilon^{d-1},
\end{align*}
using Remark~\ref{lem:IknowCalculus} and that $\epsilon< 1$ once more. If, however, $\mu_{\mathcal{Y}} > \epsilon$ then we can see that
\begin{align*}
T_2 ~&\leq~ \mu_{\mathcal{Y}}^{d-1}\sum^{d-1}_{j = 0} \left( 1 + \epsilon / \mu_{\mathcal{Y}} d \right)^{j} (1 + \epsilon / d)^j ~\leq~ \mu_{\mathcal{Y}}^{d-1} \sum^{d-1}_{j = 0} \left( 1 + 1/d \right)^{j} (1 + \epsilon / d)^j\\ 
&\leq~ \mu_{\mathcal{Y}}^{d-1} \cdot d  \left( 1 + 1/d \right)^d (1 + \epsilon / d)^d ~\leq~ \mu_{\mathcal{Y}}^{d-1} ~d~\mathbbm{e}^{1 + \epsilon} ~\leq~ d \mathbbm{e}^{2}    \mu_{\mathcal{Y}}^{d-1},
\end{align*}
where we have again utilized Remark~\ref{lem:IknowCalculus}.  The desired result now follows.
\end{proof}

\bigskip

\begin{proof}[Proof of Lemma~\ref{lem:IndivTensorJL}] Fix $t \in [p]$ and let $\mathcal{X}^{(0)} := \mathcal{Z}^{(t)}$, $\mathcal{X}^{(j)} := \mathcal{Z}^{(j,t)}$ for all $j \in [d-1]$, and $\mathcal{X}^{(d)} := \mathcal{Z}^{(d-1,t)} \times_{d} {\bf A}_{d} = \mathcal{Z}^{(t)} \times_1 {\bf A}_1 \dots \times_d {\bf A}_d$.
Choose any $j \in [d]$, and let ${\bf x}_{j,h} \in \mathbbm{C}^{n_j}$ denote the $h^{\rm th}$ column of the mode-$j$ unfolding of $\mathcal{X}^{(j-1)}$, denoted by ${\bf X}^{(j-1)}_{(j)}$.  It is easy to see that each ${\bf x}_{j,h}$ is a mode-$j$ fiber of $\mathcal{X}^{(j-1)} = \mathcal{Z}^{(j-1,t)}$ for each $1 \leq h \leq N'_j := \left(\prod^{j-1}_{\ell = 1}m_\ell \right) \left( \prod_{\ell = j+1} n_\ell \right)$.
Thus, we can see that
\begin{align*}
\left| \left\| \mathcal{X}^{(j-1)} \right\|^2 - \left\| \mathcal{X}^{(j)} \right\|^2 \right| ~&=~ \left| \left\| \mathcal{X}^{(j-1)} \right\|^2 - \left\| \mathcal{X}^{(j-1)} \times_j {\bf A}_j \right\|^2 \right| ~=~  \left| \left\| {\bf X}^{(j-1)}_{(j)} \right\|^2_{\rm F} - \left\| {\bf A}_j {\bf X}^{(j-1)}_{(j)} \right\|^2_{\rm F} \right|\\ 
&=~ \left| \sum^{N'_j}_{h=1} \| {\bf x}_{j,h} \|^2_2 -  \left\| {\bf A}_j {\bf x}_{j,h} \right\|^2_2 \right| ~\leq~ \sum^{N'_j}_{h=1} \left|  \| {\bf x}_{j,h} \|^2_2 -  \| {\bf A}_j {\bf x}_{j,h} \|^2_2 \right|\\
&\leq~ \frac{\epsilon}{\mathbbm{e}d } \sum^{N'_j}_{h=1} \| {\bf x}_{j,h} \|^2_2 = \frac{\epsilon}{\mathbbm{e}d }  \left\| {\bf X}^{(j-1)}_{(j)} \right\|^2_{\rm F} = \frac{\epsilon}{\mathbbm{e}d }   \left\| \mathcal{X}^{(j-1)} \right\|^2.
\end{align*}
A short induction argument now reveals that $\left\| \mathcal{X}^{(j)} \right\|^2~\leq~\left( 1+ \frac{\epsilon}{\mathbbm{e}d }  \right)^j\left\| \mathcal{X}^{(0)} \right\|^2$ holds for all $j \in [d]$.  As a result we can now see that
\begin{align*}
\left| \left\| \mathcal{X}^{(0)} \right\|^2 - \left\| \mathcal{X}^{(d)} \right\|^2 \right| ~&=~ \left| \sum^d_{j = 1} \left\| \mathcal{X}^{(j-1)} \right\|^2 - \left\| \mathcal{X}^{(j)} \right\|^2 \right|~\leq~ \sum^d_{j = 1} \left| \left\| \mathcal{X}^{(j-1)} \right\|^2 - \left\| \mathcal{X}^{(j)} \right\|^2 \right| ~\leq~\frac{\epsilon}{\mathbbm{e}d }  \sum^d_{j = 1} \left\| \mathcal{X}^{(j-1)} \right\|^2\\
&\leq~\frac{\epsilon}{\mathbbm{e}d } \sum^d_{j = 1}  \left( 1+ \frac{\epsilon}{\mathbbm{e}d }  \right)^{j-1}\left\| \mathcal{X}^{(0)} \right\|^2 ~\leq~\frac{\epsilon}{\mathbbm{e}} \left( 1+ \frac{\epsilon}{\mathbbm{e}d }  \right)^d\left\| \mathcal{X}^{(0)} \right\|^2.
\end{align*}
holds.  The desired result now follows from Remark~\ref{lem:IknowCalculus}.
\end{proof}

\end{document}